%% file: ac.defects-refissues.tex
\documentclass[12pt, reqno, a4paper,oneside]{amsart}
\usepackage{graphicx, epsfig, psfrag}
\usepackage{amsfonts,amsmath,amssymb,amsbsy,amsthm}
\usepackage{stmaryrd}
\usepackage{bm}
\usepackage{color}
\usepackage{float}
\usepackage{mathrsfs}
\usepackage{longtable}
\usepackage{enumerate}
\usepackage[normalem]{ulem}
\usepackage{placeins}

\def\defaultfontcolor{}

\usepackage[paper=a4paper,top=3cm,bottom=3cm,left=2.5cm,right=2.5cm]{geometry}

\usepackage{environments}
\numberwithin{theorem}{section}
\numberwithin{equation}{section}

\renewcommand{\paragraph}[1]{\subsubsection{#1}}

\input{notation}
\newcommand{\Rcore}{{R_{\rm def}}}
\newcommand{\burg}{{\sf b}}
\newcommand\up{u_0}
\newcommand{\halfspace}{\Omega_\Gamma}
\newcommand{\rdisl}{\hat{r}}
\newcommand{\Rdisl}{\rdisl}
\newcommand{\Adm}{\mathscr{A}}

\newcommand{\Del}{\tilde{D}}

\newcommand{\Rghom}{{\tt \char`\\Rghom}}
\newcommand{\pPt}{{\tt \char`\\pPt}}
\newcommand{\pDis}{{\tt \char`\\pDis}}
\newcommand{\Nhdu}{{\tt \char`\\Nhdu}}
\newcommand{\asPotentialHom}{{\tt \char`\\asPotentialHom}}
\newcommand{\asPotentialSym}{{\tt \char`\\asPotentialSym}}
\newcommand{\asCore}{{\tt \char`\\asCore}}
\newcommand{\asPotential}{{\tt \char`\\asPotential}}
\newcommand{\rhocut}{{\tt \char`\\rhocut}}
\newcommand{\halfplane}{{\tt \char`\\halfplane}}

\newcommand{\balpha}{\alpha}

\definecolor{alertcol}{rgb}{0.7, 0, 0}
\definecolor{vecol}{rgb}{0.7, 0, 0.7}
\definecolor{ascol}{rgb}{0, 0, 0.7}
\definecolor{todocol}{rgb}{0.0, 0.4, 0.0}
\definecolor{cocol}{rgb}{0.6, 0.0, 0.8}

\newcommand{\alert}[1]{#1}
\newcommand{\marg}[1]{}

\begin{document}

\defaultfontcolor

\title[Analysis of Crystal Defect Atomistic
Simulations]{Analysis of Boundary Conditions \\ for Crystal Defect Atomistic
  Simulations}

\author{V. Ehrlacher}
\address{CERMICS - ENPC \\ 6 et 8 avenue Blaise Pascal  \\ Cit\'{e} Descartes
  - Champs sur Marne  \\ 77455 Marne la Vallée Cedex 2 \\ FRANCE}
\email{ehrlachv@cermics.enpc.fr}

\author{C. Ortner}
\address{C. Ortner\\ Mathematics Institute \\ Zeeman Building \\ University of
  Warwick \\ Coventry CV4 7AL \\ UK}
\email{christoph.ortner@warwick.ac.uk}

\author{A. V. Shapeev} \address{206 Church St. SE \\ School of Mathematics \\
  University of Minnesota \\ Minneapolis \\ MN 55455 \\ USA}
\email{alexander@shapeev.com}

\date{\today}

\thanks{ %
  Some of the ideas leading to this work were developed during the
  IPAM semester programme ``Materials Defects: Mathematics,
  Computation, and Engineering''. CO's work was supported by EPSRC
  grants EP/H003096, EP/J021377/1, ERC Starting Grant 335120 and by
  the Leverhulme Trust through a Philip Leverhulme Prize. AS was
  supported by AFOSR Award FA9550-12-1-0187.}



\begin{abstract}
  Numerical simulations of crystal defects are necessarily restricted to finite
  computational domains, supplying artificial boundary conditions that emulate
  the effect of embedding the defect in an effectively infinite crystalline
  environment. This work develops a rigorous framework within which the accuracy
  of different types of boundary conditions can be precisely assessed.

  We formulate the equilibration of crystal defects as variational problems in a
  discrete energy space and establish qualitatively sharp regularity estimates
  for minimisers. Using this foundation we then present rigorous error estimates
  for (i) a truncation method (Dirichlet boundary conditions), (ii) periodic
  boundary conditions, (iii) boundary conditions from linear elasticity, and
  (iv) boundary conditions from nonlinear elasticity. Numerical results confirm
  the sharpness of the analysis.
  %
\end{abstract}

\maketitle


\section{Introduction}
\label{sec:intro}
Crystalline solids can contain many types of defects. Two of the most
important classes are dislocations, which give rise to plastic flow, and point
defects, which can affect both elastic and plastic material behaviour as well
as brittleness.

Determining defect geometries and defect energies are a key problem of
computational materials science \cite[Ch. 6]{HandbookMaterialsModelling}.
Defects generally distort the host lattice, thus generating long-ranging elastic
fields. Since practical schemes necessarily work in small computational domains
(e.g., ``supercells'') they cannot explicitly resolve these fields but must
employ {\em artificial boundary conditions} (periodic boundary conditions appear
to be the most common). To assess the accuracy and in particular the cell size
effects of such simulations, numerous formal results, numerical explorations, or
results for linearised problems can be found in the literature; see
e.g. \cite{Balluffi,Hine2009,MakovPayne1995,CaiBulChaLiYip2003} and references
therein for a small representative sample.

The novelty of the present work is that we rigorously establish explicit
convergence rates in terms of computational cell size, taking into account the
long-ranged elastic fields. Our framework encompasses both point defects and
straight dislocation lines. Related results in a PDE context have recently been
developed in \cite{BlancLeBrisLions:defects}.


The second motivation for our work is the analysis of multiscale
methods. Several multiscale methods have been proposed to accelerate crystal
defect computations (for example atomistic/continuum coupling
\cite{Ortiz:1995a}, \cite{LuOr:acta} or QM/MM
\cite{BernsteinKermodeCsanyi:2009}), and our framework provides a natural set of
benchmark problems and a comprehensive analytical substructure for these methods
to assess their relative accuracy and efficiency. In particular, it provides a
machinery for the optimisation of the (non-trivial) set of approximation
parameters in multiscale schemes.

The mathematical analysis of crystalline defects has traditionally focused on
dislocations \cite{HirthLothe, ArizaOrtiz:2005, ADLGP13, HudsonOrtner:disloc}
and on electronic structure models \cite{CattoLeBrisLions,
  CancesDeleurenceLewin:2008}; however, see \cite{CancesLeBris:preprint} for a
comprehensive recent review focused on point defects. The results in the present
work, in particular the decay estimates on elastic fields, also have a bearing
on this literature since they can be used to establish finer information about
equilibrium configurations; see e.g., \cite{2014-dislift}.

\subsection*{Acknowledgement}
We thank Brian Van Koten who pointed out a substantial flaw in our construction
of the edge dislocation predictor in an earlier version 
of this work, and made valuable comments that helped us resolve it.

\subsection{Outline}
\quad Our approach consists in placing the defect in an infinite crystalline
environment, for simplicity say $\Z^d$, where $d \in \{2, 3\}$ is the space
dimension, applying a far-field boundary condition which encodes the macroscopic
state of the system within which the defect is embedded. Let $w : \Z^d \to \R^m$
be the unknown displacement of the crystal, then we decompose $w = \up + u$,
where $\up$ is a {\em predictor} that specifies the boundary condition through
the requirement that the {\em corrector} $u$ belongs to a discrete energy space
$\UsH$ (a canonical discrete variant of $\dot{H}^1(\R^d)$). We then formulate
the condition for $w$ to be an equilibrium configuration as a (local) energy
minimisation problem,
\begin{equation}
  \label{eq:intro:min}
  \ua \in \arg\min \b\{ \E(u) \bsep u \in \UsH \b\},
\end{equation}
where $\E(u)$ is the energy difference between the total displacement
$w = \up +u$ and the {\em predictor} $\up$.

The choice of $\up$ is not arbitrary. It is crucial that $\up$ is an
``approximate equilibrium'' in the far-field, which will be expressed through
the requirement that $\del\E(0) \in (\UsH)^*$. It is clear that, if this
condition fails, then $\inf \{\E(u) \sep u \in \UsH \} = -\infty$.  For this
reason, we think of $\up$ as a {\em predictor} and $u$ as a {\em corrector}.
For the case of dislocations, the choice of $\up$ is non-trivial, as the
``naive'' linear elasticity predictor does not take lattice symmetries correctly
into account.  In \S~\ref{sec:dis:atm} we present a new construction that
remedies this issue.

We shall not be concerned with existence of solutions to \eqref{eq:intro:min};
even for the simplest classes of defects this is a difficult problem. Uniqueness
can never be expected for realistic interatomic potentials.

However, assuming that a solution to \eqref{eq:intro:min} does exist, we may
then analyze its ``regularity''. More precisely, under a natural stability
assumption we estimate the rate in terms of distance to the defect core at which
$\ua$ (and its discrete gradients of arbitrary order) approach zero. For
  example, we will prove that
\begin{displaymath}
  |D\ua(\ell)| \leq C |\ell|^{-d} (\log |\ell|)^r,
\end{displaymath}
where $Du(\ell)$ is a finite difference gradient centered at $\ell \in \Z^d$,
$r = 0$ for point defects and $r = 1$ for straight dislocation lines.

These regularity estimates then allow us to establish various approximation
results. For example, we can estimate the error committed by projecting an
infinite lattice displacement field $u$ to a finite domain by truncation. This
motivates the formulation of a Galerkin-type approximation scheme for
\eqref{eq:intro:min} (see \S~\ref{sec:pt:clamped} and \S~\ref{sec:dis:clamped})
\begin{align}
  \label{eq:intro:galerkin}
  \ua_N &\in \arg\min \b\{ \E(u) \bsep u \in \UsH_N \b\}, \\
  & \notag
  \text{where} \quad \UsH_N := \b\{ u \in \UsH \bsep u(\ell) = 0 \text{
    for } |\ell| \geq N^{1/d} \b\}.
\end{align}
This is a finite dimensional optimisation problem with
${\rm dim}(\UsH_N) \approx N$, and our framework yields a straightforward proof
of the following error estimate: {\it suppose $\ua$ is a strongly stable
  (cf. \eqref{eq:pt:strongstabeq}) solution to \eqref{eq:intro:min} then, for
  $N$ sufficiently large, there exists a solution $\ua_N$ to
  \eqref{eq:intro:galerkin} such that
  \begin{displaymath}
    \| \ua - \ua_N \|_{\UsH} \leq C N^{-1/2} (\log N)^r,
  \end{displaymath}
  where $r = 0$ for point defects, $r = 1/2$ for straight dislocation lines, and
  $\| \ua - \ua_N \|_{\UsH}$ is a natural discrete energy norm.} Note that $N$
is directly proportional to the (idealised) computational cost of solving
\eqref{eq:intro:galerkin}. We stress that we stated only that ``there exists a
$\ua_N$''; indeed, both \eqref{eq:intro:min} and \eqref{eq:intro:galerkin}
typically have many solutions. Roughly speaking, this means that ``near every
stable solution to \eqref{eq:intro:min} there exists a stable solution to
\eqref{eq:intro:galerkin}''. It is interesting to note that the rate $N^{-1/2}$
is generic; that is, it is independent of any details of the particular
defect. We prove a similar error estimate for periodic boundary conditions in
\S~\ref{sec:pt:per}.

In \S\S~\ref{sec:pt:lin}, \ref{sec:pt:ac}, \ref{sec:dis:lin}, \ref{sec:dis:ac}
we then consider two types of {\em concurrent (or, self-consistent) boundary
  conditions} that use elasticity models to improve the far field corrector. In
these approximate models, we solve a far-field problem concurrently with the
atomistic core problem in order to improve the boundary conditions placed on the
atomistic core.  First, in \S~\ref{sec:pt:lin}, \ref{sec:dis:lin} we use
linearised lattice elasticity to construct an improved far-field predictor.
Second, in \S~\ref{sec:pt:ac}, \ref{sec:dis:ac} we analyze the effect of using
nonlinear continuum elasticity to improve the far-field boundary condition. This
effectively leads us to formulate an atomistic-to-continuum coupling scheme
within our framework. For both methods we show that, in the point defect case
this yields substantial improvements over the simple truncation method, but
surprisingly, for dislocations the methods are qualitatively comparable to the
simple truncation scheme. We note, however that based on the benchmarks of the
present paper, improved a/c schemes with superior convergence rates have
recently been developed in \cite{LiOrtnerShapeevEtAl-blended,OrtnerZhang2014}.


Our numerical experiments in \S~\ref{sec:pt:num}, \ref{sec:dis:num}
mostly confirm that our analytical predictions are sharp, however, we
also show some cases where they do not capture the full complexity of
the convergence behaviour.


\subsubsection*{Restrictions}
Our analysis in the present paper is restricted to static equilibria
under classical interatomic interaction with finite interaction
range. We see no obstacle to include Lennard-Jones type interactions,
but this would require finer estimates and a more complex notation.
However, we explicitly exclude Coulomb interactions or any electronic
structure model and hence also charged defects (see, e.g.,
\cite{Hine2009,MakovPayne1995,CattoLeBrisLions,
  CancesDeleurenceLewin:2008, CancesLeBris:preprint}). Due to the
computational cost involved in these latter models, obtaining
analogous convergence results for these, would be of considerable
interest.

As reference atomistic structure we admit only single-species Bravais
lattices. Again, we see no conceptual obstacles to generalising to
multi-lattices, however, some of the technical details may require additional
work.

As already mentioned we only focus on ``compactly supported'' defects, but
exclude curved line defects, grain or phase boundaries, surfaces or
cracks. Moreover, we exclude the case of multiple or indeed infinitely many
defects.  We hope, however, that our new analytical results on single defects
will aid future studies of this setting; e.g., see \cite{2014-dislift} for an
analysis of multiple screw dislocations which is based on the present work.



\subsubsection*{Notation}
Notation is introduced throughout the article. Key symbols that are used
  across sections are listed in Appendix \ref{sec:symbols}. We only briefly
remark on some generic points. The symbol $\< \cdot, \cdot\>$ denotes an
abstract duality pairing between a Banach space and its dual. The symbol
$|\cdot|$ normally denotes the Euclidean or Frobenius norm, while $\|\cdot\|$
denotes an operator norm.

The constant $C$ is a generic positive constant that may change from one line of
an estimate to the next. When estimating rates of decay or convergence rates
then $C$ will always remain independent of approximation parameters (such as
$N$, which relates to domain size), of lattice position (such as $\ell$) or of
test functions. However, it may depend on the interatomic potential or some
fixed displacement or deformation field (e.g., on the boundary condition and the
solution). The dependencies of $C$ will normally be clear from the context, or
stated explicitly. To improve readability, we will sometimes replace
  $\leq C$ with $\lesssim$.

For a differentiable function $f$, $\nabla f$ denotes the Jacobi matrix and
$\nabla_r f = \nabla f \cdot r$ the directional derivative. The first and second
variations of a functional $E \in C^2(X)$ are denoted, respectively, by
$\< \del E(u), v \>$ and $\< \ddel E(u) w, v \>$ for $u, v, w \in X$. We will avoid use of higher variations in this explicit
way.

If $\L \subset \R^d$ is a discrete set and $u : \L \to \R^m$, $\ell \in \L$ and
$\ell+\rho \in \L$, then we define the finite difference
$D_\rho u(\ell) := u(\ell+\rho) - u(\ell)$. If $\Rg \subset \L - \ell$, then we
define $D_\Rg u(\ell) := (D_\rho u(\ell))_{\rho \in \Rg}$. We will normally
specify a specific stencil $\Rg_\ell$ associated with a site $\ell$ and define
$Du(\ell) := D_{\Rg_\ell} u(\ell)$.  For ${\bm \rho}\in(\Rg_\ell)^j$,
$D_{\bm \rho} u := D_{\rho_1}\ldots D_{\rho_j} u$ denotes a $j$-th order
derivative, and $D^j u$ defined recursively by $D^j u := D D^{j-1} u$ denotes
the $j$-th order collection of derivatives.

\section{Point Defects}
\label{sec:sum-pt}
\subsection{Atomistic Model}
\label{sec:pt:atm}
We formulate a model for a point defect embedded in a homogeneous lattice. To
simplify the presentation, we admit only a finite interaction radius (in
reference coordinates) and a smooth interatomic potential. Both are easily
  lifted, but introduce non-essential technical complications.

Let $d \in \{2, 3\}$ and $\mA \in \R^{d \times d}$ nonsingular, defining a
Bravais lattice $\mA \Z^d$. \label{sym:lattice} The {\em reference
  configuration} for the defect is a set $\L \subset \R^d$ such that, for
some \label{sym:Rcore} $\Rcore > 0$,
$\L \setminus B_{\Rcore} = \mA\Z^d \setminus B_{\Rcore}$ and
$\L \cap B_{\Rcore}$ is finite. For analytical purposes it is convenient to
assume the existence of a background mesh, that is, a regular
partition \label{sym:T-L} $\T_\L$ of $\R^d$ into triangles if $d = 2$ and
tetrahedra if $d = 3$ whose nodes are the reference sites $\L$, and which is
homogeneous in $\R^d \setminus B_{\Rcore}$. (If $T \in \T_\L$ and
$\rho \in \mA\Z^d$ with $T, \rho+T \subset \R^d \setminus B_{\Rcore}$, then
$\rho+T \in \T_\L$ as well.)
We refer to Figure \ref{fig:pt_defects_comp} for
two-dimensional examples of such triangulations.


\begin{figure}
  \begin{center}
    \includegraphics[height=5cm]{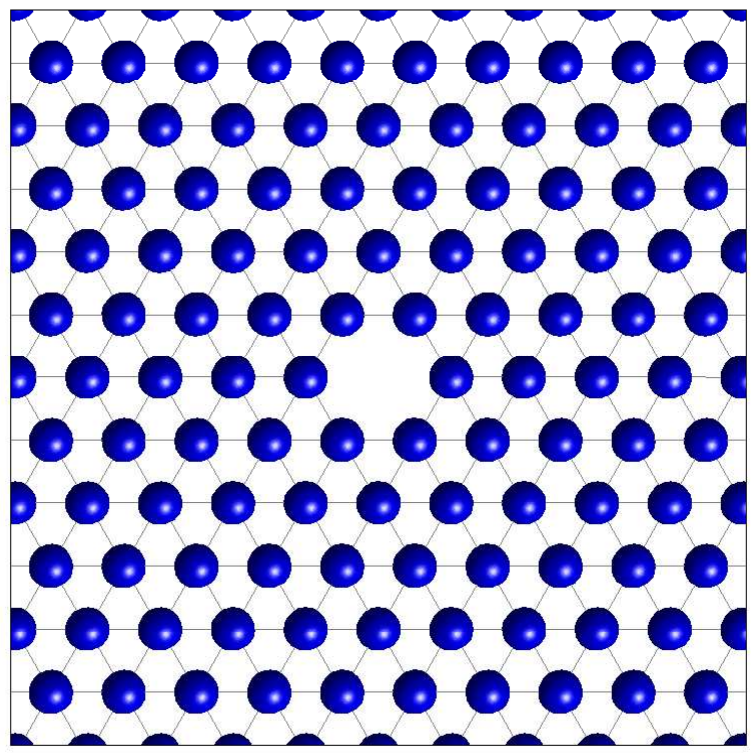}
    \qquad \includegraphics[height=5cm]{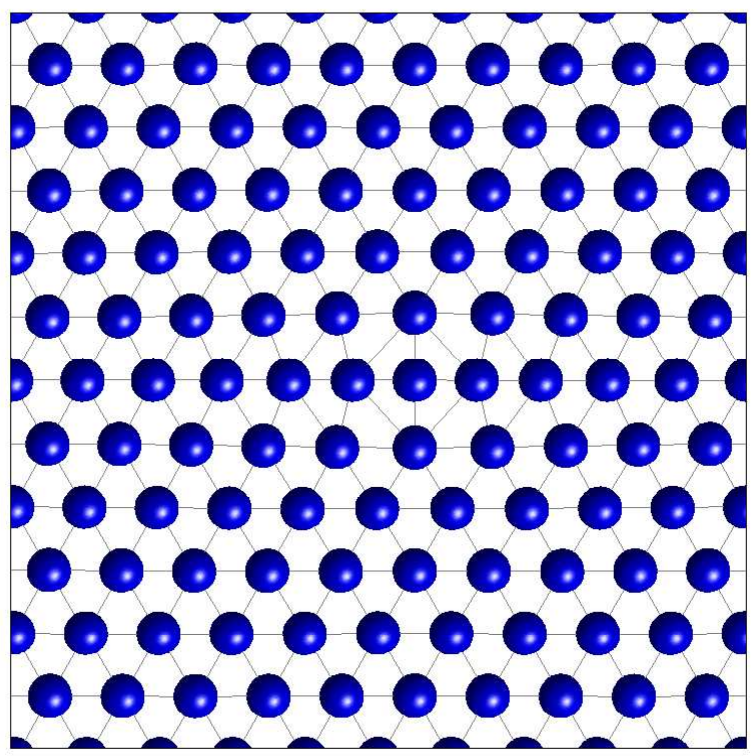} \\[1mm]
    \footnotesize (a) \hspace{5cm} (b)
  \end{center}
  \caption{\label{fig:pt_defects_comp} Illustration of (a) vacancy and
    (b) interstitial defects; relaxed configurations computed with
    ATM-DIR (cf. \S~\ref{sec:pt:num}). The grey lines indicate the
    interaction bonds, $\Rg_\ell$, between atoms, for a
    nearest-neighbour mode, as well as the auxiliary triangulation
    $\T_\L$.}
\end{figure}

For each $u : \L \to \R^m$ we denote its continuous and piecewise
affine interpolant with respect to $\T_\L$ by $Iu$. Identifying $u=Iu$ we can define the (piecewise constant) gradient $\D u
:= \D Iu : \R^d \to \R^{m \times d}$ and the spaces of compact and
finite-energy displacements, respectively, by 
\begin{equation}
  \label{eq:pt:defnUszUsH}
  \begin{split}
    \Usz &:= \b\{ u : \L \to \R^d \bsep {\rm supp}(\D u) \text{ is
    compact} \b\} \quad \text{and} \\
  \Us^{1,2} &:= \b\{ u : \L \to \R^d \bsep \D u \in L^2 \b\}.
\end{split}
\end{equation}
It is easy to see \cite{OrtnerTheil2012, OrShSu:2012} that $\Usz$ is
dense in $\Us^{1,2}$ in the sense that, if $u \in \Us^{1,2}$, then
there exist $u_j \in \Usz$ such that $\D u_j \to \D u$ strongly in
$L^2$.

Each atom $\ell \in \L$ may interact with a neighbourhood defined by the set of
lattice vectors \label{sym:Rg-rcut}
$\Rg_\ell \subset (B_{\rcut} \cap (\L-\ell)) \setminus \{0\}$, where
$\rcut > 0$, and we let $Du(\ell) := D_{\Rg_\ell} u(\ell)$. We assume without
loss of generality that
\begin{equation}\label{eq:T_Rg_conform}
	\text{if $(\ell, \ell+\rho)$ is an
		edge of $\T_\L$, then $\rho \in \Rg_\ell$.
	}
\end{equation}
This assumption implies, in particular, that $\| \D u \|_{L^2} \approx \|
  \,|Du|_p\, \|_{\ell^2}$ for any $p \in [1, \infty]$, where $|Du|_p(\ell) :=
  (\sum_{\rho \in \Rg_\ell} |D_\rho u(\ell)|^p)^{1/p}$.

\label{sym:V} For each $\ell \in \L$ let $V_\ell \in C^k( (\R^d)^{\Rg_\ell})$,
$k \geq 2$, be a smooth site energy potential satisfying $V_\ell(\bfO) = 0$ for
all $\ell \in \L$. (If $V(\bfO) \neq 0$, then it can be replaced with
$V_\ell(Du) \equiv V_\ell(Du) - V_\ell(\bfO)$; that it, $V$ should be understood
as a site energy difference.)  Then the energy functional for compact
displacements is given by \label{sym:E}
\begin{displaymath}
  \E(u) := \sum_{\ell \in \L} V_\ell(Du(\ell)) 
  \qquad \text{for $u \in \Usz$}.
\end{displaymath}
We assume throughout, that $V_\ell$ is homogeneous outside the defect
core, that is, $\Rg_\ell \equiv \Rg$ and $V_\ell \equiv V$ for all
$|\ell| \geq \Rcore$, and it is point symmetric,
\begin{equation}
  \label{eq:pt:ptsymm}
  -\Rg = \Rg \qquad \text{and} \qquad V\b( (-g_{-\rho})_{\rho \in \Rg}
  \b) = V( \bfg) \quad \forall \bfg \in (\R^m)^{\Rg}.
\end{equation}
Without loss of generality, we also assume that
\begin{equation}\label{eq:pt:nn}
\mA e_n\in\Rg,
\qquad n=1\ldots,d.
\end{equation}
Under these assumptions we can extend the definition of $\E$ to
$\Us^{1,2}$.

\begin{lemma}
  \label{th:extension_lemma}
  $\E : (\Usz, \|\D \cdot \|_{L^2}) \to \R$ is continuous. In
  particular, there exists a unique continuous extension of $\E$ to
  $\Us^{1,2}$, which we still denote by $\E$. The extended functional
  $\E : \Us^{1,2} \to \R$ is $k$ times continuously Fr\'echet
  differentiable.
\end{lemma}
\begin{proof}[Idea of the proof]
  For $u \in \Usz$, we may write
  \begin{displaymath}
    \E(u) = \sum_{\ell \in \L} \B( V_\ell(Du(\ell)) -
    V_\ell(0) - \b\< \del V_\ell(0), Du(\ell) \b\> \B) +
    \< \del \E(0), u \>.
  \end{displaymath}
  One now uses the fact that the summand in the first group scales
  quadratically, while $\del \E(0)$ is a bounded linear
  functional. The details are presented in
  \S~\ref{sec:prf_extension_pt}.
\end{proof}

\medskip

In view of Lemma \ref{th:extension_lemma} the atomistic variational problem is
``well-formulated'': we seek
\begin{equation}
  \label{eq:pt:atm}
  \ua \in \arg\min \b\{ \E(u) \bsep u \in \Us^{1,2} \b\},
\end{equation}
where $\arg\min$ denotes the set of local minimizers.

We are not concerned with the existence of solutions to
\eqref{eq:pt:atm}, but take the view that this is a property of the
lattice and the interatomic potential. We shall assume the existence
of a {\em strongly stable equilibrium} $\ua \in \Us^{1,2}$, by which
we mean that $\del\E(\ua) = 0$ and there exists $c_0 > 0$ such that
\begin{equation}
  \label{eq:pt:strongstabeq}
  \< \ddel \E(\ua) v, v \> \geq
  c_0 \| \D v \|_{L^2}^2 \quad \forall v \in \Usz.
\end{equation}
Since $\E \in C^k(\Us^{1,2})$ and $k \geq 2$ it is clear that a
strongly stable equilibrium is also a solution to \eqref{eq:pt:atm}
(but not vice-versa).

\begin{remark}\label{rem_stab}
  In \cite{2014-dislift}, \eqref{eq:pt:strongstabeq} is proven rigorously for an
  anti-plane screw dislocation, under restrictive assumptions on the interatomic
  potential. However, we cannot see how one might in general prove such a
  result. Nevertheless, in all numerical experiments that we have undertaken to
  date we do observe it {\it a posteriori}.
\end{remark}





\subsection{Regularity}
\label{sec:pt:regularity}
Our approximation error analysis in subsequent sections requires
estimates on the decay of the elastic fields away from the defect
core. These results do not require strong stability of solutions, but
only stability of the homogeneous lattice,
\begin{equation}
  \label{eq:pt:stab_lattice}
  \sum_{\ell \in \mA\Z^d} \b\< \ddel V(\bfO) Dv, Dv \b\> \geq c_\mA \|
  \D v \|_{L^2}^2 \qquad \forall v \in \Usz, \quad \text{for some $c_\mA > 0$.}
\end{equation}
It is easy to see that, if \eqref{eq:pt:strongstabeq} holds for {\em
  any} $u \in \Us^{1,2}$, then \eqref{eq:pt:stab_lattice} holds with
$c_\mA \geq c_0$; see \S~\ref{sec:prf_far_field_stab}.

Our first main result is the following decay estimate, which forms the basis of
our subsequent approximation error analysis. While it is widely assumed that the
decay $|Du(\ell)| \lesssim |\ell|^{-d}$ holds (e.g., \cite{Balluffi}), we are
unaware of rigorous proofs in this direction, or of results for higher-order
gradients.

\begin{theorem}
  \label{th:pt:regularity}
  Suppose $k \geq 3$, that the lattice is stable \eqref{eq:pt:stab_lattice}, and
  that $u \in \Us^{1,2}$ is a critical point, $\del\E(u) = 0$. Then there exist
  constants $C > 0, u_\infty \in \R^m$ such that, for $1 \leq j \leq k-2$, and for
  $|\ell|$ sufficiently large,
  \begin{equation}
    \label{eq:pt:regularity}
    |D^j u(\ell)| \leq C |\ell|^{1-d-j} \qquad \text{and} \qquad
    |u(\ell) - u_\infty| \leq C |\ell|^{1-d}.
  \end{equation}
\end{theorem}
\begin{proof}
  The proof for the cases $j = 0, 1$ is given in
  \S~\ref{sec:reg:decay_Du}. The proof for the case $j > 1$ is given
  in \S~\ref{sec:reg:higher_pt}.
\end{proof}

In what follows we assume $k\geq 4$, although some results are still
true with $k=3$.

\subsection{Clamped boundary conditions}
\label{sec:pt:clamped}
The simplest computational scheme to approximately solve \eqref{eq:pt:atm} is to
project the problem to a finite-dimensional subspace. Due to the decay estimates
\eqref{eq:pt:regularity} it is reasonable to expect that simply truncating to a
finite domain yields a convergent approximation scheme.

We choose a computational domain $\Omega_R \subset \L$ satisfying $(B_R\cap\L)
\subset \Omega_R$, define the approximate displacement space
\begin{equation}
  \label{eq:pt:defn_Us0}
  \Us^0(\Omega_R) := \b\{ v \in \Usz \bsep v = 0 \text{ in } \L \setminus
  \Omega \b\},
\end{equation}
and solve the finite-dimensional optimisation problem
\begin{equation}
  \label{eq:pt:clamped}
  u_R^0 \in \arg\min \b\{ \E(u) \bsep u \in \Us^0(\Omega_R) \b\}.
\end{equation}

Since ${\rm dim}(\Us^0(\Omega_R)) < \infty$, \eqref{eq:pt:clamped} is
computable. Moreover, since it is a pure Galerkin projection of
\eqref{eq:pt:atm} it is relatively straightforward to prove an error estimate.

\begin{theorem}
  \label{th:pt:clamped}
  Let $\ua$ be a strongly stable solution to \eqref{eq:pt:atm}, then there exist
  $C, R_0 > 0$ such that, for all $R \geq R_0$ there exists a strongly stable
  solution $\ua^0_R$ of \eqref{eq:pt:clamped} satisfying
  \begin{equation}
    \label{eq:pt:clamped:errest}
    \| \D \ua^0_R - \D \ua \|_{L^2} \leq C R^{-d/2} \qquad
    \text{and} \qquad \b| \E(\ua^0_R) - \E(\ua) \b| \leq C R^{-d}.
  \end{equation}
\end{theorem}
\begin{proof}[Idea of proof]
  We shall construct a truncation operator $T_R : \Us^{1,2} \to
  \Us^0(\Omega_R)$ such that $T_R v = 0$ in $\L \setminus B_R$, and
  which satisfies $\| \D T_R v - \D v \|_{L^2} \leq C \| \D v
  \|_{L^2(\R^d \setminus B_{R/2})}$. Since $\del\E$ and $\ddel \E$ are
  continuous it follows that $\ddel\E(T_R \ua)$ is positive definite
  for sufficiently large $R$ and that
  \begin{displaymath}
    \| \del\E(T_R \ua) \|_{(\Us^{1,2})^*} 
    = \| \del\E(T_R \ua) - \del\E(\ua) \|_{(\Us^{1,2})^*}
    \lesssim \|\D T_R \ua - \D \ua \|_{L^2} \to 0,
  \end{displaymath}
  as $R \to \infty$. The inverse function theorem (IFT) yields the existence of
  a solution $\ua^0_R$, for sufficiently large $R$, satisfying
  \begin{displaymath}
    \| \D \ua^0_R - \D T_R \ua \|_{L^2} \leq C \| \D T_R \ua - \D \ua \|_{L^2},
  \end{displaymath}
  and consequently also $\| \D \ua^0_R - \D\ua \|_{L^2} \leq C \| \D T_R \ua -
  \D \ua \|_{L^2}$.

  Finally, the regularity estimate \ref{th:pt:regularity} yields the stated rate
  in terms of $R$.  The proof is detailed in \S \ref{sec:approx:dir}.
\end{proof}

\begin{remark}[Computational cost]
  \label{rem:pt:comp_cost}
  In addition to the assumptions of Theorem~\ref{th:pt:clamped}, assume that $R
  \approx N^{1/d}$, which is a shape regularity condition for $\Omega_R$, then
  the error estimate \eqref{eq:pt:clamped:errest} reads
  \begin{equation}
    \label{eq:pt:clamped:errest_cost}
    \| \D \ua^0_R - \D \ua \|_{L^2} \leq C N^{-1/2} \qquad
    \text{and} \qquad \b| \E(\ua^0_R) - \E(\ua) \b| \leq C N^{-1}.
  \end{equation}
  In particular, if \eqref{eq:pt:clamped} can be solved with linear
  computational cost, then \eqref{eq:pt:clamped:errest_cost} is an
  error estimate in terms of the computational cost required to solve
  the approximate problem.
\end{remark}

\subsection{Periodic boundary conditions}
\label{sec:pt:per}
For simulating point defects (and often even dislocations), periodic boundary
conditions appear to be by far the most popular choice.  To implement periodic
boundary conditions, let $\omega_R \subset \R^d$ be connected such that
$B_R\subset \omega_R$, and $\mB = (b_1, \dots, b_d) \in \R^{d \times d}$ such
that $b_i \in \mA\Z^d$, $\bigcup_{\alpha \in \Z^d} \b(\mB \alpha + \omega_R\b) =
\R^d$, and the shifted domains $\mB\alpha+\omega_R$ are disjoint.

Let $\Omega_R := \omega_R \cap \L$ be the {\em periodic computational
  domain} and $\Omega_R^\per := \bigcup_{\alpha \in \Z^d} (\mB\alpha +
\Omega_R)$ the periodically repeated domain (with an infinite lattice of
defects). For simplicity, suppose that $\omega_R$ is compatible with
$\T_\L$, i.e., there exists a subset $\T_R \subset \T_\L$ such
that ${\rm clos}(\omega_R) = \cup \T_R$.
%
The space of admissible periodic displacements is given by
\begin{displaymath}
  \Us^\per(\Omega_R) := \b\{ u : \Omega_R^\per \to \R^m \bsep
  u(\ell+b_i) = u(\ell) \text{ for $\ell \in \Omega_R^\per, i = 1,
    \dots, d$} \b\}.
\end{displaymath}

The energy functional for periodic relative displacements
$u \in \Us^\per(\Omega_R)$ is given by
\begin{align*}
  \E^\per_R(u) := \sum_{\ell \in \Omega_R} V_\ell(D u(\ell)).
\end{align*}
For this definition to be meaningful, we assume for the remainder of the
discussion of periodic boundary conditions that $B_{\Rcore+\rcut} \cap \L
\subset \Omega$, that is, $R > \Rcore + \rcut$.


The computational task is to solve the finite-dimensional optimisation
problem
\begin{equation}
  \label{eq:pt:per}
  \ua^\per_R \in \arg\min \b\{ \E^\per_R(u) \bsep u \in
  \Us^\per(\Omega_R) \b\}.
\end{equation}

\begin{theorem}
  \label{th:pt:per}
  Let $\ua$ be a strongly stable solution to \eqref{eq:pt:atm}, then
  there exist $C, R_0 > 0$ such that, for any periodic computational
  domain $\Omega_R$ with associated continuous domain $\omega_R$
  satisfying $B_R \subset \omega_R$ for some $R \geq R_0$, there exists a
  strongly stable solution $\ua^\per_R$ to \eqref{eq:pt:per}
  satisfying
  \begin{equation}
    \label{eq:pt:per:errest}
    \b\| \D \ua^\per_R - \D \ua \b\|_{L^2(\omega_R)} \leq C R^{-d/2} \qquad
    \text{and} \qquad
    \b|\E(\ua) - \E^\per_R(\ua^\per_R)\b| \leq C R^{-d}.
  \end{equation}
\end{theorem}
\begin{proof}[Idea of proof]
  The proof proceeds much in the same manner as for Theorem \ref{th:pt:clamped},
  but some details are more involved due to the fact that \eqref{eq:pt:per} is
  {\em not} a Galerkin projection of \eqref{eq:pt:atm}. The main additional
  difficulty is that the strong convergence
  $\D T_R \ua|_{\omega_R} \to \D \ua|_{\omega_R}$ does not immediately imply
  stability of the periodic hessian, i.e.,
  \begin{equation}
    \label{eq:pt:per:errest_stab}
    \b\< \ddel\E_R^\per(T_R \ua) v, v \b\> 
    \geq c_0 \| \D v \|_{L^2(\omega_R)}^2 
    \qquad \forall v \in \Us^\per(\Omega_R).
  \end{equation}
  To prove this result, we consider the limit as $R \to \infty$ (with
  an arbitrary sequence of associated domains $\Omega_R$) and decompose
  test functions into a core and a far-field component $v = v^{\rm co}
  + v^{\rm ff}$, where $v^{\rm co} = T_S v$, with $S = S(R) \uparrow
  \infty$ as $R \to \infty$ ``sufficiently slowly''. We then show that
  stability of $\ddel\E(\ua)$ implies positivity of $\< H_R v^{\rm
    co}, v^{\rm co} \>$ while stability of the homogeneous lattice
  \eqref{eq:pt:stab_lattice} implies positivity of $\< H_R v^{\rm ff},
  v^{\rm ff} \>$. The cross-terms vanish in the limit. In this manner
  we obtain \eqref{eq:pt:per:errest_stab} for sufficiently large
  $R$. The details are given in~\S \ref{sec:approx:per}.
\end{proof}

\begin{remark}
  1. Remark \ref{rem:pt:comp_cost} applies verbatim to periodic
  boundary conditions.

  2. Compared with Theorem \ref{th:pt:clamped} we now only control the
  geometry in the computational domain $\Omega_R$. We can, however,
  ``post-process'' to obtain a global defect geometry $\bar{v}^\per :=
  \Pi_R \ua^\per_R$ (slightly abusing notation since
  $\ua^\per_R \not\in \UsH(\L)$), for which we still get the
  estimate $\|\D\bar{v}^\per - \D\ua \|_{L^2} \leq C R^{-d/2}$.
\end{remark}

\subsection{Boundary conditions from linear elasticity}
\label{sec:pt:lin}
In this section we consider a scheme where the elastic far-field of the crystal
is approximated by linearised lattice elasticity. The idea is to define a
computational domain $\Omega \subset \L$ and to use the lattice Green's
function, or other means, to explicitly compute the displacement field and
energy in $\L \setminus \Omega$ as predicted by linearised elasticity. Our
formulation is inspired by classical as well as recent multiscale methods of
this type \cite{WoodwardRao:2002, Trinkle:2008, Sinclair:1971, XLi:2009}, but
simplified to allow for a straightforward analysis. Such schemes are employed
primarily in the simulation of dislocations, however we shall observe here that
there is considerable potential also for the simulation of point defects.

We fix a computational domain $\Omega_R \subset \L$ such that $B_R \cap \L
\subset\Omega_R$ (for $R\geq \Rcore$) and we linearise the interaction outside
of $\Omega_R$
\begin{equation}
  \label{eq:pt:Vlin}
  V(Du) \approx V(\bfO) + \< \del V(\bfO), Du \> + \smfrac12 \<
  \ddel V(\bfO) Du, Du \> =: V^\lin(Du).
\end{equation}
This results in a modified approximate energy difference functional
\begin{align*}
  \E^\lin_R(u) &:= \sum_{\ell \in \Omega_R}
  V_\ell\b(Du(\ell)\b) +
  \sum_{\ell \in \L \setminus \Omega_R} V^\lin\b(Du(\ell)\b).
\end{align*}
Analogously to Lemma \ref{th:extension_lemma} it follows that $\E^\lin_R$ can be
extended by continuity to a functional $\E^\lin_R \in C^k(\Us^{1,2})$.

Thus, we aim to compute
\begin{equation}
  \label{eq:pt:lin}
  \ulin_R \in \arg\min \b\{ \E^\lin_R(u) \bsep u \in \UsH(\L) \b\}.
\end{equation}

\begin{remark}
  \label{rem:lin:efficient}
  The optimisation problem \eqref{eq:pt:lin} is still
  infinite-dimensional, however, by defining $\Omega_R' := \Omega_R \cup \bigcup_{\ell
    \in \Omega} \Rg_\ell$ and the effective energy
  functional
  \begin{displaymath}
    \E^{\rm red}_R(u) := \inf \B\{ \E^\lin_R(v) \bsep v \in
    \UsH(\L), v|_{\Omega_R'} =
    u|_{\Omega_R'} \B\},
  \end{displaymath}
  for any $u : \Omega_R' \to \R^m$, it can be reduced to an effectively
  finite-dimensional problem.  The reduced energy $\E^{\rm
    red}_\Omega$ can be computed efficiently employing lattice Green's
  functions or similar techniques \cite{WoodwardRao:2002,
    Trinkle:2008, Sinclair:1971, XLi:2009}. This process likely
  introduces additional approximation errors, which we ignore
  subsequently.  Thus, we only present an analysis of an idealised
  scheme, as a foundation for further work on more practical variants
  of \eqref{eq:pt:lin}.
\end{remark}

\begin{theorem}
  \label{th:pt:lin}
  Let $\ua$ be a strongly stable solution to \eqref{eq:pt:atm}, then
  there exist $C, R_0 > 0$ such that for all domains $\Omega_R \subset
  \L$ with $B_R \cap \L \subset \Omega_R$ and $R \geq R_0$, there exists a
  strongly stable solution of \eqref{eq:pt:lin} satisfying
  \begin{equation}
    \label{eq:pt:lin:errest}
     \b\| \D\ulin_R - \D\ua \b\|_{\UsH} \leq C R^{-3d/2} \qquad
     \text{and} \qquad  \b|\E(\ulin_R) - \E(\ua)\b| \leq C R^{-2d}.
  \end{equation}
\end{theorem}
\begin{proof}[Idea of proof]
  For the linear elasticity method, the computational space is the
  same as for the full atomistic problem, hence the error is
  determined by the consistency error committed when we replaced $V$
  with $V^\lin$ in the far-field. This error is readily estimated by a
  remainder in a Taylor expansion,
  \begin{displaymath}
    \b|\del V^\lin(D\ua(\ell)) - \del V(D\ua(\ell)) \b| \lesssim |D\ua(\ell)|^2,
  \end{displaymath}
  which immediately implies that
  \begin{displaymath}
    \b| \< \del \E^\lin_R(\ua) - \del \E(\ua), v \b\>\b| \lesssim \|
    D\ua \|_{\ell^4(\L \setminus \Omega_R)}^2 \| D v
    \|_{\ell^2(\L\setminus\Omega_R)} \lesssim \|
    D\ua \|_{\ell^4(\L \setminus \Omega_R)}^2 \| \D v \|_{L^2}.
  \end{displaymath}
  After establishing also stability of $\ddel\E^\lin_R$, which
  follows from a similar argument we obtain $\| \D
    \ulin_R - \D \ua \|_{L^2} \lesssim \| D \ua
  \|_{\ell^4(\L\setminus\Omega_R)}^2$, and employing the decay estimate
  \eqref{eq:pt:regularity} yields the stated error bound.

  The details of the proof are given in \S~\ref{sec:approx:lin}.
\end{proof}

\subsection{Boundary conditions from nonlinear elasticity}
\label{sec:pt:ac}
\quad A natural further question to ask is whether employing nonlinear
elasticity in the far-field instead of linear elasticity can improve further
upon the approximation error. In this context it is only meaningful to employ
{\em continuum} nonlinear elasticity, since our original atomistic model can
already be viewed as a lattice nonlinear elasticity model. This leads us into
considering a class of multiscale schemes, atomistic-to-continuum coupling
methods (a/c methods), that has received considerable attention in the numerical
analysis literature in recent years. We refer to the review article
\cite{LuOr:acta} for an introduction and comprehensive references. A key
conceptual difference, from an analytical point of view, between a/c methods and
the methods we considered until now is that they exploit higher-order
regularity, that is, the decay of $D^2 \ua$, rather than only decay of
$D\ua$. Methods of this kind were pioneered, e.g., in \cite{Ortiz:1995a,
  Shenoy:1999a, Shimokawa:2004, XiBe:2004}.

Due to the relative complexity of a/c coupling schemes we shall not present
comprehensive results in this section, but instead illustrate how existing error
estimates can be reformulated within our framework. This extends previous works
such as \cite{OrtnerShapeev:2011pre,Or:2011a,PRE-ac.2dcorners} and presents a
framework for ongoing and future development of a/c methods and their analysis;
see for example \cite{OrtnerZhang2014, LiOrtnerShapeevEtAl-blended,
  2012-CMAME-optbqce, 2013-PRE-bqcfcomp}, and references therein, for works in
this direction.

We choose an {\em atomistic region} $\Omega^\a_R \subset \L$, an {\em interface
  region} $\Omega^\i_R$ and $\omega_R \subset \R^d$ a {\em continuum} simply
connected domain such that $\Omega^\a_R \cup \Omega^\i_R \subset \omega_R$.  Let
$\T_R$ be a regular triangulation of $\omega_R$, let $h(x) := \max_{T \in \T_R,
  x \in T} {\rm diam}(T)$, and let $I_R$ denote the corresponding nodal
interpolation operator. We let $R$ and $R_\c$ denote the sizes of $\Omega^\a_R$
and $\omega_R$ in the sense that
\begin{equation}\label{eq:ac:pt:radii}
	B_R \cap \L \subset \Omega^\a_R
\qquad\text{and}\qquad
	B_{R_\c}
    \subset
    \omega_R
    \subset
    B_{c_0 R_\c}
\end{equation}
for some $c_0>0$.

As space of admissible displacements we define
\begin{align*}
  \Us^0(\T_R) &:= \b\{ u \in C(\R^d; \R^d) \bsep u|_T \text{ is
    affine for all } T \in \T_R, \text{ and } u|_{\R^d\setminus\omega_R}
  = 0 \b\}.
\end{align*}
We consider a/c coupling energy functionals, defined for $u \in \Us^0(\T_R)$, of
the form
\begin{equation} \label{eq:pt:ac:defn_E}
  \E^\ac_R(u) := \sum_{\ell \in \Omega^\a_R} V_\ell(Du(\ell))
  +\sum_{\ell \in \Omega^\i_R} V_\ell^\i(Du(\ell))
  +\sum_{T \in \T_R} v_T^{\rm eff} W(\D u),
\end{equation}
where the various new terms are defined as follows:
\begin{itemize}
\item For each $T \in \T_R$, $v^{\rm eff}(T) := {\rm vol}\b( T
  \setminus \cup_{\ell \in \Omega^\a_R \cup \Omega^\i_R} {\rm vor}(\ell) \b)$
  is the {\em effective volume} of $T$, where ${\rm vor}(\ell)$
  denotes the Voronoi cell associated with the lattice site $\ell$;
\item $V_\ell^\i \in C^k( (\R^d)^\Rg )$ is an {\em interface potential},
  which specifies the coupling scheme;
\item $W(\mF) := V( \mF \cdot \Rg)$ is the {\em Cauchy--Born strain energy
    function}, which specifies the continuum model.
\end{itemize}

With this definition it is again easy to see that $\E^\ac_R \in
C^k(\Us^0(\T_R))$. We now aim to compute
\begin{equation}
  \label{eq:pt:ac}
  u^\ac_R \in \arg\min\b\{ \E^\ac_R(u) \bsep u \in \Us^0(\T_R)  \b\}.
\end{equation}

The choice of the interface site-potentials $V^\i_\ell$ is the key component in
the formulation of a/c couplings. Many variants of a/c couplings exist that fit
within the above framework \cite{LuOr:acta}. In order to demonstrate how to
apply our framework to this setting, we shall restrict ourselves to QNL type
schemes \cite{Shimokawa:2004, E:2006, PRE-ac.2dcorners}, but our discussion
applies essentially verbatim to other force-consistent energy-based schemes such
as \cite{Shapeev:2010a, Shapeev:3D, MakrMitsRos:2012}. For other types of a/c
couplings the general framework is still applicable; see in particular
\cite{LiOrtnerShapeevEtAl-blended} for a complete analysis of blending-type a/c
methods.

As a starting point of our present analysis we {\em assume} a result that is
proven in various forms in the literature, e.g., in \cite{PRE-ac.2dcorners,
  OrtnerShapeev:2011pre, MakrMitsRos:2012}: We assume that there exist $\eta>0$
and $c_1>0$ such that there exists a strongly stable solution $\uac_R$ to
\eqref{eq:pt:ac} satisfying
\begin{equation}
  \label{eq:ac:pt:est}
  \|\D \uac_R-\D \ua\| \leq c_1 \b(\| h
  D^2 \ua \|_{\ell^2(\L\cap(\omega_R\setminus B_R))}
  +
  \|D \ua\|_{\ell^2(\L \setminus B_{R_\c/2})}\b)
  ,
\end{equation}
provided that
$\|h D^2 \ua \|_{\ell^2(\L\cap(\omega_R\setminus B_R))} + \|D \ua\|_{\ell^2(\L
  \setminus B_{R_\c/2})} \leq \eta$.
Such a result follows from consistency and stability of an a/c scheme and
applying of the Inverse Function Theorem along similar lines as in the preceding
sections.

\begin{proposition}\label{prop:ACresult}
  Let $\ua$ be a strongly stable solution of \eqref{eq:pt:atm} and assume that \eqref{eq:ac:pt:radii} and \eqref{eq:ac:pt:est} hold.
  Further we require that $\omega_R$ and
  $\T_R$ satisfy the following quasi-optimality conditions:
  \begin{align}
    c_2 R^{1+2/d} \leq R_\c \leq c_3 R^{1+2/d}
    ,\qquad \text{and} \qquad
    \label{eq:ac:restrictions}
    &|h(x)| \leq c_4 \B(\smfrac{|x|}{R} \B)^\beta ~~ \text{ with }
    ~~
      \beta < \smfrac{d+2}{2}.
  \end{align}

  Then there exists a constant $C$, depening on $\eta$, $c_2$,
    $c_3$, $c_4$, and $\beta$ such that, for $R$ sufficiently large,
  \begin{align}
    \label{eq:pt:rate-ac}
    \b\| \D\uac_R - \D\ua \b\|_{L^2} &\leq C R^{-d/2-1}.
  \end{align}
\end{proposition}
\begin{proof}[Idea of proof]
  The proof consists in estimating the right-hand side of \eqref{eq:ac:pt:est}
  in terms of $R$.  Note that assuming \eqref{eq:ac:restrictions} ensures that
  the truncation term $\|D \ua \|_{\ell^2(\L \setminus B_{R_\c/2})}$ does not
  dominate the coarse-graining term
  $\| h D^2 \ua \|_{\ell^2(\L\cap(\omega_R \setminus B_R))}$.
\end{proof}

\begin{remark}
  1. It is interesting to note that an atomistic continuum coupling is not
  competitive when compared against coupling to lattice linear elasticity. The
  primary reason for this is that the loss of interaction symmetry which causes
  a first-order coupling error at the a/c interface (the finite element error
  could be further reduced by considering higher order finite elements
  \cite{OrtWu-ac}). Since $|\D^j \ua(x)| \lesssim |x|^{1-d-j}$ the linearisation
  error $|\D\ua(x)|^2 \lesssim |x|^{-2d}$ is smaller than the coupling error
  $|\D^2 \ua(x)| \lesssim |x|^{-d-1}$.

  2. Using our framework, the analysis in \cite{OrtWu-ac} suggests that one can
  generically expect the rate $R^{-d -2}$ for the energy error.

  3. To convert \eqref{eq:pt:rate-ac} into an estimate in terms of computational
  complexity, we note that, if we also have $|h(x)| \geq c_5 (|x|/R)^{\beta'}$
  with $\beta' > 1$, then the total number of degrees of freedom (in the
  atomistic and continuum region) is bounded by $N_{\rm dof} \leq C R^d$. The
  error estimate then reads
  \begin{displaymath}
    \b\| \D\ua - \D\uac \b\|_{L^2} \leq C \cases{ N_{\rm dof}^{-1}, &
      d = 2, \\ N_{\rm dof}^{-5/6}, & d = 3. } \qedhere
  \end{displaymath}
\end{remark}

\subsection{Numerical results}
\label{sec:pt:num}
\subsubsection{Setup}
\label{sec:num:pt:setup}
We present two examples of ``hypothetical'' point defects in a 2D
triangular lattice
\begin{equation}
  \label{eq:num:pt:A}
  \mA \Z^d \quad \text{where} \quad \mA = 
  {\footnotesize \Big(\begin{matrix} 1 & 1/2 \\ 0 & \sqrt{3}/2 \end{matrix} \Big)},
\end{equation}
a vacancy and an interstitial, both displayed in Figure
\ref{fig:pt_defects_comp}. For the vacancy, let $\L = \mA \Z^2 \setminus
\{0\}$. For the interstitial, let $\L = \mA \Z^2 \cup \{ (1/2, 0)\}$. (We tested
various positions for the interstitial and the centre of a bond between two
nearest neighbours appeared to be the only stable one for the interaction
potential that we employ.)  For each $\ell \in \L$, let $\Rg(\ell)$ denote the
set of directions connecting to $\ell$, defined by the bonds displayed in Figure
\ref{fig:pt_defects_comp}. Then, the site energy is defined by
\begin{align*}
  & V_\ell(Dy(\ell)) = \sum_{\rho \in \Nhd_\ell} \phi\b( |D_\rho
  y(\ell)|\b)
  + G\bg( \sum_{\rho \in \Nhd_\ell} \psi\b(|D_\rho y(\ell)|\b) \bg),  \\
  & \phi(r) = e^{-2\alpha (r-1)} - 2 e^{-\alpha (r-1)}, \quad
  \psi(r) = e^{-\beta r}, \quad G(s) = \gamma \b( (s - s_0)^2 +
  (s-s_0)^4 \b), \\
  & \text{with parameters } \alpha = 4, \beta = 3, \gamma = 5, s_0 = 6
\psi(0.9).
\end{align*}

To compute the equilibria we employ a robust preconditioned L-BFGS
algorithm specifically designed for large-scale atomistic optimisation
problems \cite{CsGoOrPa:opt}. It is terminated at an
$\ell^\infty$-residual of $10^{-7}$.

We exclusively employ hexagonal computational domains. We slightly re-define
$N$, letting it now denote the number of atoms in the {\em inner computational
  domain}, that is, $\# \Omega_N$ in the ATM-DIR, ATM-PER and LIN methods and
$\# \Omega_N^\a$ in the AC method. Then, our analysis predicts the following
rates of convergence for both model problems, \medskip
\begin{center}
  Summary of Convergence Rates \\
  (Point Defect in Two Dimensions) \\[1mm]
  \begin{tabular}{r|cccc}
    \hline
    Method & ATM-DIR & ATM-PER & LIN & AC \\ \hline 
    Energy-Norm & $N^{-1/2}$ & $N^{-1/2}$ & $N^{-3/2}$ & $N^{-1}$ \\ 
    Energy & $N^{-1}$ & $N^{-1}$ & $N^{-2}$ & $N^{-2}$ \\
    \hline
  \end{tabular}
\end{center}  
where the rate $N^{-2}$ for the energy in the AC case is predicted in
\cite{OrtWu-ac}.

\medskip

We make some final remarks concerning the LIN and AC methods:
\begin{itemize}
\item[\bf LIN:] For the experiments in this paper, we did not
  implement an efficient variant based on Green's functions or fast
  summation methods. Instead, we chose as an {\em inner domain}
  $\Omega_N$ a hexagon of side-length $K$ (then, $N \approx 3 K^2$)
  within a larger domain of a hexagon of side-length $K^3$. It can be
  readily checked that this modification of the method does not affect
  the convergence rates.

\item[\bf AC:] To generate the finite element mesh, we first generate
  a hexagonal {\em inner domain} $\Omega_N^\a$ with sidelength $K$
  (then, $N \approx 3 K^2$), with an inner triangulation. The
  triangulation is then extended by successively adding layers of
  elements, at all time retaining the hexagonal shape of the domain,
  until the sidelength reaches $K^2 \approx N$. This construction is
  the same as the one used in \cite{2013-PRE-bqcfcomp,
    2012-CMAME-optbqce}.
\end{itemize}

\subsubsection{Discussion of results}\label{sec:num:pt:disc}
The graphs of $N$ versus the geometry error and the energy error for the vacancy
problem are presented in Figure \ref{fig:vacancy_err} and for the insterstitial
problem in Figure \ref{fig:interstitial_err}.

All slopes are as predicted with mild pre-asymptotic regimes for the ATM-PER and
AC methods. The only exception is the energy for the LIN method, which displays
a faster decay than predicted. We can offer no explanation at this point.

The main feature we wish to point out is the difference of at least
an order of magnitude in the prefactor for the geometry error and of {\em three}
orders of magnitude in the prefactor for the energy error. Most likely, this
discrepancy is simply due to the fact that the interstitial causes a much more
substantial distortion of the atom positions.

The prefactor is a crucial piece of information about the accuracy of
computational schemes that our analysis does not readily reveal. Ideally, one
would like to establish estimates of the form
$\| D \bar{u}^{\rm apx}_N - D \ua \|_{\ell^2} \leq C_* N^{-p} + o(N^{-p})$,
where $C_*$ and $p$ can be given explicitly, however much finer
context-sensitive estimates would be required to achieve this.

\begin{figure}
  \begin{center}
    \includegraphics[height=6.8cm]{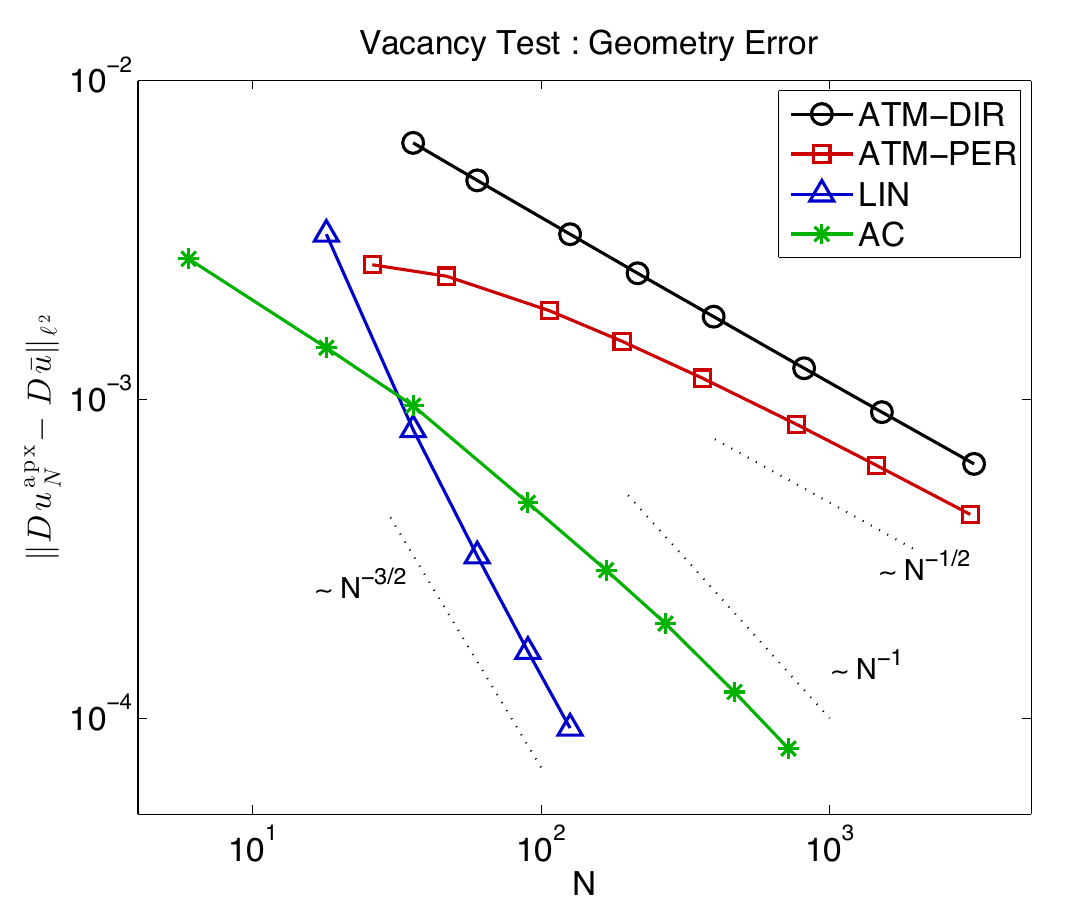}
    \includegraphics[height=6.8cm]{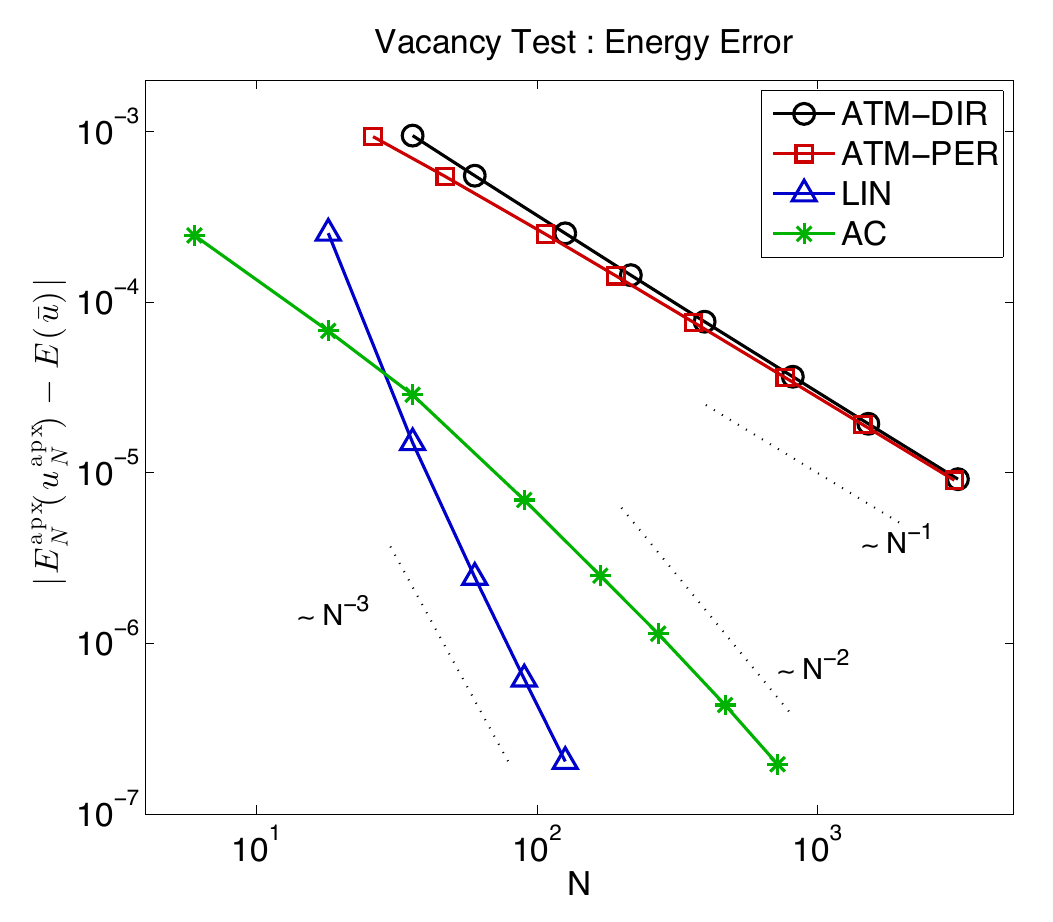} \\
    \footnotesize (a) \hspace{7cm} (b)
  \end{center}
  \caption{\label{fig:vacancy_err} Rates of convergence, in the
    vacancy example, of four types of boundary conditions for (a) the
    geometry error and (b) the energy error. $N$ denotes the number of
    atoms in the {\em inner computational domain}; see
    \S~\ref{sec:num:pt:setup} for definitions.}
\end{figure}

\begin{figure}
  \begin{center}
    \includegraphics[height=6.8cm]{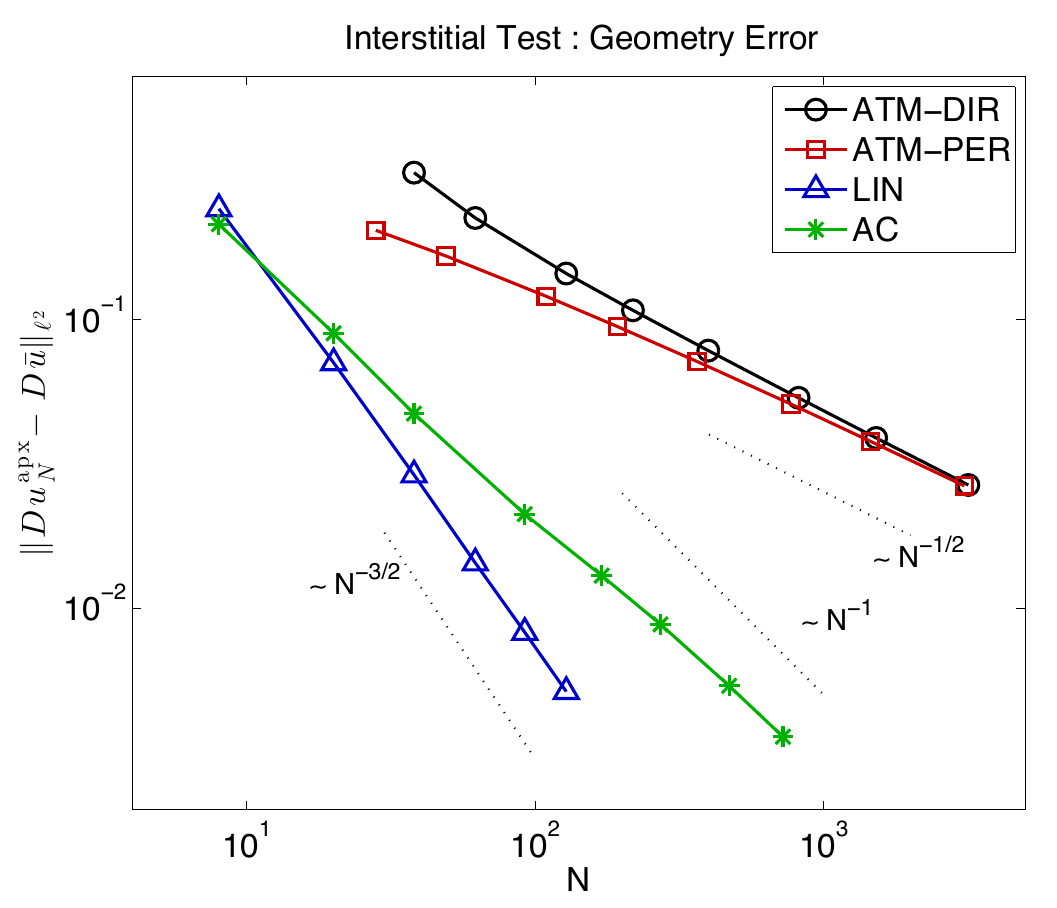}
    \includegraphics[height=6.8cm]{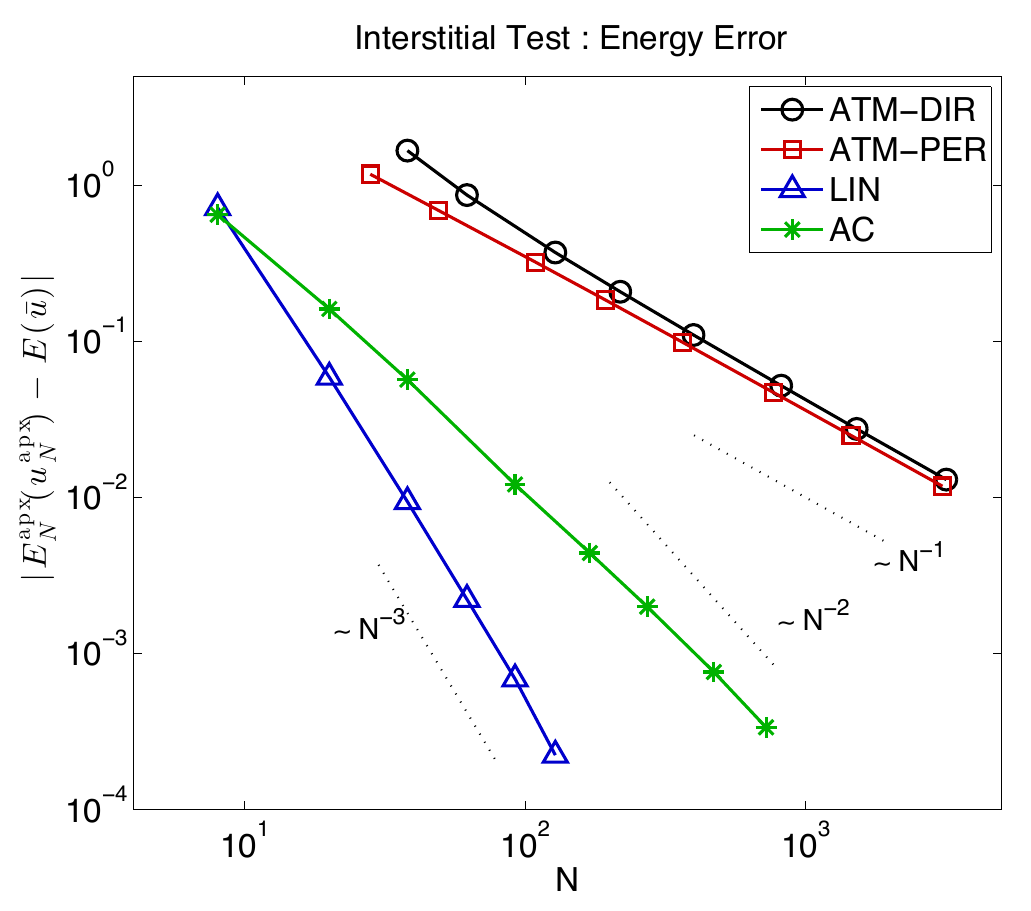} \\
    \footnotesize (a) \hspace{7cm} (b)
  \end{center}
  \caption{\label{fig:interstitial_err} Rates of convergence, in the
    interstitial example, of four types of boundary conditions for (a)
    the geometry error and (b) the energy error;
    see \S~\ref{sec:num:pt:setup} for definitions.}
\end{figure}

\FloatBarrier

\section{Dislocations}
\label{sec:sum-dis}
We now present an atomistic model for dislocations and analogous
regularity and approximation results. To avoid excessive duplication
we will occasionally build on and reference \S~\ref{sec:sum-pt}. Our
presentation also builds on the descriptions in \cite{ArizaOrtiz:2005,
  Hudson:stab}. For more general introductions to dislocations,
including modeling aspects as well as analytical and computational
solution strategies we refer to \cite{BulatovCai, HirthLothe}.

\subsection{Atomistic model}
\label{sec:dis:atm}
We consider a model for straight dislocation lines obtained by projecting a 3D
crystal.
Briefly, let $\mB \Z^3$ denote a 3D Bravais lattice, oriented in such a way that
the dislocation direction can be chosen parallel to $e_3$ and the Burgers vector
can be chosen as $\burg = (\burg_1, 0, \burg_3) \in \mB \Z^3$. \label{sym:burg}
We consider displacements $W : \mB \Z^3 \to \R^3$ of the 3D lattice that
\alert{are independent of the direction of the dislocation direction, i.e.,
  $e_3$.} Thus, we choose a projected reference lattice
$\L := \mA \Z^2 := \{ (\ell_1,\ell_2) \sep \ell \in \mB \Z^3 \}$, and identify
$W(X) = w(X_{12})$, where $w : \L \to \R^3$, and here and throughout we write
$a_{12} = (a_1, a_2)$ for a vector $a \in \R^3$.  It can be readily checked that
this projection is again a Bravais lattice.

We may again choose a regular triangulation $\T_\L$ satisfying
$\T_\L + \rho = \T_\L$ for all $\rho \in \L$. Each lattice function
$v : \L \to \R^m$ has an associated P1 interpolant $I v : \R^2 \to \R^m$ and we
identify $\D v= \D Iv$. Further, we recall the definition of the spaces
$\Usz, \UsH$ from \eqref{eq:pt:defnUszUsH}.

Let $\hat{x} \in \R^2$ be the position of the dislocation core and
$\Gamma := \{ x \in \R^2 \sep x_2 = \hat{x}_2, x_1 \geq \hat{x}_1 \}$ the
``branch-cut'' (cf. \eqref{eq:disl:linel_pde}),
chosen such that $\Gamma \cap \L = \emptyset$.
\label{sym:xhat-Gamma} In order to model dislocations the site energy potential
must be invariant under lattice slip. Normally, this is a consequence of
permutation invariance of the site energy, but here we will
formulate a minimal assumption. To that end, we define the slip operator $S_0$
acting on a displacement $w : \L \to \R^3$, or $w : \R^2 \to \R^3$, by
\begin{equation}
  \label{eq:disl:defn_Sop}
  S_0 w(x) := \cases{ w(x), & x_2 > \hat{x}_2, \\
    w(x - \burg_{12}) - \burg, & x_2 < \hat{x}_2.}
\end{equation}
This operation leaves the 3D atom configuration corresponding to the
displacement $u$ invariant: if $Y(X) = X + w(X_{12})$ and
$Y^S(X) = X + S_0 w(X_{12})$, then $Y(X)=Y^S(X)$ for $X_2 > \hat{x}_2$, while
for $X_2 < \hat{x}_2$,
\begin{displaymath}
  Y^S(X) = X + w(X_{12} - \burg_{12}) - \burg 
  = [X-\burg] + w([X-\burg]_{12})
  = Y(X-\burg),
\end{displaymath}
that is, $Y^S$ represents only a relabelling of the atoms. Therefore, formally,
if $V(Dw)$ is the site energy potential as a function of displacement, then it
must by invariant under the map $w \mapsto S_0 w$:
\begin{equation}
  \label{eq:slipinvariance_formal}
  \cases{V(DS_0 w(\ell)) = V(Dw(\ell)), & \text{ for $\ell_2 > \hat{x}_2$}, \\
    V(DS_0w(\ell+\burg_{12})) = V(Dw(\ell)), & \text{ for $\ell_2 < \hat{x}_2$.}
  }
\end{equation}
In \eqref{eq:disl:slip_invariance} below we will restate this assumption for a
restricted class of displacements only, which will allow us to continue to
employ the finite range interaction assumption.

Dislocations in an infinite lattice store infinite energy due to their
topological singularity. We therefore decompose the total displacement
$w = \up+\ua$ into a {\em far-field predictor} $\up$ and a finite energy {\em
  core corrector} $\ua$, where the latter belongs again to the energy space
$\UsH$.  There is no unique way to specify $\up$, but a natural choice is the
continuum elasticity solution: For a function
$u : \R^2 \setminus \Gamma \to \R^m$ that has traces from above and below, we
denote these traces, respectively, by $u(x\pm), x \in \Gamma$.  Then we seek
$\ulin \in C^\infty(\R^2 \setminus \Gamma; \R^3)$ satisfying
\begin{equation}
  \label{eq:disl:linel_pde}
  \begin{split}
    \bbC_{i\alpha}^{j\beta} \frac{\partial^2 \ulin_i}{\partial
      x_\alpha \partial x_\beta} &= 0 \qquad \text{ in } \R^2
    \setminus \Gamma, \\
    \ulin(x+) - \ulin(x-) &= -\burg \qquad \text{ for } x \in \Gamma
    \setminus \{\hat{x}\},  \\
    \D_{e_2} \ulin(x+) - \D_{e_2} \ulin(x-) &= 0 \qquad \text{ for } x \in \Gamma
    \setminus \{\hat{x}\},
  \end{split}
\end{equation}
where the tensor $\bbC$ is the linearised Cauchy--Born tensor (derived
from the interaction potential $V$; see \S~\ref{sec:elast:lin} for
more detail).

In our analysis we require that applying the slip operator to the predictor map
$\up$ yields a smooth function in the half-space
\begin{equation}
  \label{eq:disl:defn_halfspace}
  \halfspace := \{ x_1 \geq \hat{x}_1 \} \setminus B_{\rdisl+\burg_1}(\hat{x})
\end{equation}
where $\rdisl$ is defined in Lemma \ref{th:disl:up_lemma} below. That is, we
require that $S_0\up \in C^\infty(\halfspace)$. \alert{Except in the pure screw
  dislocation case ($\burg_{12} = 0$) $\ulin$ does not satisfy this
  property. The origin of this conundrum is that linearised elasticity assumes
  infinitesimal displacements, yet we apply it in the large deformation regime
  near the defect core.}  To overcome this technical difficulty, instead of
$\up = \ulin$, we define the predictor
\begin{equation}
  \label{eq:disl:defn_u0}
  \up(x) := \ulin(\xi^{-1}(x)), \qquad \text{where} \qquad \xi(x) := x
  - \burg_{12} \smfrac{1}{2\pi} \eta(|x-\hat{x}|/\rdisl) \arg(x-\hat{x}),
\end{equation}
$\arg(x)$ denotes the angle in $(0,2\pi)$ between $\burg_{12}\propto e_1$ and
$x$, and $\eta \in C^\infty(\R)$ with $\eta = 0$ in $(-\infty, 0]$,
$\eta = 1$ in $[1, \infty)$ and $\eta' > 0$ in $(0, 1)$. While the distinction
between $\up$ and $\ulin$ is crucial, it arises from a subtle technical issue
and could be ignored on a first reading, especially in view of the following
lemma.

\begin{lemma}
  \label{th:disl:up_lemma}
  (i) Suppose that the lattice is stable \eqref{eq:pt:stab_lattice}, then
  $\ulin$ is well-defined. \alert{For $\hat{r}$ sufficiently large},
  $\xi : \R^2 \setminus \Gamma \to \R^2 \setminus \Gamma$ is a bijection, hence
  $\up$ is also well-defined on $\R^2 \setminus \Gamma$.

  (ii) We have $\D^j S_0\up(x+) = \D^j S_0\up(x-)$ for all $j \geq 0$ and
  for all $x \in \Gamma \cap \halfspace$.  In particular, upon
  extending $\up$ continuously to $\Gamma \cap \halfspace$ we obtain that
  $S_0\up \in C^\infty(\halfspace)$.

  (iii) There exists $C$ such that $|\D^n \up(x) - \D^n \ulin(\xi^{-1}(x))| \leq
  C |x|^{-n-1}$ for $x \in \R^2 \setminus (\Gamma \cup B_r)$; in
  particular $|\D^n \up(x)| \leq C |x|^{-n}$ for all $n \in \N$.
\end{lemma}
\begin{proof}
The proof is given in \S~\ref{sec:prf_disl_yd_lemma}.
\end{proof}

Statement (ii) implies that the net-Burgers vector of $\up$ (and hence of any
$\up+u, u \in \UsH$) is indeed $\burg$. Moreover, the fact that
$S_0\up \in C^\infty(\halfspace)$ will allow us to perform Taylor expansions of
finite differences.\alert{Statement (iii) indicates that $\up$ is an approximate
  far-field equilibrium,} which allows us to use $\up$ as a far-field boundary
condition (see Lemma \ref{th:disl:extension_lemma} below).

In order to keep the analysis as simple as possible we would like to keep the
convenient assumption made in the point defect case of a finite interaction
range in reference configuration. At first glance this contradicts the
invariance of the site energy under lattice slip \eqref{eq:disl:defn_Sop}, but
we can circumvent this by restricting the admissible corrector
displacements. Arguing as in \S~\ref{sec:app_cutoff_ref} we may choose
sufficiently large radii $\hat{r}_{\Adm}, \hat{m}_{\Adm}$ and define
\begin{displaymath}
  \label{eq:disl:defn_Adm}
  \Adm := \b\{ u : \L \to \R^3 \bsep \| \D u \|_{L^\infty} < \hat{m}_{\Adm}
  \text{ and } |\D u(x)| < 1/2 \text{ for } |x| > \hat{r}_{\Adm} \b\}.
\end{displaymath}
Upon choosing $\hat{m}_{\Adm}, \hat{r}_{\Adm}$ sufficiently large, we can ensure
that any potential equilibrium solution is contained in $\Adm$. Thus, the
restriction of admissible displacements to $\Adm$ is purely an analytical tool,
which ensures that we can treat $V$ as having finite range, despite admitting
slip-invariance.

For $w = \up + u, u \in \Adm$, we shall write $S_0 w = S_0 \up + Su$, where $S$
is an $\ell^2$-orthogonal operator, with dual $R = S^* =
S^{-1}$, \label{sym:S-R}
\begin{align*}
  Su(\ell) 
  := \cases{ u(\ell), & \ell_2 > \hat{x}_2, \\
  u(\ell-\burg_{12}), & \ell_2 < \hat{x}_2 }
  \quad 
   \text{and} \quad
   Ru(\ell) 
  := \cases{ u(\ell), & \ell_2 > \hat{x}_2, \\
  u(\ell+\burg_{12}), & \ell_2 < \hat{x}_2.  }
\end{align*}

We can now rigorously formulate the assumptions on the site energy potential: We
assume that $V \in C^k((\R^3)^\Rg)$, $k \geq 4$, where
$\Rg \subset \L \setminus \{0\}$ such that for each $u \in \Adm$, and
$w = \up + u$, the site energy associated with a lattice site $\ell$ is given by
$V(Dw(\ell))$, where $Dw(\ell) \equiv D_\Rg w(\ell)$. We assume again that
$V({\bm 0}) = 0$ (that is, $V$ is the energy difference from the reference
lattice) and that $\Rg, V$ are point symmetric \eqref{eq:pt:ptsymm}. We shall
assume throughout that $V$ is invariant under lattice slip, reformulating
\eqref{eq:slipinvariance_formal} as
\begin{equation}
  \label{eq:disl:slip_invariance}
  V\b(D(\up+u)(\ell)\b) 
  = V\b(R D S_0(\up+u)(\ell)\b)
  \qquad \forall u \in \Adm, \ell \in \L.
\end{equation}%
In addition, to guarantee lattice stability (both before and after shift) we
assume that not only $D$ but also $RDS$ include nearest-neighbour finite
differences (or equivalent):
\begin{equation}\label{eq:disl:nn}
  |u(\ell+\mA e_n) - u(\ell)| \leq |RDS u(\ell)|,
  \qquad \forall\, \ell\in\L, \quad  n\in\{1,2\}, \quad u : \L \to \R^3.
\end{equation}

The global energy (difference) functional is now defined by
\begin{equation}
  \label{eq:disl:Ediff}
  \E(u) := \sum_{\ell \in \L} \B(V\b( D\up(\ell)+Du(\ell) \b) -
  V\b(D\up(\ell)\b) \B) =: \sum_{\ell \in\L} V_\ell(Du(\ell)),
\end{equation}
where $V_\ell(\bfg) := V(D\up(\ell)+\bfg)-V(D\up(\ell))$.

\begin{lemma}
  \label{th:disl:extension_lemma}
  $\E : (\Usz \cap \Adm, \|\D \cdot \|_{L^2}) \to \R$ is
  continuous. In particular, there exists a unique continuous
  extension of $\E$ to $\Adm$, which we still denote by $\E$. The
  extended functional $\E \in C^k(\Adm)$ in the sense of Fr\'echet.
\end{lemma}
\begin{proof}[Idea of the proof]
  The main idea is the same as in the point defect case.  The proof
  that $\del\E(0) \in \UsHd$ is based on the construction of $\up$ in
  terms of the linear elasticity predictor $\ulin$, which guarantees
  that $\up$ is an ``approximate equilibrium'' in the far-field.
  See \cite{HudsonOrtner:disloc} for a similar proof applied in the simplified
  context of a screw dislocation. The complete proof (given in
  \S~\ref{sec:prf:E_disl}) for our general case requires a combination of the
  proof in \cite{HudsonOrtner:disloc} and the concept of elastic strain
  introduced in \S~\ref{sec:prf:elastic_strain}.
\end{proof}

The variational problem for the dislocation case is
\begin{equation}
  \label{eq:disl:atm}
  \ua \in \arg\min \b\{ \E(u) \bsep u \in \Adm \b\}.
\end{equation}
Since $\Adm$ is open, if a minimiser $\ua$ exists, then $\del\E(\ua) = 0$. We
call a minimiser strongly stable if, in addition, it satisfies the positivity
assumption \eqref{eq:pt:strongstabeq}.

\begin{remark}
  One can also formulate anti-plane models for pure screw dislocations by
  restricting $\Adm$ to displacements of the form $u = (0, 0, u_3)$ and also
  computing a predictor of the form $\ulin = (0, 0, (\ulin)_3)$. Note also that
  for pure screw dislocations, \eqref{eq:disl:defn_u0} is ignored. In the
  anti-plane case we may also choose $\Adm = \UsH$ since only slip-invariance in anti-plane direction is
    required, that is, the topology of the projected 2D lattice remains
    unchanged.

  To define in-plane models for pure edge dislocations one restricts
  $\Adm$ to displacements of the form $u = (u_1, u_2, 0)$. The
  predictor $\ulin$ does not simplify in this case.

  All our results carry over trivially to these simplified models.
\end{remark}

\begin{remark}
  \label{rem:reflection}
  The definition of the reference solution with branch-cut
  $\Gamma = \{ (x_1, \hat{x}_2) \sep x_1 \geq \hat{x}_1 \}$ was somewhat
  arbitrary, in that we could have equally chosen
  $\Gamma_S := \{ (x_1, \hat{x}_2) \sep x_1 \leq \hat{x}_1 \}$. In this case the
  predictor solution $\up$ would be replaced with $S_0 \up$. Let the resulting
  energy functional be denoted by
  \begin{displaymath}
    \E_S(v) := \sum_{\ell\in\L} V\b( DS_0\up(\ell)+Dv(\ell)\b) 
    - V\b( DS_0\up(\ell)\b).
  \end{displaymath}
  It is straightforward to see that, if $\del\E(\ua) = 0$, then
  $\del\E_S(S\ua) = 0$ as well. This observation means, that in certain
  arguments, an estimate on $\ua$ in the left half-space where no branch-cut is
  present immediately yields the corresponding estimate on $S\ua$ in the right
  half-space as well.
\end{remark}

\alert{
  \begin{remark}
    Another source of arbitrariness comes from the precise definition of the
    predictor $\up$, e.g., through the choice of the dislocation core position
    $\hat{x}$ or the choice of smearing function $\eta$. Indeed, more generally,
    arbitrary smooth modifications to $\up$ are allowed as long as they do not
    significantly change the far-field behaviour. While such changes to the
    predictor $\up$ affect the resulting corrector $\ua$ (the solution of
    \eqref{eq:disl:atm}, the {\em total displacement} $\up + \ua$ remains
    unchanged in the sense that, if $\up' = \up + w_0$ is a modified predictor,
    then $\ua' = \ua - w_0$ is again a solution of \eqref{eq:disl:atm}.
  \end{remark}
}

\subsection{Elastic strain}
\label{sec:prf:elastic_strain}
The transformation $\up \mapsto S_0\up$ produces a map that is smooth in
$\halfspace$, and which generates the same atomistic configuration. It is
therefore natural to define the {\em elastic strains}
\begin{equation}
  \label{eq:defn_elastic_strain}
 e(\ell) := (e_\rho(\ell))_{\rho \in \Rg} \quad \text{where} \quad
 e_\rho(\ell) 
 := \cases{
   R D_\rho S_0\up(\ell), & \ell \in \halfspace, \\
   D_\rho \up(\ell), & \text{otherwise.}
 }
\end{equation}
The analogous definition for the corrector displacement $u$ is
\begin{equation}
  \label{eq:defn_Dprime}
  \Del u(\ell) := 
  (\Del_\rho u(\ell))_{\rho \in \Rg} \quad \text{where}
  \quad
  \Del_\rho u(\ell) := \cases{
    R D_\rho Su(\ell), & \ell \in \halfspace, \\
    D_\rho u(\ell), & \text{otherwise.}
  }
\end{equation}

The slip invariance condition \eqref{eq:disl:slip_invariance} can now
be rewritten as
\begin{equation}
  \label{eq:disl:slip_invariance_D'_formulation}
  V\b(D(\up+u)(\ell)\b) = V\b(e(\ell)+\Del u(\ell)\b) \qquad
  \forall u \in \Adm, \ell \in \L.
\end{equation}
Linearity of $S$ and hence of $\Del$ implies
\begin{align}
  \label{eq:disl:slip_delV}
  \b\< \del V(D(\up+u)), Dv \b\> &= \b\< \del V(e+\Del u), \Del v \b\>, \\
  \label{eq:disl:slip_ddelV}
  \b\< \ddel V(D(\up+u)) Dv, Dw \b\> &= \b\< \ddel V(e+\Del u) \Del v, \Del w \b\>,
\end{align}
and so forth.

\subsection{Regularity}
\label{sec:dis:reg}
The regularity of the predictor $\up$ is already stated in Lemma
\ref{th:disl:up_lemma}. We now state the regularity of the corrector $\ua$.  It
is interesting to note that the regularity of the dislocation corrector $\ua$
is, up to log factors, identical to the regularity of the displacement field in
the point defect case, which indicates that the dislocation problem is
  computationally no more demanding than the point defect problem. Indeed,
  this will be confirmed in \S~\ref{sec:dis:clamped}.

\begin{theorem}
  \label{th:disl:regularity}
  Suppose that the lattice is stable \eqref{eq:pt:stab_lattice}.  Let
  $u \in \Adm$ be a critical point, $\del\E(u) = 0$, then there exist
  constants $C > 0, u_\infty \in \R^3$ such that, for $1 \leq j \leq
  k-2$ and for $|\ell|$ sufficiently large,
  \begin{equation}
    \label{eq:disl:regularity}
    |\Del^{j} u(\ell)| \leq C |\ell|^{-1-j} \log |\ell|
    \qquad \text{and} \qquad
    |u(\ell) - u_\infty| \leq C |\ell|^{-1} \log |\ell|.
  \end{equation}
\end{theorem}

\begin{remark}
  It can be immediately seen that the decay
  $|\Del u(\ell)| \lesssim |\ell|^{-2} \log |\ell|$ is equivalent to
  $|Du(\ell)| \lesssim |\ell|^{-2} \log |\ell|$. For higher-order derivatives,
  it is necessary to make a case distinction. While,
  $\Del^j u(\ell) = D^j u(\ell)$ at sufficient distance from $\Gamma$, ``close
  to'' the branchcut $\Gamma$ we could alternatively write
  $|D^j S u(\ell)| \lesssim |\ell|^{-1-j}$.

  In the pure screw case where $\burg_{12} = 0$ we simply have $D = \Del$.
\end{remark}

\subsection{Clamped boundary conditions}
\label{sec:dis:clamped}
To extend clamped boundary conditions to the dislocation problem, we prescribe
the displacement to be the predictor displacement outside some finite
computational domain $\Omega_R \subset \L$. Thus, we may think of these boundary
conditions as {\em asynchronous continuum linearised elasticity boundary
  conditions.}

This amounts to choosing a corrector displacement space analogous to
$\Us^0(\Omega_R)$ in the point defect case,
\begin{displaymath}
  \Adm^0(\Omega_R) := \b\{ v \in \Adm \bsep v = 0 \text{ in } \L
  \setminus \Omega_R \b\},
\end{displaymath}
and the associated finite-dimensional optimisation problem reads
\begin{equation}
  \label{eq:disl:clamped}
  u^0_R \in \arg\min \b\{ \E(u) \bsep u \in \Adm^0(\Omega_R) \b\}.
\end{equation}

\begin{theorem}
  \label{th:disl:clamped}
  Let $\ua$ be a strongly stable solution to \eqref{eq:disl:atm}, then
  there exist $C, R_0 > 0$ such that, for all $\Omega_R \subset \L$
  satisfying $B_R\cap\L \subset \Omega_R$ for some $R \geq R_0$, there
  exists a strongly stable solution $\ua^0_R$ of
  \eqref{eq:disl:clamped} satisfying
  \begin{equation}
    \label{eq:disl:clamped:errest}
    \| \D \ua - \D \ua^0_R \|_{L^2} \leq C R^{-1} \log(R) \qquad
    \text{and} \qquad \b| \E(\ua) - \E(\ua^0_R) \b| \leq C R^{-2}
    (\log R)^2.
  \end{equation}
\end{theorem}

\subsection{Periodic boundary conditions}
\label{sec:dis:per}
It is possible to extend periodic boundary conditions to the dislocation case by
considering a periodic array of dislocations with alternating signs. In practise
the computational domain then contains a dipole or a quadrupole. It then becomes
necessary to estimate image effects, for which our regularity results are still
useful, but which requires substantial additional work. Hence, we postpone the
analysis of periodic boundary conditions for dislocation to future work, but
refer to \cite{CaiBulChaLiYip2003} for an interesting discussion of these
issues.

\subsection{Boundary conditions from linear elasticity}
\label{sec:dis:lin}
We now extend the lattice linear elasticity boundary conditions to the
dislocation case.  The linearisation argument \eqref{eq:pt:Vlin}
should now be carried out for the full displacement $w = \up + u$, and
reads
\begin{displaymath}
  V(Dw) \approx V(0) + \< \del V(0), D w \> + \smfrac12 \< \ddel V(0)
  Dw, Dw \>,
\end{displaymath}
but this is invalid whenever the interaction neighbourhood
crosses the slip plane $\Gamma$.

Instead, we must first transform the finite difference stencils as follows:
recall the definition of $\halfspace$ from \eqref{eq:disl:defn_halfspace} and
the definition of elastic strain $e$ and $\Del u$ from
\eqref{eq:defn_elastic_strain} and \eqref{eq:defn_Dprime}, then we define
\begin{equation}
  \label{eq:defn_Del}
  \Del_0 w(\ell) = \Del_0(\up+u)(\ell) := e(\ell)+\Del u(\ell).
\end{equation}
According to Lemma \ref{th:disl:up_lemma} and Theorem \ref{th:disl:regularity},
if $u = \ua$, then $|\Del_0 w(\ell)| = O(|\ell|^{-1})$, hence we may linearize
with respect to this transformed finite different stencil. Using the slip
invariance condition~\eqref{eq:disl:slip_invariance}, we obtain
\begin{displaymath}
  V(Dw) = V(\Del_0  w) = V(0) + \< \del V(0), \Del_0 w \> + \smfrac12 \< \ddel V(0)
  \Del_0 w, \Del_0 w \> + O(|\Del_0 w|^3),
\end{displaymath}
and we therefore define the energy difference functional
\begin{align*}
  \E^\lin_R(u) &:=
  \sum_{\ell \in \Omega_R}
       V_\ell(Du(\ell))
  +
  \sum_{\ell \in \L \setminus \Omega_R} 
    \B( V^\lin\b(e(\ell)  + \Del u(\ell)\b) 
        - V^\lin\b(e(\ell)\b)\B)
  ,
\end{align*}
where $V^\lin$ is the same as in the point defect case,
\begin{align*}
   V^\lin(\bfg) &:= V(\bfO) + \b\< \del V(\bfO),
    \bfg \b\> + \smfrac{1}{2} \b\< \ddel V(\bfO) \bfg, \bfg \b\>
\end{align*}
and where $\Omega_R \subset \L$ is the ``inner'' computational domain.
It follows from minor modifications of the proof of
Lemma \ref{th:disl:extension_lemma} that $\E^\lin_R$ can be extended by
continuity to a functional $\E^\lin_R \in C^k(\Adm)$.

Thus, we aim to compute
\begin{equation}
  \label{eq:dis:lin}
  \ulin_R \in \arg\min \b\{ \E^\lin_R(u) \bsep u \in \Adm \b\}.
\end{equation}

\begin{theorem}
  \label{th:dis:lin}
  Let $\ua$ be a strongly stable solution to \eqref{eq:disl:atm}, then
  there exist $C, R_0 > 0$ such that for all domains $\Omega_R \subset
  \L$ with $B_R \cap \L \subset \Omega_R$ and $R \geq R_0$, there exists a
  strongly stable solution of \eqref{eq:dis:lin} satisfying
  \begin{equation}
    \label{eq:dis:lin:errest}
     \b\|\D\ua - \D\ulin_R \b\|_{L^2} \leq C R^{-1} \qquad
     \text{and} \qquad  \b|\E^\lin_R(\ulin_R) - \E(\ua)\b| \leq C R^{-2}
     \log R.
  \end{equation}
\end{theorem}

\begin{proof}[Idea of proof]
  The proof is similar to the point defect case, the main additional step to
  take into account being that the linearisation is with respect to the full
  displacement $\up + \ua$. Since $\D\up \sim |x|^{-1}$ it therefore follows
  that the linearisation error at site $\ell$ is only of order $O(|\ell|^{-2})$,
  while in the point defect case it was of order $O(|\ell|^{-2d})$. This
  accounts for the reduced convergence rate.
\end{proof}

\begin{remark}\label{rem_lin_elast}
  1. The key difference between the schemes \eqref{eq:disl:clamped} and
  \eqref{eq:dis:lin} is that the former employs a {\em precomputed continuum
    linear elasticity boundary condition} while the latter computes a {\em
    lattice linear elasticity boundary condition on the fly}.
  %
  It is therefore interesting to note that, for dislocations, solving the
  relatively complex exterior problem yields almost no qualitative improvement
  over the basic Dirichlet scheme \eqref{eq:disl:clamped}. Indeed, if the cost
  of solving the exterior problem is taken into account as well, then the scheme
  \eqref{eq:dis:lin} may in practice become more expensive than
  \eqref{eq:disl:clamped}.

  The main advantage of \eqref{eq:dis:lin} appears to be that the boundary
  condition need not be computed beforehand, but could be computed ``on the
  fly''. We speculate that this can give a substantially improved prefactor when
  the dislocation core is spread out, e.g., in the case of partials.

  2. If, instead of linearising about the homogeneous lattice configuration we
  were to linearise about the predictor $\up$, then the rate of convergence for
  dislocations would become the same (up to log factors) as for point
  defects. However, since lattice Green's function and similar techniques are no
  longer available we cannot conceive of an efficient implementation of such a
  scheme without reverting again to complex atomistic/continuum type
  coarse-graining techniques.
\end{remark}

\subsection{Boundary conditions from nonlinear elasticity for screw
  dislocations}
\label{sec:dis:ac}
The formulation of a/c coupling methods for general dislocations is not
straightforward. We therefore consider only the case of {\em pure screw
  dislocations} and postpone the general case to future work. Thus, we assume
that $\burg = e_3$, and in this case, only the invariance of $V$ in the normal
direction is relevant:
\begin{displaymath}
  V\b( \bfg + \bfh e_3 \b) = V(\bfg) \qquad \forall \bfg \in (\R^3)^\Rg,
  \bfh \in \Z^\Rg.
\end{displaymath}

We set up the computational domain and approximation space as in
\S~\ref{sec:pt:ac}. To define the energy functional, we first
construct a modified interpolant that takes into account the
discontinuity of the full displacement across the slip plane,
similarly to the elastic strain used in \S~\ref{sec:dis:lin},
\begin{displaymath}
  I_R^{\rm el} u(x) := \cases{
    I_R u(x), & x \in T, T \cap \Gamma = \emptyset, \\
    I_R (u+ \burg \chi_{x_2 < \hat{x}_2})(x), & x \in T, T \cap \Gamma
    \neq \emptyset,
  }
  ,
\end{displaymath}
where $I_R$ is the nodal interpolation with respect to $\T_R$.
With this definition, the energy difference functional is given by
\begin{align}
  \label{eq:ac:defn_E}
  \E^\ac_R(u) &:= \sum_{\ell \in \Omega^\a_R} V_\ell(Du(\ell))
  + \sum_{\ell \in \Omega^\i_R} V_\ell^\i(Du(\ell))  \\
  \notag
  &\qquad \qquad + \sum_{T \in \T_R} v_T^{\rm eff} \B( W(\D
  I^{\rm el}_R (\up+u)) -
  W(\D I^{\rm el}_R \up) \B),
\end{align}
where $V_\ell^\i, W, v_T^{\rm eff}$ are defined as in
\S~\ref{sec:pt:ac}.

We seek to compute
\begin{align}
  \label{eq:dis:ac}
  &u^\ac_R \in \arg\min \b\{ \E^\ac_R(u) \bsep u \in \Adm(\T_R) \b\}, \qquad
  \text{where} \\
  \notag
  &\Adm(\T_R) := \Adm \cap \Us^0(\T_R).
\end{align}
We again let $R$ and $R_\c$ be the sizes of $\Omega^\a_R$ and $\omega_R$,
\begin{equation}\label{eq:ac:dis:radii}
	B_R \cap \L \subset \Omega^\a_R
\qquad\text{and}\qquad
	B_{R_\c}
    \subset
    \omega_R
    \subset
    B_{c_0 R_\c}
   .
\end{equation}
and assume that there exists $\eta>0$ and $c_1>0$ such that there exists a strongly stable solution $\uac_R$ to \eqref{eq:dis:ac} satisfying
\begin{equation}
  \label{eq:ac:dis:est}
  \|\D \uac_R - \D \ua\| \leq c_1 \b(\| h
  \Del^2 (\up+\ua) \|_{\ell^2(\L\cap(\omega_R\setminus B_R))}
  +
  \|\Del \ua\|_{\ell^2(\L \setminus B_{R_\c/2})}\b)
  ,
\end{equation}
provided that $\|h \Del^2(\up+\ua) \|_{\ell^2(\L\cap(\omega_R\setminus B_R))}
  +
  \|\Del \ua\|_{\ell^2(\L \setminus B_{R_\c/2})} \leq \eta$.

\begin{proposition}
  \label{th:dis:ac}
  Let $\ua$ be a strongly stable solution of \eqref{eq:disl:atm} and assume that
  \eqref{eq:ac:dis:radii} and \eqref{eq:ac:dis:est} hold.  Further we require
  that $\omega_R$ and $\T_R$ satisfy the following quasi-optimality conditions:
  \begin{align}
    \label{eq:dis:restriction_Rc+mesh}
    &c_2 R^{p} \leq R_\c \leq c_3 R^{p}, \text{ for some $p > 0$, \qquad and}
    \qquad
    |h(x)| \leq c_4 \smfrac{|x|}{R}.
  \end{align}
  Then there exist $R_0, C$ depending on $\eta$, $c_2$, $c_3$, $c_4$, and $p$,
  such that for all $R \geq R_0$ there exists a strongly stable solution
  $\uac_R$ to \eqref{eq:dis:ac} satisfying
  \begin{align}
    \label{eq:dis:ac:rate}
    \b\| \D\uac_R - \D\ua \b\|_{L^2} &\leq C R^{-1}.
  \end{align}
\end{proposition}


\subsection{Numerical results}
\label{sec:dis:num}

\subsubsection{Setup}
We consider the anti-plane deformation model of a screw dislocation in a BCC
crystal from \cite{HudsonOrtner:disloc}, the main difference being that we admit
nearest neighbour many-body interactions instead of only pair
interactions. Thus, we only give a brief outline of the model setup. The choice
of dislocation type is motivated by the fact that the linearised elasticity
solution is readily available.

Briefly, let $\mB \Z^3 = \Z^3 \cup (\Z^3 + (1/2,1/2,1/2)^T)$ denote a BCC
crystal, then both the dislocation core and Burgers vector point in the
$(1,1,1)^T$ direction. Upon rotating and possibly dilating, the projection $\mA
\Z^2$ of the BCC crystal is a triangular lattice, hence we again assume
\eqref{eq:num:pt:A}. The linear elasticity predictor is now given by $u^\lin(x) = \smfrac{1}{2\pi} {\rm arg}(x - \hat{x})$, where we assumed that the
Burgers vector is $b = (0, 0, 1)^T$ and $\hat{x}$ is the centre of the
dislocation core. We shall slightly generalise this, by admitting
\begin{displaymath}
  u^\lin(x) = \mF \cdot (x - \hat{x}) + \smfrac{1}{2\pi} {\rm arg}(x - \hat{x}),
\end{displaymath}
which is equivalent to applying a shear deformation of the form
\begin{displaymath}
  \left( \begin{array}{ccc} 1 & 0 & 0 \\ 0 & 1 & 0 \\ \mF_1 & \mF_2 &
      1 \end{array} \right)
\end{displaymath}
to the rotated BCC crystal and is thus still included within our framework
through a modification of the potential $V$.

Let the unknown for the anti-plane model, the displacement in $e_3$
direction, be denoted by $z(\ell) := y_3(\ell)$, then we use the
EAM-type site potential
\begin{align*}
  & V(Dy(\ell)) = V^{\rm anti}(Dz(\ell)) = \sum_{\rho \in \Nhd_\ell} \phi\b(|D_\rho z(\ell)|\b) +
  G\B( \sum_{\rho \in \Nhd_\ell} \psi\b(|D_\rho z(\ell)|\b) \B) \\
  & \text{with } \phi(r) = \psi(r) = \sin^2(\pi r) \quad \text{ and }
  G(s) = \smfrac12 s^2.
\end{align*}
The $1$-periodicity of $\phi, \psi$ emulates the fact that displacing
a line of atoms by a full Burgers vector leaves the energy invariant.

We apply again the remaining remarks in \S~\ref{sec:num:pt:setup}.

\subsubsection{Discussion of results}
\label{sec:num:dis:results}
We consider three numerical experiments:
\begin{enumerate}
\item $\mF = (0, 0)^T, x_0 = (1/3, 1/(2\sqrt{3}))^T$: \\[1mm]
  $ $ \quad The results are shown in Figure \ref{fig:screw1_err}. We
  observe precisely the predicted rates of convergence. However, it is
  worth noting that although the asymptotic rates for ATM, LIN and AC
  are identical (up to log-factors), the prefactor varies by an order
  of magnitude.

  The ``dip'' in the energy error for the LIN method is likely due to
  a change in sign of the error.

\item $\mF = (0, 0)^T, x_0 = \smfrac12 (1, 1/\sqrt{3})^T$ is the
  centre of a triangle: \\[1mm]
  $ $ \quad The results are shown in Figure \ref{fig:screw2_err}. In
  this case, the AC method exhibits the predicted convergence rate,
  while both the ATM and LIN methods show subtantially better
  rates. The explanation for ATM (but not for LIN) is that the
  solution displacement $\ua$ has three-fold symmetry, from which one
  can {\em formally} deduce the improved decay estimate $|D\ua(\ell)|
  \leq C |\ell|^{-4}$. This readily implies the observed rate.

  This test demonstrates that, in general, our estimates are only
  upper bounds, but that in special circumstances (e.g., additional
  symmetries), better rates can be obtained. It is moreover
  interesting to note that the most basic scheme, ATM, is the most
  accurate with this setup.

\item $\mF = (10^{-3}, 3 \times 10^{-4})^T, x_0 = \smfrac12 (1,
  1/\sqrt{3})^T$: \\[1mm]
  $ $ \quad The results are shown in Figure \ref{fig:screw3_err}. In this final
  test, we chose $\mF$ to push the dislocation core close to instability. We
  included this test to demonstrate that one cannot always expect the clean
  convergence rates displayed in the point defect tests, or in the first screw
  dislocation test, but that there may be significant pre-asymptotic regimes.
\end{enumerate}

\begin{figure}
  \begin{center}
    \includegraphics[height=5cm]{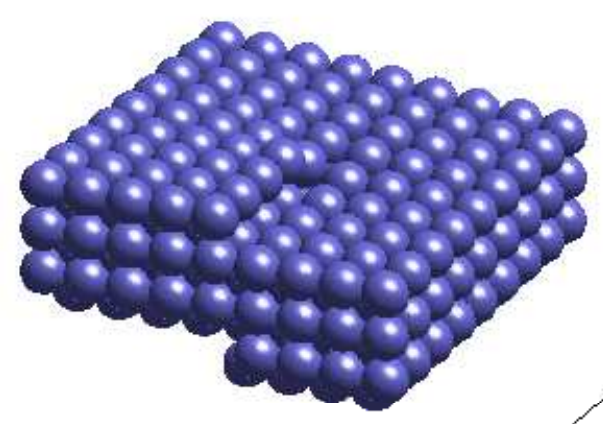}
  \end{center}
  \caption{\label{fig:screw} Illustration of a screw dislocation
    configuration in a BCC crystal.}
\end{figure}

\begin{figure}
  \begin{center}
    \includegraphics[height=6.8cm]{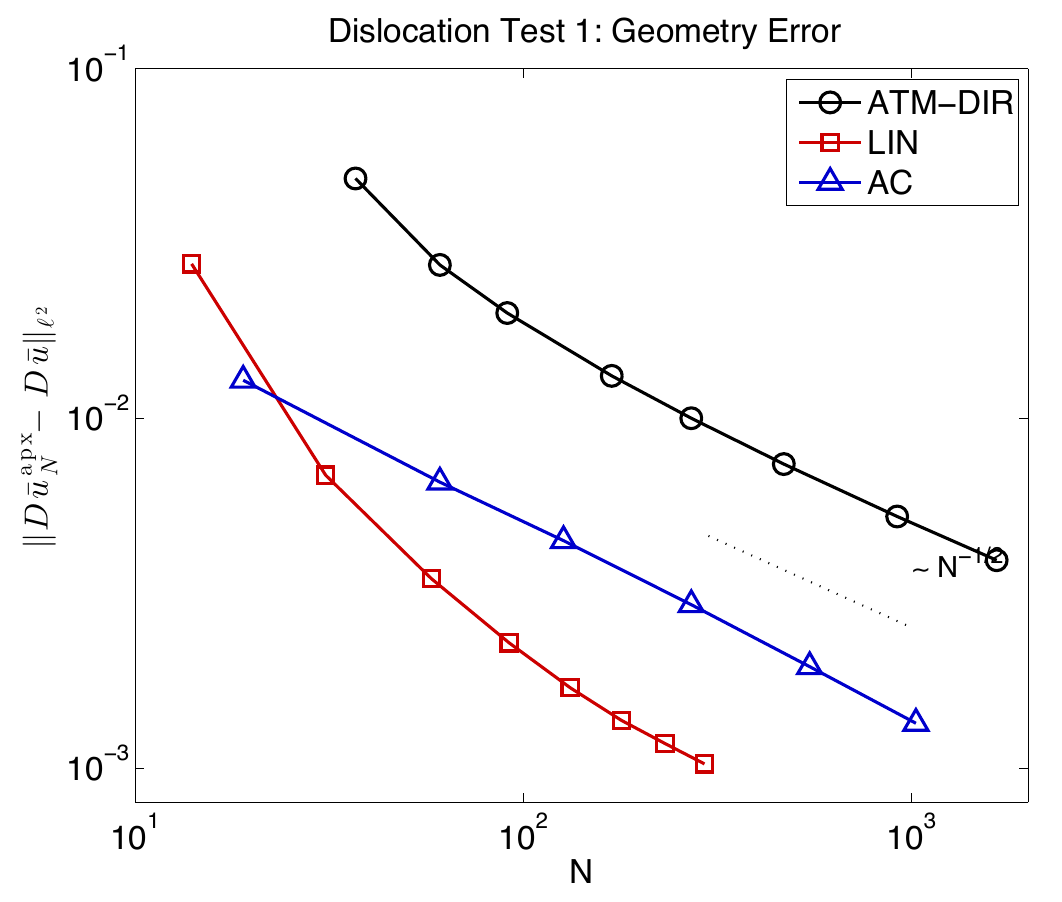}
    \includegraphics[height=6.8cm]{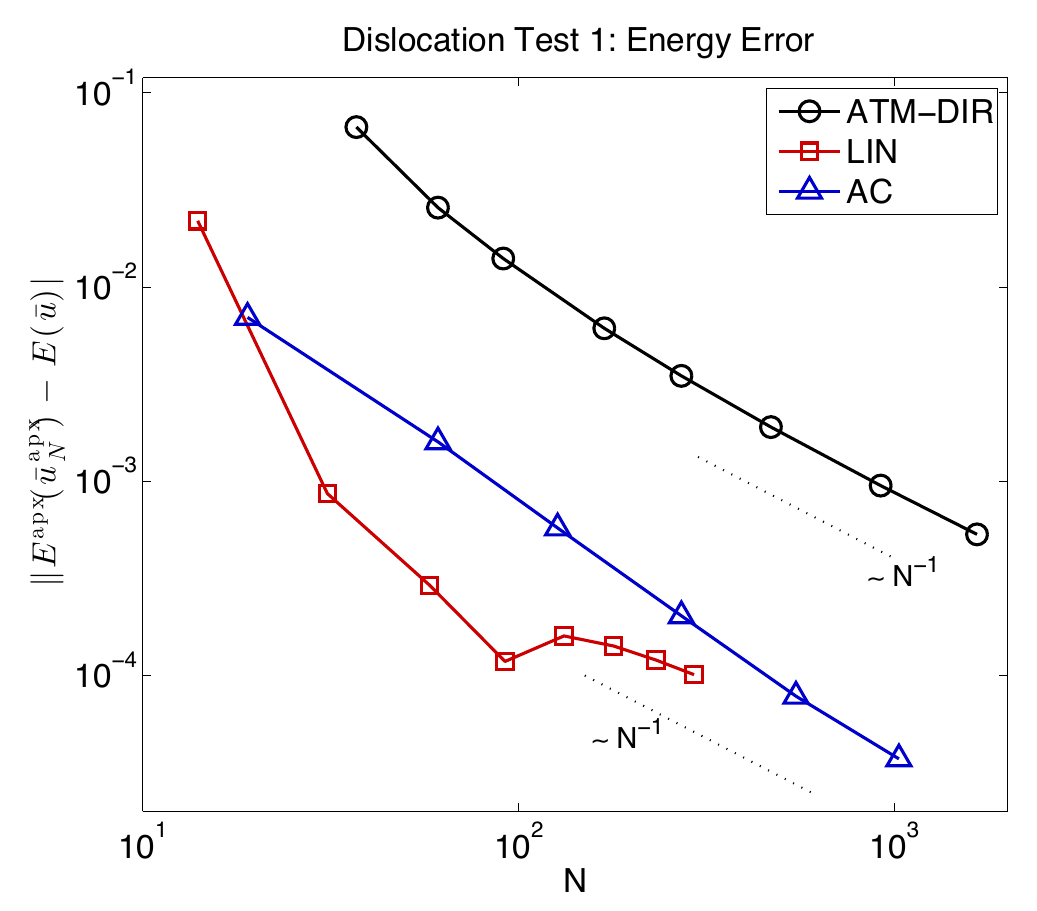} \\
    \footnotesize (a) \hspace{7cm} (b)
  \end{center}
  \caption{\label{fig:screw1_err}Rates of convergence, in the {\em
      first} dislocation test, of the ATM-DIR, LIN and AC methods. $N$
    denotes the number of atoms in the {\em inner computational
      domain}; see \S~\ref{sec:num:pt:setup} for definitions.}
\end{figure}

\begin{figure}
  \begin{center}
    \includegraphics[height=6.8cm]{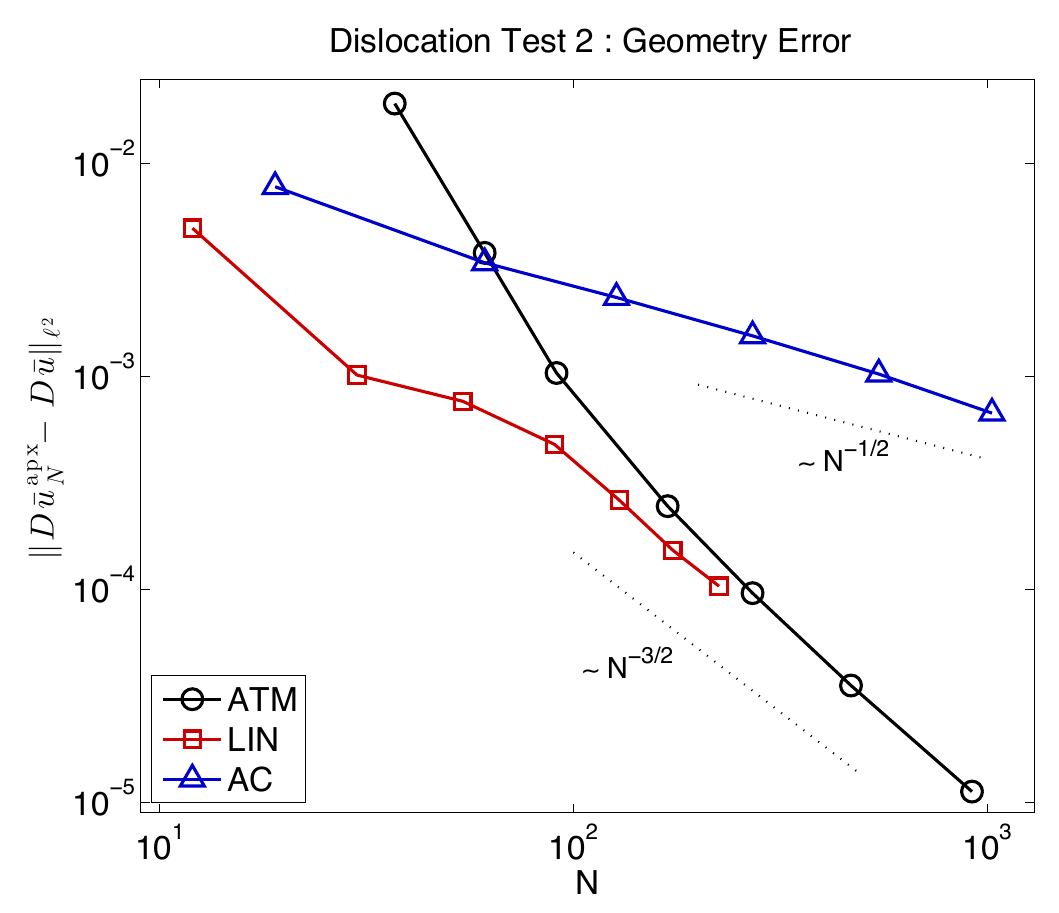}
    \includegraphics[height=6.8cm]{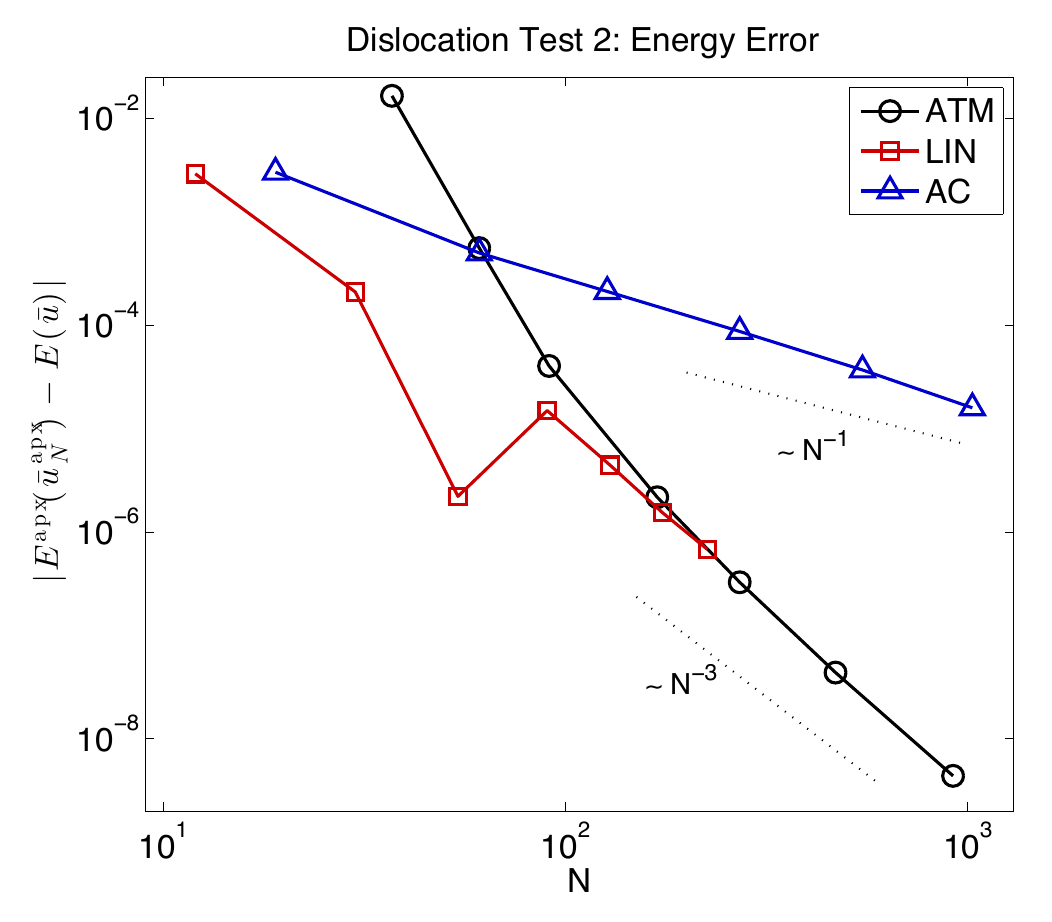} \\
    \footnotesize (a) \hspace{7cm} (b)
  \end{center}
  \caption{\label{fig:screw2_err}Rates of convergence, in the {\em
      second} dislocation test, of the ATM-DIR, LIN and AC
    methods. $N$ denotes the number of atoms in the {\em inner
      computational domain}; see \S~\ref{sec:num:pt:setup} for
    definitions.}
\end{figure}

\begin{figure}
  \begin{center}
    \includegraphics[height=6.8cm]{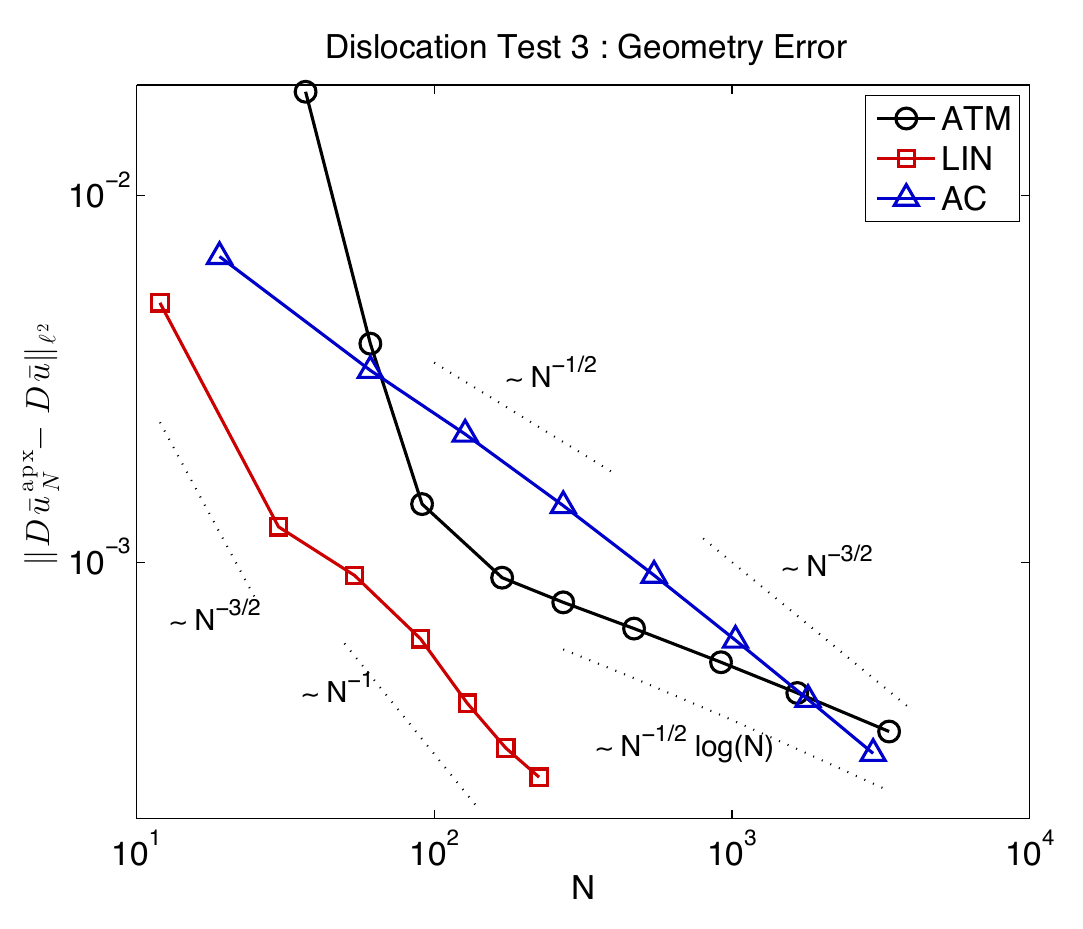}
    \includegraphics[height=6.8cm]{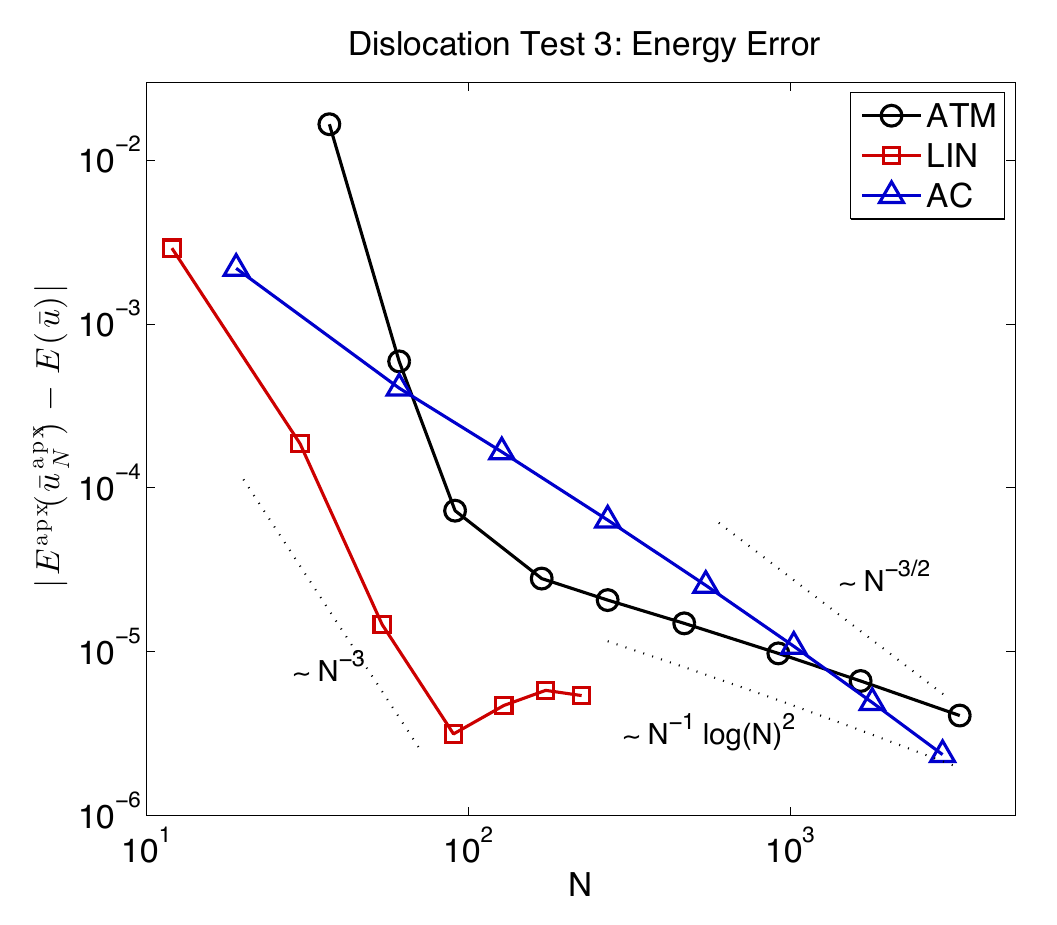} \\
    \footnotesize (a) \hspace{7cm} (b)
  \end{center}
  \caption{\label{fig:screw3_err}Rates of convergence, in the {\em
      third} dislocation test, of the ATM-DIR, LIN and AC methods. $N$
    denotes the number of atoms in the {\em inner computational
      domain}; see \S~\ref{sec:num:pt:setup} for definitions.}
\end{figure}

\section{Conclusion}
\label{sec:discussion}
We have introduced a flexible analytical framework to study the effect of
embedding a defect in an infinite crystalline environment. Our main analytical
results are (1) the formulation of equilibration as a variational problem in a
discrete energy space; and (2) a qualitatively sharp regularity theory for
minimisers.

These results are generally useful for the analysis of crystalline defects,
however, our own primary motivation was to provide a foundation for the analysis
of atomistic multi-scale simulation methods, which in this context can be
thought of as different means to produce boundary conditions for an atomistic
defect core simulation. To demonstrate the applicability of our framework we
analyzed simple variants of some of the most commonly employed schemes:
Dirichlet boundary conditions, periodic boundary conditions, far-field
approximation via linearised lattice elasticity and via nonlinear continuum
elasticity (Cauchy--Born, atomistic-to-continuum coupling).
%
%
In parallel works \cite{OrtnerZhang2014,LiOrtnerShapeevEtAl-blended,
  2013-PRE-bqcfcomp,ChenOrtner-qm_mm} this framework has
already been exploited resulting in new and improved formulations of
atomistic/continuum and quantum/atomistic coupling schemes.

There are numerous practical and theoretical questions that we have left open in
the present work, some of which we commented on throughout the article. Possibly
the key ``bottleneck'' in our analysis is that it only provides a rate of
convergence, i.e.,
\begin{displaymath}
  \text{error} \leq C N^{-r},
\end{displaymath}
where $N$ is the number of unknowns in the approximate problem, however, we have
not been able to provide estimates on the prefactor. We speculate that such
estimates may not be obtained {\it a priori} but only {\it a posteriori}, as
  it requires considerably more detailed information about a defects structure
  and stability than one would normally assume {\it a priori}.

\section{Proofs: The Energy Difference Functionals}
\label{sec:prf_E}
This section is concerned with proofs for Lemma \ref{th:extension_lemma} and
Lemma \ref{th:disl:extension_lemma} which state that the energy $\E$ can be
understood as a smooth functional on the energy space, i.e., $\E \in C^k(\UsH)$
in the point defect case and $\E \in C^k(\Adm)$ in the dislocation case.

\subsection{Conversion to divergence form}
We begin by establishing an auxiliary result that allows us to convert pointwise
forces into divergence form without sacrificing fundamental decay
properties.

\begin{lemma}
  \label{th:divg_f}
  Let $d \in \N$, $p > d \geq 2$ and $f : \Z^d \to \R$ such that
  $|f(\ell)| \leq C_f |\ell|^{-p}$ for all $\ell \in \Z^d$. Suppose,
  in addition, that $\sum_{\ell \in \Z^d} f(\ell) = 0$. Then, there
  exists $g : \Z^d \to \R^d$ and a constant $C$ depending only on $p$
  and $d$ such that
  \begin{equation}
    \label{eq:divg_f}
    \sum_{j = 1}^d D_{e_j} g_j(\ell) = f(\ell) \quad \text{and}
    \quad |g(\ell)| \leq CC_f |\ell|^{-p+1} \quad \text{for all
    } \ell \in \Z^d.
  \end{equation}
  If $f$ has compact support, then $g$ can be chosen to have compact
  support as well.
\end{lemma}
\begin{proof}
  Denote $\bar\ell := (\ell_1, \dots, \ell_{d-1})^T$.  We define the
  operator $\mathcal{C}_d (f, g) := (\tilde{f}, \tilde{g})$, where
  $\tilde{g} := g + \Delta\tilde{g} \, e_d$,
  \begin{align*}
    \Delta\tilde{g}(\bar\ell,\ell_d) :=
    \cases{ \sum_{\lambda =
        \ell_d}^{3\ell_d-2} f(\bar\ell, \lambda), & \ell_d \geq 1, \\
      - \sum_{\lambda = 3 \ell_d-1}^{\ell_d-1} f(\bar\ell, \lambda), &
      \ell_d \leq 0, } \quad \text{and} \quad \tilde{f}(\bar\ell,\ell_d)
    := \sum_{\lambda=3\ell_d-1}^{3\ell_d+1} f(\bar\ell,\lambda).
  \end{align*}
  One can then readily verify that
  \begin{equation} \label{eq:divg_f:12}
    D_{e_d} g_d(\ell) - f(\ell) = D_{e_d} \tilde{g}_d(\ell) -
    \tilde{f}(\ell) \qquad \forall \ell \in \Z^d.
  \end{equation}
  Moreover it is easy to see from the definition that $\sum_{\ell \in
      \Z^d} \tilde{f}(\ell) = \sum_{\ell \in
          \Z^d} f(\ell) = 0$.

  Let the operators $\mathcal{C}_1, \dots, \mathcal{C}_{d-1}$ be
  defined analogously and let $\mathcal{C}$ be their composition
  $\mathcal{C} := \mathcal{C}_1 \circ \dots \circ \mathcal{C}_d$. If $(f^+,
  g^+) = \mathcal{C}(f, g)$, then from \eqref{eq:divg_f:12} we obtain that
  \begin{equation}
    \label{eq:divg_f:15}
    f^+(\ell) - \sum_{j = 1}^d D_{e_j} g^+_j(\ell) = f(\ell) - \sum_{j =
      1}^d D_{e_j} g_j(\ell).
  \end{equation}

  Define the seminorm $[f]_q := \sup_{\ell\in \Z^d\setminus \{0\}} (|\ell|_\infty-\smfrac12)^{q} |f(\ell)|$, and a norm $\llbracket g\rrbracket_q := \sup_{\ell\in \Z^d} (|\ell|_\infty+\smfrac12)^{q} |g(\ell)|$. We
  claim that, if $(f^+, g^+) = \mathcal{C}(f, g)$, then
  \begin{equation}
    \label{eq:divg_f:20}
    [f^+]_p \leq 3^{d-p} [f]_p \quad \text{and} \quad
    \llbracket g^+ - g\rrbracket_{p-1} \lesssim [f]_p,
  \end{equation}
  where $\lesssim$ denotes comparison up to a multiplicative constant that may only depend on $p$ and $d$.
  Suppose that we have established \eqref{eq:divg_f:20}. We define
  \begin{displaymath}
    f^{(0)} := f, \quad g^{(0)} := 0, \quad \text{and} \quad
    (f^{(n+1)}, g^{(n+1)}) := \mathcal{C}(f^{(n)}, g^{(n)})
    \quad \text{for all $n\in\Z_+$}
    .
  \end{displaymath}
  Since $p > d$, we obtain that $[f^{(n)}]_p \to 0$. Moreover, since
  $\sum_{\ell\in\Z^d} f^{(n)}(\ell) = 0$ for all $n$ it follows that $\| f^{(n)}
  \|_{\ell^1} \to 0$. Further, \eqref{eq:divg_f:20} implies
  \begin{displaymath}
    \llbracket g^{(n+1)} - g^{(n)}\rrbracket_{p-1} \lesssim [f^{(n)}]_p \leq
    3^{n(d-p)} [f]_p,
  \end{displaymath}
  and hence the series $\sum_{n = 0}^\infty g^{(n+1)} - g^{(n)}$
  converges. Let
  $g(\ell) := \lim_{n\to\infty} g^{(n)}(\ell)$, then
  \eqref{eq:divg_f:15} implies that $g$ satisfies the identity in
  \eqref{eq:divg_f}, and the bound on $\llbracket g\rrbracket_{p-1}$ implies
  the inequality in \eqref{eq:divg_f}.
  It remains to note that if $f=f(\ell)=0$ outside the region $|\ell|_\infty \leq L$ for some $L$, then $f^{(n)}$, $g^{(n)}$, and hence $g$, are also zero outside this region.


To show the first inequality in \eqref{eq:divg_f:20}, we fix $\ell\ne 0$, express $f^+(\ell)$ through $f(\ell)$, and estimate
\begin{align*}
	|f^+(\ell)|
=~&
	\Bigg|\sum_{\substack{\lambda\in\Z^d \\ |\lambda-3\ell|_{\infty} \leq 1}} f(\lambda)\Bigg|
 \leq ~
	\sum_{\substack{\lambda\in\Z^d \\ |\lambda-3\ell|_{\infty} \leq 1}} \big(|\lambda|_{\infty}-\smfrac12\big)^{-p} [f]_p
\\ \leq ~&
	\sum_{\substack{\lambda\in\Z^d \\ |\lambda-3\ell|_{\infty} \leq 1}} \big(|3\ell|_{\infty}-1-\smfrac12\big)^{-p}[f]_p
 =
	3^d \, 3^{-p} \big(|\ell|_{\infty}-\smfrac12\big)^{-p} [f]_p
.
\end{align*}

The second inequality in \eqref{eq:divg_f:20} is based on the following two estimates:
  \begin{displaymath}
    |\tilde{f}(\ell)| \leq 3 (|\ell|_\infty-\smfrac12)^{-p} [f]_p \qquad \text{and} \qquad
    |\Delta\tilde{g}(\ell)| \lesssim (|\ell|_\infty+\smfrac12)^{-p+1} [f]_p,
  \end{displaymath}
where we denote again $(\tilde{f}, \tilde{g}) := \mathcal{C}_d(f, g)$ and $\Delta\tilde{g} := (\tilde{g} - g)\cdot e_d$.
The first estimate follows from arguments similar to the above.
The second estimate, for $\ell=(\bar{\ell}, \ell_d)$ with $\ell_d \leq 0$, is proved in the following calculation:
\begin{align*}
	|\Delta\tilde{g}(\ell)|
\leq~&
	\sum_{\lambda=3 \ell_d-1}^{\ell_d-1} |f(\bar{\ell}, \lambda)|
 \leq ~
	[f]_p \sum_{\lambda=3 \ell_d-1}^{\ell_d-1} (|(\bar{\ell}, \lambda)|_\infty-\smfrac12)^{-p}
\\ \leq ~&
	[f]_p \sum_{\lambda=3 \ell_d-1}^{\ell_d-1} (|(\bar{\ell}, \ell_d-1)|_\infty-\smfrac12)^{-p}
 \leq ~
	[f]_p |2\ell_d-1| \, (\smfrac13 (|(\bar{\ell}, \ell_d-1)|_\infty+\smfrac12))^{-p}
\\ \leq ~&
	[f]_p \, |2\ell_d+1| \, \smfrac1{3^p} (|\ell|_\infty+\smfrac12)^{-p}
 \leq ~
	[f]_p \, \smfrac2{3^p} (|\ell|_\infty+\smfrac12)^{-p+1},
\end{align*}
where we used that for $\ell_d\leq 0$, $|(\bar{\ell}, \ell_d-1)|_\infty \geq 1$ and the fact that $x-\smfrac12 \geq \smfrac13 (x+\smfrac12)$ for any $x\geq 1$.
For $\ell_d>0$ this estimate is obtained in a similar way.

The analogous estimates hold for applications of $\mathcal{C}_{d-1},
  \dots, \mathcal{C}_1$ and combining these yields the second inequality in
  \eqref{eq:divg_f:20}.
\end{proof}

\begin{corollary}\label{th:Dg_f} Let $p > d$ ($d\in\{2,3\}$), and
  $f : \mA\Z^d \to \R$ such that $|f(\ell)| \leq C_f |\ell|^{-p}$ for all
  $\ell \in \mA\Z^d$, and $\sum_{\ell \in \mA\Z^d} f(\ell) = 0$.  Then under the
  assumptions of \S~\ref{sec:pt:atm}, there exists $g : \mA\Z^d \to \R^{\Rg}$
  and a constant $C$ depending only on $p$ such that
  \begin{align*}
    \sum_{\ell\in\mA\Z^d} f(\ell) v(\ell)
    &=
      \sum_{\ell\in\mA\Z^d} \<g(\ell), Dv(\ell)\>
      \quad |g(\ell)| \leq CC_f |\ell|^{-p+1} \quad \text{for all
      } \ell \in \mA\Z^d.
  \end{align*}   
  In addition, if $d=2$, under the assumptions of \S~\ref{sec:dis:atm}, there
  exists $\tilde{g} : \mA\Z^2 \to \R^{\Rg}$ such that
  \begin{align*}
    \sum_{\ell\in\mA\Z^2} f(\ell) v(\ell)
    &=
    \sum_{\ell\in\mA\Z^2} \<\tilde{g}(\ell), \tilde{D}v(\ell)\>
    \quad |\tilde{g}(\ell)| \leq CC_f |\ell|^{-p+1} \quad \text{for all
    } \ell \in \mA\Z^2.
  \end{align*}
  If $f$ has compact support, then $g$ and $\tilde{g}$ can be chosen to have compact support as well.
\end{corollary}
\begin{proof}
One only needs to notice that the assumptions that the operators $D$ and $\tilde{D}$ contain nearest-neighbor finite differences (cf.\ \eqref{eq:pt:nn} and \eqref{eq:disl:nn}) allow to use Lemma \ref{th:divg_f} to construct the needed $g$ and $\tilde{g}$.
\end{proof}

\subsection{Proof of Lemma \ref{th:extension_lemma}}
\label{sec:prf_extension_pt}
The proof relies on two prerequisites.

\begin{lemma}
  Under the conditions of Lemma \ref{th:extension_lemma},
  \begin{displaymath}
    \Equad(u) := \sum_{\ell \in \L} \B( V_\ell(Du(\ell)) - \< \del
    V_\ell(\bfO), Du(\ell) \> \B)
  \end{displaymath}
  is well-defined for any $u \in \UsH$, and $\Equad \in C^k(\UsH)$.
\end{lemma}
\begin{proof}
  For a very similar argument that can be followed almost verbatim see
  \cite{OrtnerTheil2012}, hence we only give a brief idea of the proof.

  Since $|Du(\ell)| \in \ell^2(\L)$
  implies $|Du(\ell)| \in \ell^\infty$ and since $V_\ell \equiv V$ for $|\ell|
  \geq R_0$, we obtain that $\|\del^2 V_\ell(t Du(\ell))\| \leq C$, where $C$ is
  independent of $t \in [0, 1]$, and $\ell$. It follows that
  \begin{displaymath}
    \b| V_\ell(Du(\ell)) - \< \del
    V_\ell(\bfO), Du(\ell) \> \b| \leq C_u | Du(\ell) |^2,
  \end{displaymath}
  where $C_u$ depends only on $\|\,|Du|\, \|_{\ell^\infty}$. In
  particular, $\ell \mapsto V_\ell(Du(\ell)) - \< \del V_\ell(\bfO),
  Du(\ell) \> \in \ell^1(\L)$, and hence $\Equad(u)$ is well-defined.

  Using similar lines of argument one can prove that $\Equad \in
  C^k(\Adm)$. 
\end{proof}

\begin{lemma}
  Under the conditions of Lemma \ref{th:extension_lemma}, $\del \E(0)
  \in \UsHd$.
\end{lemma}
\begin{proof}
  Let $v \in \Usz$, then we can write the first variation in the form
  \begin{displaymath}
    \< \del \E(0), v \> = \sum_{\ell \in \L} \< \del V_\ell(\bfO),
    Dv(\ell) \> = \sum_{\ell \in \L} f(\ell) \cdot v(\ell).
  \end{displaymath}
  where $f(\ell)$ is given in terms of the $V_{\ell,\rho}$; the
  precise form is unimportant. Point symmetry of the lattice implies
  that $f(\ell) = 0$ for $|\ell| > \Rcore+\rcut$. Since $\E$ is
  translation invariant ($\E(u+c) = \E(u)$ for $c(\ell) = c \in \R$),
  it follows that $\sum_{\ell\in\L} f(\ell) = 0$. Therefore,
    \begin{displaymath}
      \b|\< f, u \>\b| = \b|\< f, u - u(0) \> \b| \leq 
      \|f \|_{\ell^2} \| u - u(0) \|_{\ell^2(\L \cap B_{\Rcore+\rcut})} \leq
      C \|f \|_{\ell^2} \|\nabla u \|_{L^2(B_\Rcore+\rcut)},
    \end{displaymath}
    where the inequality
    $\| u - u(0) \|_{\ell^2(\L \cap B_{\Rcore+\rcut})} \leq \| \nabla u
    \|_{L^2(B_\Rcore+\rcut)}$
    follows from the fact that only finite-dimensional subspaces are involved,
    and for these it is enough to see that for any $u$ such that the right-hand
    side vanishes, the left-hand side must vanish as well. But this is
    immediate. This completes the proof.
    %
    %
\end{proof}

For $u \in \Usz$,
\begin{displaymath}
  \E(u) = \Equad(u) + \< \del \E(0), u \>,
\end{displaymath}
which according to the two foregoing Lemmas is continuous with respect
to the $\UsH$-topology and thus has a unique extension to
$\UsH$. Since the first term is $C^k$ and the second is linear and
bounded, the result $\E \in C^k$ follows as well. This completes the
proof of Lemma \ref{th:extension_lemma}.

\subsection{Proof of Lemma \ref{th:disl:up_lemma} (properties of the
  dislocation predictor)}
\label{sec:prf_disl_yd_lemma}
Before we move on to prove the extension lemma in the dislocation case, Lemma
\ref{th:disl:extension_lemma}, we establish the facts about the dislocation
predictor displacement $\up$, summarized in Lemma \ref{th:disl:up_lemma}. We
begin by analyzing the auxiliary deformation map $\xi$ defined in
\eqref{eq:disl:defn_u0} in more detail. To simplify the notation let $\zeta :=
\xi^{-1}$ throughout this section.

\begin{lemma}
  \label{th:facts_xi}
  (a) If $\Rdisl$ is sufficiently large, then $\xi : \R^2 \setminus
  (\Gamma \cup B_{\Rdisl/4}) \to \R^2 \setminus \Gamma$ is injective.

  (b) The range of $\xi$ contains $\R^2 \setminus (\Gamma \cup
  B_{\Rdisl/4})$.

  (c) The map $\zeta^S(x) := ${\scriptsize
    $\cases{\zeta(x-\burg_{12}), & x_2 > \hat{x}_2, \\ \zeta(x), & x_2
      \leq \hat{x}_2}$} can be continuously extended to the half-space
  $\halfspace = \{ x_1 > \Rdisl+\burg_1 \}$, and after this extension
  we have $\zeta^S \in C^\infty(\halfspace)$.
\end{lemma}
\begin{proof}
  {\it (a) } Suppose that $x, x' \in \R^2 \setminus (\Gamma \cup
  B_{\Rdisl/4})$ and $\xi(x) = \xi(x')$, then $x_2 = x_2'$ and since
  $s \mapsto s + \smfrac{\burg_1}{2\pi} \arg( (s-\hat{x}_1, x_2-\hat{x}_2) )$ is clearly
  injective, it follows $x_1 = x_1'$ as well.

  {\it (b) } The map $\xi$ leaves the $x_2$ coordinate unchanged and
  only shifts the $x_1$ coordinate by a number between $0$ and
  $\burg_1$. Thus, for $\Rdisl/4 > |\burg_{1}|$, the statement clearly
  follows.

  {\it (c) } To compute the jump in $\zeta$ let $x \in \Gamma$, $x_1
  > \Rdisl+\burg_1$, then we see that $\xi(x+) = x$, $\xi(x-) = x -
  \burg_{12}$, and hence $\zeta(x+) = x$ and $\zeta(x-) =
  x+\burg_{12}$. Thus, we have
  \begin{displaymath}
    \zeta(x+) - \zeta\b( (x-\burg_{12})- \burg_{12}) = x - [x - \burg_{12} + \burg_{12}] = 0.
  \end{displaymath}
  Consequently, using also $\D \zeta(x) = \D \xi(\zeta(x))^{-1}$ and
  $\D\xi \in C^\infty(\R^2 \setminus \{0\})$, we obtain
  \begin{align*}
    \D\zeta(x+) - \D\zeta\b( (x-b_{12})- \b) &= \D\xi( \zeta(x+) )^{-1} -
    \D\xi\b(\zeta\b( (x-b_{12})-\b) \b)^{-1} \\
    &= \D\xi( \zeta(x+) )^{-1} - \D\xi(\zeta(x+))^{-1} = 0.
  \end{align*}
  The proof for higher derivatives is a straightforward induction
  argument.
\end{proof}

We now proceed with the proof of Lemma \ref{th:disl:up_lemma}.

{\it Proof of (i): $\up$ is well-defined. } The elasticities tensor
$\bbC$ is derived from the interaction potential and due to the
lattice stability assumption \eqref{eq:pt:stab_lattice} satisfies the
strong Legendre--Hadamard condition (see \S~\ref{sec:elast:lin} for
more detail). It is then shown in \cite[Sec. 13-3,
Eq. 13-78]{HirthLothe} that one can always find a solution to
\eqref{eq:disl:linel_pde} of the form
\begin{displaymath}
  u^\lin_i(\hat{x} + x) = {\rm Re} \Bg( \sum_{n = 1}^3 B_{i, n} \log\b( x_1 + p_n x_2 \b) \Bg),
\end{displaymath}
with parameters $B_{i,n}, p_n \in \C, i, n = 1, 2, 3$. (We use
$B_{k,n} \equiv -A_k(n)D(n) / (2\pi i)$ in the notation of Hirth and Lothe
\cite{HirthLothe}.)  The
logarithms are chosen with branch cut $\Gamma$.

Having seen that $\ulin$ is well-defined, Lemma \ref{th:facts_xi}
immediately implies that $\up$ is also well-defined. This completes
the proof of Lemma \ref{th:disl:up_lemma} (i).

Before we go on to prove statements (ii) and (iii) of Lemma
\ref{th:disl:up_lemma} we establish another auxiliary result.

\begin{lemma}
  Let $\partial_\alpha$, $\alpha \in \N^2$ be the usual multi-index
  notation for partial derivatives, then there exist maps
  $g_{\alpha,\beta} \in C^\infty(\R^2 \setminus \Gamma)$ satisfying
  $|\D^j g_{\alpha,\beta}| \lesssim |x|^{-1-j-|\alpha|_1+|\beta|_1}$
  such that
  \begin{equation}
    \label{eq:derivatives_u0}
    \partial_{\alpha} \up(x) = \b(\partial_{\alpha} \ulin\b)\b(\xi^{-1}(x)\b) +
    \sum_{j = 1}^{|\alpha|_1} \sum_{\substack{\beta \in \N^d \\ |\beta|_1 = j}}
    g_{\alpha,\beta}(x) \b( \partial_\beta \ulin\b)\b(\xi^{-1}(x)\b) \quad
    \text{for } \alpha \in \N^2.
  \end{equation}
  Moreover, for all $\alpha$ and $\beta$, $g_{\alpha,\beta} \circ S$ can be
  extended to a function in $C^\infty(\halfspace)$.
\end{lemma}
\begin{proof}
  We only need to consider $|x| > \Rdisl + |\burg|$.

  For $\alpha = 0$ the result is trivial (with $g_{0,0} = 0$). For the
  purpose of illustration, consider $\alpha = e_s$, $s \in \{1,2\}$,
  which we treat as the entire gradient:
  \begin{align*}
    \D \up &= \D\ulin(\xi^{-1}(x)) \D\xi^{-1}(x) \\
    &= \D\ulin(\xi^{-1}(x)) + \D\ulin(\xi^{-1}(x)) \b( \D \xi^{-1}(x)
    - \mI \b).
  \end{align*}
  Since $|\D \xi^{-1}(x) - \mI| \lesssim |x|^{-1}$, the result follows
  for this case.

  In general the proof proceeds by induction. Suppose the result is
  true for all $\alpha$ with $|\alpha|_1 \leq m$.

  We use induction over $|\alpha|_1$. For $|\alpha|_1 = 0$ the result
  is trivial with $g_{0, 0} = 0$. Let $|\bar{\alpha}|_1 = n-1 \geq 0$,
  $\alpha = \bar{\alpha} + e_s$ for some $s \in \{1,2\}$. Then,
  \begin{align*}
    \partial_{\alpha} \up &= \partial_{e_s}
    \bg[ \partial_{\bar\alpha} \ulin + \sum_{|\beta|_1 \leq
      |\bar\alpha|_1} g_{\bar\alpha,\beta} \partial_{\beta} \ulin \bg] \\
    &= \partial_{e_1 + \bar\alpha} \ulin \partial_{e_s} \zeta_1
    + \partial_{e_2 + \bar\alpha}\ulin \partial_{e_s} \zeta_2
    \\
    & \qquad
    + \sum_{|\beta|_1 \leq |\alpha|_1} \B[ \partial_{e_s}
    g_{\bar\alpha,\beta} \partial_{\beta} \ulin + g_{\bar\alpha,\beta}
    \B( \partial_{e_1+\beta} \ulin \partial_{e_s} \zeta_1
    + \partial_{e_2+\beta} \ulin \partial_{e_s} \zeta_2 \B) \B]\\
    &= \partial_{\alpha} \ulin + \partial_{e_1+\bar\alpha}\ulin
    \b( \partial_s \zeta_1 - \delta_{1s}\b)
    + \partial_{e_2+\bar\alpha} \ulin \b( \partial_s\zeta_2 -
    \delta_{2s}\b)  + \sum_{|\beta|_1 \leq |\alpha|_1 + 1} g_{\alpha,
      \beta}' \partial_{\beta} \ulin.
  \end{align*}
  for some $g_{\alpha,\beta}'$ that depend on $g_{\bar\alpha,\beta}$
  and its derivatives and have the same regularity and decay as stated
  for $g_{\alpha,\beta}$.

  Finally, the coefficient functions $( \partial_s\zeta_i -
  \delta_{is})$ are readily seen to also satisfy the same regularity
  and decay as stated for $g_{\alpha,\beta}$ with any $|\beta|_1 =
  |\alpha|_1$. This concludes the proof.
\end{proof}

{\it Proof of (ii)} Let $x \in \Gamma \cap \halfspace$, then
\begin{align*}
  S_0\up(x+) - S_0\up( x-) &= \up(x+) - \b[\up\b((x-\burg_{12})-\b)-\burg \b] \\
  &= \ulin(x+)  - \b[ \ulin\b( (x-\burg_{12} + \burg_{12})- \b) - \burg \b] \\
  &= \ulin(x+) - \ulin(x-) - \burg = \burg - \burg = 0.
\end{align*}

For derivatives of arbitrary order, the result is an immediate
consequence of \eqref{eq:derivatives_u0} and of Lemma
\ref{th:facts_xi}(c).  For illustration only, we show directly that
$\D \up$ is continuous across $\Gamma$: if $x \in \Gamma \cap \Omega$,
then, employing Lemma \ref{th:facts_xi} in the second identity,
\begin{align*}
  \D \up(x+) - \D\up( (x-b_{12})- )
  &= \D \ulin(\zeta(x+)) \D
  \zeta(x) - \D \ulin(\zeta((x-b_{12})-)) \D
  \zeta(x-b_{12}) \\
  &= \D \ulin(x) \D \xi(x)^{-1} -  \D\ulin(x)
  \D\xi(x)^{-1} = 0.
\end{align*}

{\it Proof of (iii): } This statement is an immediate consequence of
\eqref{eq:derivatives_u0}.

This completes the proof of Lemma \ref{th:disl:up_lemma}.

\subsection{Proof of Lemma \ref{th:disl:extension_lemma}}
\label{sec:prf:E_disl}
The main idea of the proof is the same as in the point defect case,
\S~\ref{sec:prf_extension_pt}. For $u \in \Usz$ we write
\begin{displaymath}
  \E(u) = \Equad(u) + \< \del \E(0), u \>,
\end{displaymath}
where now
\begin{align} \notag
  \Equad(u) &= \sum_{\ell \in \L} V_\ell(Du(\ell)) - \< \del
  V_\ell(\bfO), Du(\ell)\> \\ \notag
  &= \sum_{\ell \in \L} \B( V(D(\up+u)(\ell)) - V(D\up(\ell)) - \<
  \del V(D\up(\ell)), Du(\ell) \> \B), \qquad \text{and} \\
  \< \del \E(0), u \> &= \sum_{\ell \in \L} \< \del V(D\up(\ell)),
  Du(\ell) \>.
  \label{eq:prf:eqn_weak}
\end{align}
It is an analogous argument as in the point defect case to show that
$\Equad \in C^k(\Adm)$.

To prove that $\del\E(0)$ is a bounded linear functional, we first use
\eqref{eq:disl:slip_delV} to rewrite it in the form
\begin{displaymath}
  \<\del\E(0), u \> = \sum_{\ell \in \L} \b\< \del V(\bfO), \Del u(\ell) \b\>.
\end{displaymath}
Next, we convert it to a force-displacement formulation, by generalising
summation by parts to incompatible gradients $\Del$.

\begin{lemma}\label{th:disl:D_rho_conj}
  Let $v\in\UsH$ be such that $v(\ell)=0$ for all $\ell$ such that
  $|\ell|\leq 2|\rdisl|+|\burg_1|$.  Then $\Del_\rho^* v = \Del_{-\rho} v$ for all
  $\rho\in\Rg$. 
\end{lemma}
\begin{proof}
  We let $k\in\L$ and $u\in\UsH$, $u(\ell) := \delta_{k\ell}$.  Then we form the
  expression
  \[
  \sum_{\ell\in\L} \Del_\rho u(\ell) \cdot  v(\ell)
  -
  \sum_{\ell\in\L} u(\ell) \cdot \Del_{-\rho} v(\ell)
  \]
  and show that it vanishes. This result is geometrically evident, but could
    also be proved by a direct (yet tedious) calculation whose details we omit.
\end{proof}

We can now deduce that
\begin{equation}
  \label{eq:disl:delE0-f}
  \begin{split}
    \< \del \E(0), v \> &= \sum_{\ell \in \L} f(\ell) \cdot v(\ell),
    \qquad \text{where, } \\
    f(\ell) &= \sum_{\rho \in \Rg} \b[\Del_{-\rho} V_{,\rho}(e)\b](\ell),
    \qquad \text{for $|\ell|$ sufficiently large.}
  \end{split} 
\end{equation}

To prove that $\del \E(0)$ is bounded we must establish decay of
$f$. For future reference, we establish a more general result than
needed for this proof.

\begin{lemma}
  \label{th:disl:delE0}
  Let $f$ be given by \eqref{eq:disl:delE0-f}, and $0 \leq j \leq k-2$, then
  there exists $C$ such that
  \begin{equation}
    \label{eq:disl:decay_f}
    |\Del^j f(\ell)| \leq C |\ell|^{-3-j}.
  \end{equation}
\end{lemma}
\begin{proof}
  Throughout this proof we will implicitly assume that $|\ell|$ is sufficiently
  large so that the defect core $B_{\rdisl+|\burg|}(\hat{x})$ does not affect
  the computation. We first consider the case $j = 0$.

  {\it Case 1: left halfspace: } We first consider the simplified situation when
  $\ell_1 < \hat{x}_1$, that is we can simply replace $\Del \equiv D$ throughout. We
  will see below that a generalisation to $\ell_1 > \hat{x}_1$ is straightforward.

  We begin by expanding $V_{,\rho}$ to second order, 
  \begin{equation}
    \label{eq:appprf:disl_res_expVrho}
    V_{,\rho}(e) = V_{,\rho}(\bfO) + \< \del V_{,\rho}(\bfO), e \> + \int_0^1 (1-t)
    \< \ddel V_{,\rho}(te) e, e \> \dt.
  \end{equation}
  Point symmetry of $V$ implies that $\sum_\rho V_{,\rho}(\bfO) = 0$. Hence, we
  obtain
  \begin{align}
    \label{eq:appprf:disl:res1-def_f}
    f &= \sum_{\rho,\vsig \in \Rg} V_{,\rho\vsig}(\bfO) D_{-\rho}
        e_{\vsig} + \sum_{\rho, \in \Rg} \int_0^1 (1-t) D_{-\rho}
        \< \ddel V_{,\rho}(te) e, e \> \dt \\
    \notag
      &=: f^{(1)} + f^{(2)}.
  \end{align}
  Since $|D_\rho e(\ell)| \lesssim |\ell|^{-2}$, we easily obtain
  $|f^{(2)}(\ell)| \lesssim |\ell|^{-3}$.

  To estimate the first group we expand
  \begin{align*}
    \b|e_\rho(\ell) - \D_\rho \up(\ell) - \smfrac12 \D_\rho^2 \up(\ell)\b| &\lesssim
    \| \D^3 \up \|_{L^\infty(B_{\rcut}(\ell))} \lesssim |\ell|^{-3}, \quad
    \text{and hence} \\
    \b|D_{-\rho} e_\vsig(\ell) +  \D_\rho \D_\vsig \up(\ell)\b| &\lesssim |\ell|^{-3}.
  \end{align*} 
  Lemma \ref{th:disl:up_lemma}(iii) ($\D^2\up = \D^2\ulin + O(|x|^{-3})$, where
  $\bbC:\D^2 \ulin \equiv 0$) yields
  \begin{displaymath}
    f^{(1)} = - \sum_{\rho,\vsig \in \Rg} V_{,\rho\vsig}(\bfO)
    \D_\rho\D_\vsig \ulin(\ell) + O(|\ell|^{-3}) = O(|\ell|^{-3}).
  \end{displaymath}
  We have therefore shown \eqref{eq:disl:decay_f} for the case $j = 0$, when
  $\ell$ lies in the left half-space.
  
  {\it Case 2: right halfspace: } To treat the case $\ell_1 > \hat{x}_1, |\ell|$
  sufficiently large, we first rewrite
  \begin{displaymath}
    f = \Del_{-\rho} V_{,\rho}(\Del_0 \up)
    = [R D_{-\rho} S] V_{,\rho}\b([R D S_0]\up\b)
    = R D_{-\rho} V_{,\rho}(D S_0\up).
  \end{displaymath}
  Since $S_0 \up$ is smooth in a neighbourhood of $|\ell|$ (even if that
  neighbourhood crosses the branch-cut), we can now repeat the foregoing
  argument to deduce again that $|Sf(\ell)| \lesssim |\ell|^{-3}$ as well
  (cf.\ Remark \ref{rem:reflection}).  But since $S$ represents an $O(1)$
  shift, this immediately implies also that $|f(\ell)| \lesssim
  |\ell|^{-3}$. This completes the proof of \eqref{eq:disl:decay_f}.  
 
  {\it Proof for the case $j > 0$: } To prove higher-order decay,
    assume again at first that $\ell_1 < \hat{x}_1$ and consider
    $\bftau \in \Rg^j$, $j \geq 1$, then
    \begin{displaymath}
      D_\bftau f = \sum_{\rho,\vsig} V_{,\rho\vsig} D_{\bftau} D_{-\rho} e_\vsig
      + \sum_{\rho \in \Rg} \int_0^1 (1-t) D_\bftau D_{-\rho} \< \ddel V_{,\rho}(te) e, e \> \dt =: f^{(1)} + f^{(2)}.
    \end{displaymath}
    An analogous Taylor expansion as above yields
    \begin{displaymath}
      f^{(1)} = - \D_\bftau \sum_{\rho,\vsig \in \Rg} V_{,\rho\vsig}(\bfO)  \D_\rho \D_\vsig \ulin(\ell) + O(|\ell|^{-3-j}) =  O(|\ell|^{-3-j}), 
    \end{displaymath}
    applying again $\sum_{\rho,\vsig \in \Rg} V_{,\rho\vsig}(\bfO)
    \D_\rho \D_\vsig \ulin = 0$.

    The term $f^{(2)}$ is readily estimated by multiple applications
    of the discrete product rule, from which we obtain that
    $|f^{(2)}(\ell)| \lesssim |\ell|^{-j-3}$ again.

    The generalisation to the case $\ell_1 > \hat{x}_1$ is again
    analogous to above, due to the fact that
    \begin{displaymath}
      \Del_{\tau_1} \cdots \Del_{\tau_j} \Del_{-\rho} V_{,\rho} (e)
      = R D_{\tau_1} \cdots D_{\tau_j} D_{-\rho} V_{,\rho}(D S_0 \up).
    \end{displaymath}
    From this point, the argument continues verbatim to the case
    $\ell_1 < \hat{x}_1$.
\end{proof}

Applying Corollary \ref{th:Dg_f} to $f$ yields a map
$g : \L \to (\R^3)^\Rg$ such that
\begin{displaymath}
  \< \del\E(0), v \> = \< g, D v \>, \qquad \text{where} \qquad
  |g(\ell)| \lesssim |\ell|^{-2}.
\end{displaymath}
Thus,
$\< \del\E(0), v \> \leq \|g\|_{\ell^2}\|Dv\|_{\ell^2} \lesssim \|g\|_{\ell^2}
\|\D v\|_{L^2}$, and hence $\del\E(0) \in \UsHd$.

This completes the proof of Lemma \ref{th:disl:extension_lemma}.

\section{Proofs: Regularity}
\label{sec:reg}
In this section we prove the regularity results, Theorem
\ref{th:pt:regularity} and Theorem \ref{th:disl:regularity}.

\subsection{First-order residual for point defects}
\label{sec:reg:prelims}
\def\Rghom{\Rg}
Assume, first, that we are in the setting of the point defect case,
\S~\ref{sec:pt:atm}.  To motivate the subsequent analysis we first
convert the first-order criticality condition $\del\E(\ua) = 0$ for
\eqref{eq:pt:atm}.

Since $\D\ua \in L^2$, $D_\rho\ua(\ell) \to 0$ uniformly as $|\ell| \to \infty$,
for all $\rho \in \Rghom$. Consequently, for $|\ell|$ large, linearised lattice
elasticity provides a good approximation to $\del\E(\ua) = 0$. To exploit this
observation we first define the homogeneous lattice hessian operator
(cf.~\eqref{eq:pt:stab_lattice})
\begin{equation}
  \label{eq:reg:prelims:explicit_H}
  \< Hu, v \> = \sum_{\ell \in \mA\Z^d} \b\< \ddel V(\bfO) Du(\ell), Dv(\ell) \b\> 
  = \sum_{\ell \in \mA \Z^d} \sum_{\rho,\vsig \in \Rghom} D_\rho
  u(\ell)^T V_{\rho\vsig} (\bfO) D_\vsig v(\ell).
\end{equation}
We assume throughout that it is stable in the sense of
\eqref{eq:pt:stab_lattice}.

Finally, to state the first auxiliary result, we recall from \S~\ref{sec:pt:atm}
the definition of the interpolant $Iu$ for discrete displacements $u : \L \to
\R^d$, which provides point values $Iu(\ell)$ for all $\ell \in \mA\Z^d$.

\begin{lemma}[First-Order Residual for Point Defects]
  \label{th:reg:lin_eqn}
  Under the assumptions of Theorem \ref{th:pt:regularity}~there exists $g :
  \mA\Z^d \to (\R^m)^{\Rghom}$ and $R_1, C > 0$ such that
  \begin{align}
    \label{eq:reg:lin_eqn}
    \< H I\ua, v \> &= \< g, D v \>, \qquad \forall v \in
    \Usz(\mA\Z^d), \qquad \text{where } \\
   \label{eq:reg:bound_g_pt}
   \b|g(\ell)\b| &\leq C |D \ua(\ell)|^2 \qquad \forall \ell
   \in \mA\Z^d \setminus B_{R_1}.
  \end{align}
\end{lemma}
\begin{proof}
  Let $u \equiv I\ua$. We rewrite the residual $\< H u, v \>$ as
  \begin{align}
    \notag
    \< Hu, v \> = \sum_{\ell \in \mA\Z^d}  \,& \b\< \ddel V(\bfO)
    Du(\ell), Dv(\ell) \b\> \\
    \label{eq:reg:prelims:10}
    = \sum_{\ell \in \mA \Z^d} \,& \B( \b\< \del V(\bfO) + \ddel
    V(\bfO) Du(\ell)
    - \del V(Du(\ell)), Dv(\ell) \b\> \\[-3mm]
    \notag & \quad + \b\< \del V(Du(\ell)), Dv(\ell) \b\> -
    \b\<\del V(\bfO), Dv(\ell) \b\> \B).
  \end{align}
  The first group can be written as
  \begin{align*}
    & \b\< \del V(\bfO) + \ddel
    V(\bfO) Du(\ell)
    - \del V(Du(\ell)), Dv(\ell) \b\>  =: \< g_1(\ell),
    Dv(\ell) \>,
  \end{align*}
  and where we note that $g_1(\ell)$ is a linearisation remainder and hence
  $|g_1(\ell)| \lesssim |Du(\ell)|^2$ for $|\ell|$ sufficiently large.

  The second group is the residual of the exact solution after
  projection to the homogeneous lattice $\mA\Z^d$. Writing this group
  in ``force-displacement'' format,
  \begin{displaymath}
    \sum_{\ell \in \mA\Z^d} \< \del V(Du(\ell)), Dv(\ell) \> =
    \sum_{\ell \in \mA \Z^d} f(\ell) v(\ell),
  \end{displaymath}
  we observe that $f(\ell) = \sum_{\rho \in \Rg} D_{-\rho}
  V_{,\rho}(Du(\ell))$ has zero mean as well as compact support due to
  symmetry of the lattice. Because of the mean zero condition, we can
  write it in the form $\< f, v \> = \< g_2, D v \>$ where $g_2$ also
  has compact support (cf. Corollary \ref{th:Dg_f}).

  Finally, the third group vanishes identically, which can for example
  be seen by summation by parts. Setting $g = g_1 + g_2$ this
  completes the proof.
\end{proof}

\subsection{The Lattice Green's Function}
\label{sec:reg:lgf}
To obtain estimates on $\ua$ and its derivatives from
\eqref{eq:reg:lin_eqn} we now analyse the lattice Green's function
(inverse of $H$). The following results are widely expected but we
could not find rigorous statements in the literature in the generality
that we require here.

Using translation and inversion symmetry of the lattice, the
homogeneous finite difference operator $H$ defined in
\eqref{eq:reg:prelims:explicit_H} can be rewritten in the form
\begin{equation}
  \label{eq:reg:prelims:explicit_H_2}
  \< H u, u \> = \sum_{\ell \in \mA \Z^d} \sum_{\rho \in \Rghom'} D_\rho
  u(\ell)^T A_\rho D_\rho u(\ell)
\end{equation}
where $\Rghom' := \{ \rho + \vsig \sep \rho, \vsig \in \Rghom\}
\setminus \{0\}$ and $A_\rho \in \R^{d \times d}$. (Written in terms
of $V_{,\rho\vsig}$, $A_\rho = \sum_{\vsig, \tau \in \Rg, \vsig-\tau =
  \rho} V_{,\vsig\tau}$. Alternatively, one can define $A_\rho = -2
\frac{\partial^2 \< Hu, u \>}{\partial u(0) \partial u(\rho)}$ and
arrive at the same result; cf.~\cite[Lemma 3.4]{Hudson:stab}.)  Since
the Green's function estimates hold for general operators of the form
\eqref{eq:reg:prelims:explicit_H_2} we recall the associated stability
\begin{equation}
  \label{eq:reg:lgf:stab_H}
  \< H v, v \> \geq \gamma \| \D v \|_{L^2}^2 \qquad \forall v \in
  \Usz(\mA\Z^d),
\end{equation}
for some $\gamma > 0$.



Next, we recall the definitions of the semi-discrete Fourier transform
and its inverse, 
\begin{equation}
  \label{eq:sd_fourier_transform}
  \Ftd[u](k) := \sum_{\ell \in \mA\Z^d} e^{i k \cdot \ell} u(\ell),
  \qquad \text{and} \qquad
  \Ftd^{-1}[\hat{u}](\ell) = \int_{\Br} e^{-ik \cdot \ell} \hat{u}(k)\,dk,
\end{equation}
where $\Br \subset \R^d$ is the first Brillouin zone. As usual, the
above formulas are well-formed for $u \in \ell^1(\mA\Z^d; \R^m)$ and
$\hat{u} \in L^1(\Br; \R^m)$, and are otherwise extended by
continuity.

Transforming \eqref{eq:reg:prelims:explicit_H_2} to Fourier space, we
get
\begin{displaymath}
  \< H u, u \> = \int_{\Br} \hat{u}(k)^* \hat{H}(k) \hat{u}(k) \dk,
  \quad \text{where} \quad
  \hat{H}(k) = \sum_{\rho \in \Rg'} 4 \sin^2\b( \smfrac12 k \cdot
  \rho \b) A_\rho.
\end{displaymath}
Lattice stability \eqref{eq:reg:lgf:stab_H} can equivalently be
written as $\hat{H}(k) \geq \gamma' |k|^2 \mI$. Thus, if
\eqref{eq:reg:lgf:stab_H} holds, then the lattice Green's function can
be defined by
\begin{displaymath}
  \Gr(\ell) := \Ftd^{-1}[\hat{\Gr}](\ell), \quad \text{where} \quad
  \hat{\Gr}(k) := \hat{H}(k)^{-1}.
\end{displaymath}

We now state a sharp decay estimate for $\Gr$. 

\begin{lemma}
  \label{th:green_fcn}
  Let $H$ be a homogeneous finite difference operator of the form
  \eqref{eq:reg:prelims:explicit_H_2} satisfying the lattice stability
  condition \eqref{eq:reg:lgf:stab_H}, and let $\Gr$ be the associated
  lattice Green's function.
  
  Then, for any $\bfrho \in \Rghom^j, j > 0$, or $j = 0$ if $ d = 3$, there
  exists a constant $C$ such that
  \begin{equation}
    \label{eq:green_fcn_decay}
    \b| D_\bfrho \Gr(\ell) \b| \leq C (1+|\ell|)^{-d-j+2} \qquad \forall
    \ell \in \mA\Z^d.
  \end{equation}
\end{lemma}
\begin{proof}
  The strategy of the proof is to compare the lattice Green's function
  with a continuum Green's function.

  {\it Step 1: Modified Continuum Green's Function: } Let $\GrC$
  denote the Green's function of the associated linear elasticity
  operator $L = -\sum_{\rho \in \Rghom'} \D_\rho \cdot A_\rho
  \D_\rho$, and $\hat{\GrC}(k)$ its (whole-space) Fourier
  transform. Then, $\hat{\GrC}(k) = (\sum_{\rho \in \Rghom'}
  (\rho\cdot k)^2 A_\rho)^{-1}$, where we note that lattice stability
  assumption \eqref{eq:reg:lgf:stab_H} immediately implies that
  $\sum_{\rho \in \Rghom'} (\rho\cdot k)^2 A_\rho \geq \gamma' |k|^2
  \mI$, where $\gamma' > 0$. We shall exploit the well-known fact that
  \begin{equation}
    \label{eq:grfcn_prf:10}
    |\D^j \GrC(x)| \leq C |x|^{-d-j+2} \quad \text{for } |x|  \geq 1,
  \end{equation}
  where $C = C(j, \{A_\rho\})$; see \cite[Theorem 6.2.1]{Morrey}.
  
  Let $\hat{\eta}(k) \in C^\infty_{\rm c}(\Br)$, with $\hat{\eta}(k) = 1$ in a
  neighbourhood of the origin. Then, it is easy to see that its inverse
  (whole-space) Fourier transform
  $\eta := \Ft^{-1}[\hat{\eta}] \in C^\infty(\R^d)$ with superalgebraic decay. From this and
  (\ref{eq:grfcn_prf:10}) it is easy to deduce that
  \begin{equation}
    \label{eq:grfcn_prf:20}
    \b| D_\balpha (\eta \ast \GrC)(\ell) \b| \leq C |\ell|^{2-d-j}
    \quad \text{for } |\ell| \geq 1,
  \end{equation}
  where $C = C(j, H)$ and $\balpha \in \Rg^j$ is the multi-index defined in the
  statement of the theorem.

  {\it Step 2: Comparison of Green's Functions: } Our aim now is to
  prove that
  \begin{equation}
    \label{eq:grfcn_prf:30}
    \b| D_\balpha (\Gr - \eta \ast \GrC)(\ell) \b| \leq C |\ell|^{1-d-j},
  \end{equation}
  which implies the stated result. (In fact, it is a stronger
  statement.)

  We write
  \begin{displaymath}
    \Ftd[D_\balpha (\Gr - \eta \ast \GrC)] = (\hat{\Gr} - \hat{\eta}
    \hat{\GrC})  p_\balpha(k),
  \end{displaymath}
  where $p_\balpha(k) \in C^\infty_{\rm per}(\Br)$ with $|p_\balpha(k)|
  \lesssim |k|^j$. (To be precise, $p_\balpha(k) \sim (-i)^j \prod_{s =
    1}^j (\alpha_s \cdot k)$ as $k \to 0$.) Fix some $\epsilon > 0$
  such that $\hat{\eta} = 1$ in $B_\epsilon$. The explicit
  representations of $\hat{\Gr}$ and $\hat{\GrC}$ make it
  straightforward to show that (one employs the fact that
  $\hat{\Gr}^{-1} - \hat{\GrC}^{-1}$ has a power series starting with
  quartic terms)
  \begin{displaymath}
    \b| \Delta^n (\hat{\Gr} - \hat{\GrC}) p_\balpha(k) \b| \lesssim
    |k|^{-2n + j} 
  \end{displaymath}
  for $k \in B_\epsilon$, while $\Delta^n (\hat{\Gr} - \hat{\eta}\hat{\GrC})$ is
  bounded in $\Br \setminus B_\epsilon$. Thus, if $d-1+j$ is even and
  we choose $2n := d-1+j$, then we obtain that $\Delta^n (\hat{\Gr} -
  \hat{\GrC}) p_\balpha(k) \in L^1(\Br)$, which implies that
  \begin{align*}
    \b| D_\balpha (\Gr - \eta \ast \GrC)(\ell) \b| &= \b|\Ftd^{-1}[ \Delta^{-n}
    \Delta^n (\hat{\Gr} - \hat{\eta}
    \hat{\GrC})  p_\balpha(k)](\ell)\b| \\
    &\lesssim |\ell|^{-2n} = |\ell|^{1-d-j},
  \end{align*}
  which is the desired result (\ref{eq:grfcn_prf:30}).

  If $d-1+j$ is odd, then we can deduce (\ref{eq:grfcn_prf:30}) from
  the result for a larger multi-index $\balpha' = (\balpha, \rho')$ of
  length $j'$. Namely, fix $\ell \in \mA \Z^d$ and choose $\rho'$ a
  nearest-neighbour direction pointing away from the origin, then
  \begin{displaymath}
    D_\balpha \Gr(\ell) = \sum_{n = 0}^\infty D_{\balpha'}
    \Gr(\ell+n\rho')
  \end{displaymath}
  from which (\ref{eq:grfcn_prf:30}) easily follows.
\end{proof}


\subsection{Decay estimates for $Du$, point defect case}
\label{sec:reg:decay_Du}
At the end of this section we prove Theorem \ref{th:pt:regularity} for
the cases $j = 0, 1$. In preparation we first prove a more general
technical result.

\begin{lemma}
  \label{th:reg1}
  Let $H$ be a homogeneous finite difference operator of the form
  \eqref{eq:reg:prelims:explicit_H_2} satisfying the stability
  condition \eqref{eq:reg:lgf:stab_H}. Let $u \in \UsH(\mA\Z^d)$
  satisfy
  \begin{equation}
    \label{eq:reg:decay_Du:abstract_eqn}
    \< H u, v \> = \< g, Dv\>,  \qquad  \text{where} \quad
    \cases{ & \hspace{-4mm} g : \mA\Z^d \to (\R^m)^{\Rg}, \\[1mm]
      & \hspace{-4mm} |g(\ell)| \leq C 
      (1+|\ell|)^{-p} +C
    h(\ell) |Du(\ell)|, }
  \end{equation}
  $p \geq d$ and $h \in \ell^2(\mA\Z^d)$. Then, for any $\rho \in
  \Rghom$, there exists $C \geq 0$ such that, for $|\ell| \geq 2$,
  \begin{displaymath}
    |D_{\rho} u(\ell)| \leq \cases{ C|\ell|^{-d}, & \text{if } p > d, \\
      C|\ell|^{-d} \log |\ell|, & \text{if } p = d.}
  \end{displaymath}
\end{lemma}
\begin{proof}
  Recall the definition of the Green's function $\Gr$ from
  \S~\ref{sec:reg:lgf} and its decay estimates stated in Lemma
  \ref{th:green_fcn}.  Then, for all $\ell \in \mA\Z^d$, it holds that
  \begin{align*}
    u(\ell) &= - \sum_{k \in \mA\Z^d} \sum_{\rho \in \Rg} D_\rho
    \Gr(\ell -k) g_{\rho}(k), \qquad \text{and
      hence, for all $\sigma \in \Rg$, } \\
    D_\sigma u(\ell) &= - \sum_{k\in \mA\Z^d} \sum_{\rho \in \Rg}
    D_\sigma D_\rho \Gr(\ell -k) g_{\rho}(k) 
    = - \sum_{k\in \mA\Z^d} \sum_{\rho \in \Rg} D_\sigma D_\rho
    \Gr(k) g_{\rho}(\ell - k).
  \end{align*}
  Applying Lemma \ref{th:green_fcn} and the assumption
  \eqref{eq:reg:decay_Du:abstract_eqn}, we obtain
  \begin{equation}
    \label{eq:reg1:5}
    \b|D_\sigma u(\ell)\b| \leq C \sum_{k \in \mA \Z^d} (1+|k|)^{-d}
    \B((1+|\ell-k|)^{-p} + h(\ell-k) |D u(\ell-k)| \B).
  \end{equation}
  
  For $r > 0$, let us define $w(r):= \sup_{\ell \in \mA\Z^d, \; |\ell|
    \geq r} |D u(\ell)|$. Our goal is to prove that there exists a
  constant $C>0$ such that
  \begin{equation}
    \label{eq:decayw}
    w(r) \leq C z(r)(1 + r)^{-d} \qquad \text{for all $r > 0$,}
  \end{equation}
  where $z(r) = 1$ if $p > d$ and $z(r) = \log(2+r)$ if $p = d$. The
  proof of (\ref{eq:decayw}) is divided into two steps.

  {\it Step 1: } We shall prove that there exists a constant $C>0$ and
  $\eta: \R_+ \to \R_+$, $\eta(r) \longrightarrow 0$ as $r\to +
  \infty$, such that for all $r>0$ large enough,
  \begin{equation}
    \label{eq:ineqcentral}
    w(2r) \leq C z(r) (1+r)^{-d} + \eta(r)w(r).
  \end{equation}

  {\it Step 1a: } Let us first establish that, for all $|\ell| \geq 2
  r$, we have
  \begin{equation}
    \label{eq:fterm}
    \bg| \sum_{k\in \mA\Z^d} (1+|k|)^{-d} (1+|\ell-k|)^{-p} \bg| \leq C z(r) (1+r)^{-d}. 
  \end{equation}

  We split the summation into $|k| \leq r$ and $|k| > r$. We shall
  write $\sum_{|k| \leq r}$ instead of $\sum_{k \in \mA\Z^d, |k| \leq
    r}$, and so forth.

  For the first group, the summation of $|k| \leq r$, we estimate
  \begin{align}
    \sum_{|k|\leq r} (1+|k|)^{-d}  \,  (1+|\ell-k|)^{-p} 
    \notag
    & \leq (1 + r)^{-p} \sum_{|k|\leq r} (1+|k|)^{-d}  \\
    \label{eq:reg1:30}
    & \leq C (1 + r)^{-p} \log(2+r).
  \end{align}

  We now consider the sum over $|k| > r$. If $p > d$, then
  $(1+|\ell-k|)^{-p}$ is summable and we can simply estimate
  \begin{align}
    \notag
    \sum_{|k| > r} (1+|k|)^{-d} \b(1+|\ell-k|\b)^{-p} 
    & \leq  (1+r)^{-d} \sum_{|k| > r} (1+|\ell-k|)^{-p}  \\
    \label{eq:reg1:35a}
    & \leq   C (1+r)^{-d}, \qquad \text{ if $p > d$.}
  \end{align}
  If $p = d$, then we introduce an exponent $\delta > 0$, which we
  will specify momentarily, and estimate
  \begin{align*}
    \sum_{|k| > r} (1+|k|)^{-d} \b(1+|\ell-k|\b)^{-d}
    &  \leq  (1+r)^{-d + \delta} \sum_{ |k| > r} (1+|k|)^{-\delta}
    (1+|\ell-k|)^{-d} \\
    & \hspace{-2cm}\leq (1+r)^{-d + \delta} \bg( \sum_{|k|> r} \b(1+|k|\b)^{-(d +
      \delta)}\bg)^{\frac{\delta}{d + \delta}}\bg( \sum_{|k| > r} \b(1
    +|\ell-k|\b)^{-(d+\delta)}\bg)^{\frac{d}{d + \delta}} \\
    & \hspace{-2cm} \leq (1+r)^{-d+\delta} \, \sum_{k \in \mA\Z^d} \b(1+|k|\b)^{-(d+\delta)}.
  \end{align*}
  Applying the bound $\sum_{k \in \mA\Z^d} \b(1 + |k|\b)^{-(d+\delta)}
  \leq C \delta^{-1}$ we deduce that
  \begin{displaymath}
    \sum_{|k| > r} (1+|k|)^{-d} \b(1+|\ell-k|\b)^{-d} \leq  C
    (1+r)^{-d} \frac{(2+r)^{\delta}}{\delta}.
  \end{displaymath}
  Finally, we verify that, choosing $\delta = 1/\log(2+r)$ ensures
  $(2+r)^{\delta} \delta^{-1} = e \log(2+r)$, and hence we conclude
  that
  \begin{equation}
    \label{eq:reg1:35b}
    \sum_{|k| > r} (1+|k|)^{-d} \b(1+|\ell-k|\b)^{-d}  \leq C
    (1+r)^{-d} \log(2+r), \qquad \text{if $p = d$.}
  \end{equation}
  Combining \eqref{eq:reg1:30}, \eqref{eq:reg1:35a} and
  \eqref{eq:reg1:35b} yields \eqref{eq:fterm}.

  {\it Step 1b: } Let us now consider the remaining group in
  \eqref{eq:reg1:5}, 
  \begin{align*}
    \sum_{k \in \mA\Z^d} (1+|k|)^{-d} h(\ell-k) |D u(\ell-k)|,
  \end{align*}
  which we must again estimate for all $|\ell| \geq 2 r$.

  Recall that $h, |Du| \in \ell^2$. Defining $\tilde{h}(r) :=
  \sup_{|k| \geq r} h(k)$, we have $\tilde{h}(r) \to 0$ as $r \to
  +\infty$, and
  \begin{align*}
    \hspace{4mm} & \hspace{-7mm} \sum_{k \in \mA\Z^d} (1+|k|)^{-d} h(\ell-k)
    |Du(\ell-k)| \\
    &= \sum_{|k| \geq r} (1+|k|)^{-d} h(\ell - k) |Du(\ell -k)|
    + \sum_{|k| <  r} (1+|k|)^{-d}  h(\ell - k) |Du(\ell -k)| \\
    & \leq  C (1+r)^{-d} \sum_{|k|\geq r} |h(\ell - k)| |Du(\ell -k)|
    +  w(r) \sqrt{\tilde{h}(r)} 
    \sum_{|k| < r} (1+|k|)^{-d} |h(\ell - k)|^{1/2}\\
    & \leq  C (1+r)^{-d} \|h\|_{\ell^2}\|Du\|_{\ell^2} + w(r)
    \sqrt{\tilde{h}(r)}  \| (1+|k|)^{-d} \|_{\ell^{4/3}}\|h\|_{\ell^2}^{1/2}\\
    & \leq  C \bg( (1+r)^{-d} + w(r) \sqrt{\tilde{h}(r)} \bg).
  \end{align*}
  Combining this estimate with \eqref{eq:fterm} we have proved
  \eqref{eq:ineqcentral} with $\eta(r):= C \sqrt{\tilde{h}(r)}$.

  {\it Step 2: } Let us define $v(r):= \frac{r^d}{z(r)} w(r)$ for all
  $r>0$. We shall prove that $v$ is bounded on $\R_+$, which implies
  the desired result. Multiplying \eqref{eq:ineqcentral} with $2^d r^d
  / z(2r)$, we obtain
  \begin{displaymath}
    v(2r) \leq C \b( 1 + \eta(r) v(r)\b).
  \end{displaymath}
  There exists $r_0 > 0$ such that, for all $r > r_0$, $C \eta(r) \leq
  \frac{1}{2}$. This implies that, for all $r>r_0$,
  \begin{displaymath}
    v(2r) \leq C + \frac{1}{2} v(r). 
  \end{displaymath}
  Denoting $F:= \mathop{\sup}_{r\leq r_0} v(r)$ and reasoning by
  induction, we obtain that, for all $r>r_0$,
  \begin{displaymath}
    v(r) \leq  C + \frac{1}{2}\left( C + \frac{1}{2}\left( \ldots \left(C + \frac{1}{2}F\right)\ldots \right)\right)
    \leq C \sum_{k=1}^{N(r)} \frac{1}{2^k} + \frac{1}{2^{N(r)}}F,\\
  \end{displaymath}
  where $N(r) \leq C \log(2+r)$. Finally, the above inequality implies
  that $v(r) \leq C +F$ and thus $v$ is bounded on $\R_+$.

  This implies \eqref{eq:decayw} and thus completes the proof of the
  lemma.
\end{proof}

\begin{proof}[Proof of Theorem \ref{th:pt:regularity}, $j = 0, 1$]
  The case $j = 1$ is an immediate corollary of Lemma~\ref{th:reg1}
  and Lemma \ref{th:reg:lin_eqn}.

  To establish the case $j = 0$ we first note that, due to $|D_{\rho}
  \ua(\ell)| \leq C |\ell|^{-d}$ for all $\rho$ it can be easily shown
  that $\ua(\ell) \to c$ uniformly as $|\ell| \to \infty$. Thus,
  \begin{displaymath}
    \ua(\ell) - c = \sum_{i = 1}^\infty \B( \ua\b(\ell+i\rho\b) - \ua\b(\ell+(i-1)\rho\b) \B).
  \end{displaymath}
  Choosing $\rho$ such that $|\ell+i\rho| \geq c (|\ell|+i)$, we
  obtain the stated bounds.
\end{proof}

\subsection{Decay estimates for higher derivatives, point defect case}
\label{sec:reg:higher_pt}
From \S~\ref{sec:reg:decay_Du} we now know that $|D u(\ell)| \leq C
|\ell|^{-d}$ for $|\ell|$ sufficiently large, and more generally we
can hope to, inductively, obtain that $|D^i u(\ell)| \leq C
|\ell|^{1-d-i}$. Using this induction hypothesis we next establish
additional estimates on the right-hand side $g$ in
\eqref{eq:reg:lin_eqn}. 

Note that, if $|D^i u(\ell)| \lesssim |\ell|^{-p-i}$, then
\begin{equation}
  \label{eq:reg:higher_pt:observation}
  |D_\rho D^i u(\ell)| \leq |D^i u(\ell+\rho)|
  +|D^i u(\ell)| \lesssim |\ell|^{-p-i}
\end{equation}
as well, which gives a first crude estimate for the decay. Exploiting
this observation, the proofs of the higher-order decay estimates take
a somewhat simpler form, as they need to address the nonlinearity.

\begin{lemma}[Higher Order Residual Estimate, Point Defect Case]
  \label{th:reg:higher_res_pt}
  Suppose that the assumptions of Lemma \ref{th:reg:lin_eqn} are
  satisfied and that
  \begin{displaymath}
    \label{eq:reg:higher_res_pt-Diuasm}
    |D^i u(\ell)| \leq C |\ell|^{1-d-i} \qquad \text{for } i = 1,
    \dots, j, \quad |\ell| \geq R_1,
  \end{displaymath}
  then there exist $R_2, C$ such that
  \begin{displaymath}
    |D^j g(\ell)| \leq C |\ell|^{-1-d-j}
    \qquad \text{ for } |\ell| \geq R_2,
  \end{displaymath}
  where $g$ is defined in \eqref{eq:reg:lin_eqn}.
\end{lemma}
\begin{proof}
  The elementary but slightly tedious proof is a continued application
  of a discrete product rule, exploiting the observation
  \eqref{eq:reg:higher_pt:observation}. We begin by noting that
  $A_\rho f(\ell) := \smfrac12 (f(\ell+\rho) + f(\ell))$ yields the
  discrete product rule
  \begin{equation}
    \label{eq:disc_product_rule}
    D_\rho (f_1(\ell) f_2(\ell)) =
    D_\rho f_1(\ell) A_\rho f_2(\ell) + A_\rho f_1(\ell) D_\rho
    f_2(\ell), \qquad \rho \in \Rghom.
  \end{equation}

  Let $1\leq j \leq k-2$. Recall from the proof of Lemma
  \ref{th:reg:lin_eqn} that, for $|\ell| \geq R_1$, chosen
  sufficiently large,
  $ g(\ell) = \del V(\bfO) + \ddel V(\bfO) D u(\ell) - \del
  V(Du(\ell))$.
  Let $R_2 \geq R_1$ such that all the subsequent operations are
  meaningful.  We expand to order $j$ with explicit remainder of order
  $j+1$:
  \begin{align*}
    g_\rho(\ell) 
    &= 
      \frac12 \sum_{\vsig,\tau \in \Rg} \int_{s = 0}^1
      V_{,\rho\vsig\tau}(Du(\ell)) (1-s) \ds D_\vsig u(\ell) D_\tau
      u(\ell), \quad \text{ if } j = 1, \quad \text{and in general,} \\
    g_\rho(\ell) 
    &= 
      \frac12 \sum_{\bftau \in \Rg^2} \<V_{,\rho\bftau},
      D_\bftau^\otimes u(\ell) \> + \dots + \frac{1}{j !} \sum_{\bftau \in
      \Rg^j}
      \< V_{,\rho\bftau}, D_\bftau^\otimes u(\ell) \> \\
    & \qquad
      + \frac{1}{(j+1)!} \sum_{\bftau \in \Rghom^{j+1}} \int_0^1 \<
      V_{\rho, \bftau}(sDu(\ell)), D_\bftau^\otimes u(\ell) \> (1-s)^{j} \ds, 
  \end{align*}
  where $V_{,\rho\bftau} = V_{,\rho\bftau}(\bfO)$ and
  $D_\bftau^\otimes u(\ell) = \bigotimes_{k = 1}^i D_{\tau_k} u(\ell)$
  for $\bftau = (\tau_1, \dots, \tau_i)$.

  Let $\balpha = (\alpha_1, \dots, \alpha_{j}) \in \Rg^j$ be a
  multi-index. For any ``proper subset''
  $\balpha' = (\alpha_i)_{i \in I}, I \subsetneq \{1, \dots, j\}$, we
  have by the assumptions made in the statement of the lemma that
  \begin{displaymath}
    |D_{\balpha'} u(\ell)| \leq C |\ell|^{1-d - \# I} \qquad \text{for }
    |\ell| \geq R_1.
  \end{displaymath}
  Thus, applying the discrete product rule
  \eqref{eq:disc_product_rule}, we obtain, for $\bftau \in \Rg^{s}$,
  $s \geq 2$,
  \begin{equation}
    \label{eq:appprf:highdec_pt_100a}
    \b|D_{\alpha_1} \cdots D_{\alpha_j} \b(D_{\bftau}^\otimes u(\ell)
    \b) \b| \leq C |\ell|^{-ds-j} \leq C |\ell|^{-1-d-j}.
  \end{equation}
  Using, moreover, the estimates
  \begin{equation}
    \label{eq:appprf:highdec_pt_100b}
    \b|D_{\alpha_1} \cdots D_{\alpha_j}  V_{\rho,
      \bftau}(sDu(\ell)) \b| \leq C \quad
    \text{and} \quad \b| D_\bftau^\otimes u \b| \leq C |\ell|^{-d(j+1)}
    \leq C |\ell|^{-1-d-j},
  \end{equation}
  for $\bftau \in \Rg^{j+1}$, we can conclude that
  \begin{displaymath}
    \b| D_{\alpha_1} \cdots D_{\alpha_j} g_\rho(\ell) \b| \leq C
    |\ell|^{-1-d-j} + C |\ell|^{-d} |D^{j+1} u(\ell)|
    \quad \text{for } |\ell| \geq R_1.
  \end{displaymath}
  This, together with \eqref{eq:reg:higher_pt:observation}, completes
  the proof.
\end{proof}

To complete the proof of Theorem \ref{th:pt:regularity} we need a final
auxiliary lemma that estimates decay for a linear problem. 

\begin{lemma}
  \label{th:reg2}
  Let $H$ be a homogeneous finite difference operator of the form
  \eqref{eq:reg:prelims:explicit_H_2} satisfying the stability condition
  \eqref{eq:reg:lgf:stab_H}. Let $u \in \UsH(\mA\Z^d)$ satisfy
  \begin{displaymath}
    \< Hu, v \> = \< g, Dv \> \quad \text{where} \quad
    \cases{ & g : \mA \Z^d \to (\R^m)^\Rg, \\
      & |D^i g(\ell)| \leq C (1+|\ell|)^{-p-i}, \quad i = 0, \dots, j-1,
    }
  \end{displaymath}
  where $p > d$ and $j \geq 0$.  Then, for $i = 1, \dots, j$ and
  $\bfrho \in \Rghom^i$, there exists $C > 0$ such that
  \begin{displaymath}
    |D_{\bfrho} u(\ell)| \leq C (1+|\ell|)^{1-d-i}.
  \end{displaymath}
\end{lemma}
\begin{proof}
  The proof is a straightforward application of the decay estimates for the
  Green's function.  For the sake of brevity, we shall only carry out the
  details for the case $j = 2$. This will reveal immediately how to proceed for
  $j > 2$.
  
  For all $\ell \in \mA\Z^d$, $\vsig,\vsig' \in \Rg$, we have
  \begin{equation}
    \label{eq:reg2:10}
    D_\vsig D_{\vsig'} u(\ell) =  - \sum_{k\in \mA\Z^d} \sum_{\rho \in
      \Rg} D_{\vsig'}D_\vsig D_\rho \Gr(k)\, g_\rho(\ell-k).
  \end{equation}
  We again split the summation over $|k| \leq |\ell|/2 =: r$ and $|k| >
  r$. In the set $|k| > r$ the estimate is a simplified
  version (due to the absence of the nonlinearity) of {\it Step~1b} in
  the proof of Lemma \ref{th:reg1}, which yields
  \begin{displaymath}
    \bg|\sum_{|k| > r} \sum_{\rho \in
      \Rg} D_{\vsig'}D_\vsig D_\rho \Gr(k)\, g_\rho(\ell-k) \bg|
    \leq C r^{-1-d}.
  \end{displaymath}

  In the set $|k| < r$, we carry out a summation by parts,
  \begin{align}
    \notag
    \sum_{|k| \leq r} \sum_{\rho \in \Rg} D_{\vsig'} D_\vsig D_\rho
    \Gr(k)\, g_\rho(\ell-k) &= \sum_{|k| \leq r + |\vsig'|}
    \chi_{r,\vsig'}(k) D_\vsig D_\rho \Gr(k) D_{-\vsig'}
    g_\rho(\ell-k) \\ & \qquad +
    \label{eq:reg2:20}
    \sum_{r-|\vsig'| \leq |k| \leq r+|\vsig'|} \nu_{r,\vsig'}(k)
    D_{\vsig} D_\rho \Gr(k)\, g_\rho(\ell-k),
  \end{align}
  where $\chi_{r,\vsig'}(k), \nu_{r,\vsig'}(k) \in \{-1,0,
  1\}$. To see this, consider two discrete functions $a, b$ and
    the characteristic function $\chi(k) = 1$ if $|k| \leq r$ and
    $\chi(k) = 0$ otherwise. Then,
    \begin{align*}
      \sum_{|k| \leq r} \b(D_\tau a(k)\b) b(k) 
      &=
        \sum_{k  \in \L} \b(D_\tau a(k)\b) b(k) \chi(k)
        = 
        \sum_{k  \in \L} a(k) D_{-\tau}( b(k) \chi(k) ) \\
      &= 
        \sum_{k  \in \L} a(k) D_{-\tau} b(k) \chi(k+\tau) 
        + \sum_{k \in \L} a(k) b(k) D_{-\tau} \chi(k).
    \end{align*}
    This establishes the claim that the coefficients
    $\chi_{r,\vsig'}, \nu_{r,\vsig'}$ belong indeed to $\{-1,0,1\}$.

  The summation over $|k| \leq r+|\vsig'|$ can be bounded using a simplified
  variant of the estimates in {\it Step 1a} of the proof of Lemma \ref{th:reg1}
  and the decay assumption for $g$. This yields
  \begin{displaymath}
    \bg|\sum_{|k| \leq r + |\vsig'|}
    \chi_{r,\vsig'}(k) D_\vsig D_\rho \Gr(k) D_{-\vsig'}
    g_\rho(\ell-k)\bg| \leq C r^{-1-d}.
  \end{displaymath}
  
  The ``boundary terms'' in \eqref{eq:reg2:20} (second group on the right-hand
  side) are estimated by
  \begin{align*}
    \hspace{2cm} &\hspace{-2cm} \bg| \sum_{r-|\vsig'| \leq |k| \leq r+|\vsig'|} \nu_{r,\vsig'}(k)
    D_{\vsig} D_\rho \Gr(k)\, g_\rho(\ell-k) \bg| \\
    &\leq C
    \sum_{r-|\vsig'| \leq r \leq r+|\vsig'|} (1+|k|)^{-d}
    (1+|\ell-k|)^{-p} \\
    &\leq C r^{d-1} (1+r)^{-d-p} \leq C (1+r)^{-p-1} \leq C (1+r)^{-d-1}
  \end{align*}
  Thus, in summary, we can conclude that
  \begin{displaymath}
    \bg|\sum_{\substack{k \in \mA \Z^d \\ |k| \leq r}} \sum_{\rho \in \Rghom} D_{\vsig'} D_\vsig D_\rho
    \Gr(k)\, g_\rho(\ell-k)\bg| \leq C (1+r)^{-d-1}.
  \end{displaymath}
  
  The only modification for the case $j > 2$ is that $j-1$
    summation by part steps are required instead of a single one.
  This completes the proof of Lemma \ref{th:reg2}.
\end{proof}

\begin{proof}[Proof of Theorem \ref{th:pt:regularity}, Case $j \geq
  2$]
  The statement of Theorem \ref{th:pt:regularity}, Case $j \geq 2$, is
  an immediate corollary of Lemmas \ref{th:reg:higher_res_pt} and
  \ref{th:reg2}.
\end{proof}

\subsection{Proof of Theorem \ref{th:disl:regularity}, Case $j = 1$}
\label{sec:regdisl:lin+res}
We now adapt the arguments of the foregoing sections to the dislocation
case. Remembering that $D\up(\ell) \not\to 0$ as $|\ell| \to \infty$ we begin by
recalling the definitions of $e = \Del_0 \up$ and $\Del u$ from
\S~\ref{sec:prf:elastic_strain}, noting that $|e(\ell)|\lesssim |\ell|^{-1}$.

Let $u := \ua$, $v \in \Usz$ and $|\ell|$ sufficiently large, then
\eqref{eq:disl:slip_delV} yields
\begin{align*}
  \b\< \del V(D(\up+u)(\ell)), Dv(\ell) \b\> &= \b\< \del V(e+\Del u(\ell)), \Del 
  v(\ell)\b\> \\
  &= \b< \del V(e+\Del u) - \del V(e) - \ddel V(e) \Del  u, \Del  v \b\> \\
  & \qquad +
  \b\< (\ddel V(e) - \ddel V(\bfO) )\Del  u, \Del  v \b\> \\
  &\qquad + \b<\ddel V(\bfO) \Del  u, \Del  v \b\> + \< \del V(e),
  \Del  v \>.
\end{align*}
Upon defining the linear operator
\begin{equation}
  \label{eq:reg:disl:Htilde}
  \< \tilde{H} v, w \> := \sum_{\ell \in \L} \< \ddel V(\bfO) \Del  u, \Del  
  v \>, \qquad \text{for } v, w \in \UsH(\L),
\end{equation}
we obtain that 
\begin{equation}
  \label{eq:reg:linearisation_disl}
  \begin{split}
    \< \tilde{H} u, v \> 
    &= \sum_{\ell \in \L} \B(
    \b< \del V(e) + \ddel V(e) \Del  u - \del V(e+\Del u), \Del  v \b\> \\
    & \qquad \qquad + \b\< (\ddel V(\bfO) - \ddel V(e))\Del  u, \Del  v
    \b\>\B) - \< \del\E(0), v \>.
  \end{split}
\end{equation}

We can now generalise Lemma \ref{th:reg:lin_eqn} as follows.

\begin{lemma}[First-Order Residual Estimate, Dislocations]
  \label{th:regdisl:residuals}
  Under the conditions of Theorem~\ref{th:disl:regularity} there exists
  $g : \L \to (\R^d)^{\Rg}$ and constants $C_1, R_1$ such that
  \begin{align*}
    \< \tilde{H} \ua, v \> &= \< g, \Del v \> \qquad \forall v \in
    \Usz, \quad \text{where} \\
    |g(\ell)| &\leq C_1 \b( |\ell|^{-2} + |\Del  \ua(\ell)|^2 \b)
    \qquad \text{for } |\ell| \geq R_1.
  \end{align*}
\end{lemma}
\begin{proof}
  Setting again $u = \ua$, we can write
  \begin{align}
    \< \tilde{H} u, v \>
    \notag
    &= 
      \sum_{\ell \in \L} \B( \b\< (\ddel V(\bfO) - \ddel V(e))\Del u, \Del v
      \b\>
    \\[-3mm]
    \notag
    & 
      \qquad \qquad + \b< \del V(e) + \ddel V(e) \Del u - \del
      V(e+\Del u), \Del v \b\>  \B) - \< \del\E(0), v \> \\
    \label{eq:prf:disl:tilH-residual-eqn}
    &=: 
      \<g^{(1)} + g^{(2)}, \Del v \> - \< f, v \>,
  \end{align}
  where we employed Lemma \ref{th:disl:delE0} in the last step.
  
  {\it The $\<f, v\>$ group: } 
  The decay $|f(\ell)| \lesssim |\ell|^{-3}$ implies that also
  $|Sf(\ell)| \lesssim |\ell|^{-3}$, hence Corollary \ref{th:Dg_f} implies the
  existence of $g^{(3)}$, $|g^{(3)}(\ell)|\lesssim |\ell|^{-2}$ such that
  \begin{displaymath}
    \< \del\E(0), v \> = \< f, v \> = \< g^{(3)}, \Del v \>.
  \end{displaymath}

The first two groups are linearisation errors and it is easy to see
that, for $|\ell| \geq R_1$, with $R_1$ chosen sufficiently large, we
have
\begin{align*}
  \b|g^{(1)}(\ell)\b| \leq C |\ell|^{-1} |\Del u(\ell)| \quad
  \text{and} \quad \b| g^{(2)}(\ell) \b| \leq C |\Del u(\ell)|^2.
\end{align*}

Setting $g := g^{(1)} + g^{(2)} - g^{(3)}$ we obtain that stated result.
\end{proof}

An obstacle we encounter trying to extend the regularity proofs in the point
defect case (Lemma \ref{th:reg1} and Lemma \ref{th:reg2}) are the ``incompatible
finite difference stencils'' $\Del u(\ell)$, which occur in
\eqref{eq:reg:disl:Htilde}.  Interestingly, we can bypass this obstacle without
concerning ourselves too much with their structure, but instead using a
relatively simple boot-strapping argument starting from the following
sub-optimal estimate.

\begin{lemma}[Suboptimal estimate for $\Del u$]
  \label{th:regdisl:subopt}
  Under the conditions of Theorem \ref{th:disl:regularity}, there
  exists $R_1 > 0$ such that
  \begin{displaymath}
    |\Del  \ua(\ell)| \leq C |\ell|^{-1} \quad \text{ for all }
    |\ell| > R_1.
  \end{displaymath}
\end{lemma}
\begin{proof}
  In the following let $u := \ua$,
  $s_1 := \smfrac12 |\ell|-\rcut, s_2 := \smfrac12 |\ell|$ and assume
  that $|\ell|$ is always large enough so that
  $s_1 \geq \smfrac13|\ell| \geq \rdisl+|\burg_{12}|$.

  We first consider the case that $B_{\frac34 |\ell|}(\ell)$ does not
  intersect $\Gamma$. We will then extend the argument to the case
  when it does intersect.

  Let $\eta_1$ be a cut-off function with $\eta_1(x) = 1$ in
  $B_{s_1/2}(\ell)$, $\eta_1(x) = 0$ in $\R^2 \setminus B_{s_1}(\ell)$
  and $|\nabla \eta_1| \leq C |\ell|^{-1}$. Further, let $v(k) :=
  D_\tau \Gr(k - \ell)$, where $\Gr$ is the lattice Green's function
  associated with the homogeneous finite difference operator $H$
  defined in \eqref{eq:reg:prelims:explicit_H}. Then,
  \begin{equation}
    \label{eq:regdisl:prf:20}
    D_\tau u(\ell) = \< H u, v \> = \< Hu, [\eta_1 v]\> + \< H u,
    [(1-\eta_1)v] \>
  \end{equation}
  where $\eta_1 v, (1-\eta_1) v$ are understood as pointwise function
  multiplication.

  For the first group in \eqref{eq:regdisl:prf:20}, and assuming that
  $|\ell|$ is sufficiently large, $B_{3|\ell|/4}(\ell)$ does
  not intersect the branch-cut $\Gamma$, hence we have
  \begin{align*}
     \< Hu, [\eta_1 v]\> &= \< \tilde{H} u, [\eta_1 v] \> = \< g, D [\eta_1
     v] \>  \\
     &\lesssim \sum_{k \in B_{s_2}(\ell)} \b( |k|^{-2} +
     |D u(k)|^2 \b) \, \b| D[\eta_1 v](k) \b|.
  \end{align*}
  Using the decay estimates for $\Gr$ established in Lemma
  \ref{th:green_fcn} and the assumptions on $\eta_1$ it is
  straightforward to show that $| D[\eta_1 v](k) | \lesssim
  (1+|\ell-k|)^{-2}$, and hence we can continue to estimate
  \begin{align}
    \notag
    \b|\< Hu, [\eta_1 v]\>\b| &\lesssim \sum_{k \in
      B_{s_2}(\ell)} \b( |\ell|^{-2} + |D u(k)|^2 \b) \,
    (1+|\ell-k|)^{-2} \\
    \label{eq:regdisl:prf:25}
    &\lesssim |\ell|^{-2} \log |\ell| + \b\| \psi_\ell
    D u \b\|_{\ell^2(\L \cap B_{s_2}(\ell))}^2,
  \end{align}
  where $\psi_\ell(k) := (1+|\ell-k|)^{-1}$.
  
  To estimate the second group in \eqref{eq:regdisl:prf:20} we note
  that
  \begin{displaymath}
    D_\rho [(1-\eta_1) v](k) = -D_\rho \eta_1(k) A_\rho
    v(k) + A_\rho (1-\eta_1)(k) D_\rho v(k),
  \end{displaymath}
  where $A_\rho w(k) = \smfrac12 (w(k) + w(k+\rho))$. We first note
  that the first term on the right-hand side is only non-zero for $s_1
  \geq |\ell-k| \geq s_1/4$, while the second term on the right-hand side is
  only non-zero for $|\ell-k| \geq s_1 / 4$, both provided that $|\ell|$ is
  sufficiently large.  Applying the bounds for $\eta_1$ and $\Gr$
  again, we therefore obtain that
  \begin{align*}
    \b|D [(1-\eta_1) v](k)\b| \lesssim |\ell|^{-1} |\ell-k|^{-1}
    \chi_{[s_1/4, s_1]}(|\ell-k|) + |\ell-k|^{-2} \lesssim |\ell-k|^{-2}.
  \end{align*}
  Thus, we can estimate
  \begin{align*}
    \b| \b\< H u, [(1-\eta_1) v] \b\>\b| 
    &\lesssim \sum_{|k-\ell| > s_1/4} |D u(k)|\, |\ell-k|^{-2} \\
    &\lesssim \| D u \|_{\ell^2}  \bg(\sum_{|k-\ell| > s_1/4} |\ell-k|^{-4} \bg)^{1/2}
    \lesssim |\ell|^{-1}.
  \end{align*}

  To summarize the proof up to this point, we have shown that, if
  $|\ell|$ is sufficiently large and if $B_{3|\ell|/4}(\ell) \cap
  \Gamma = \emptyset$, then 
  \begin{equation}
    \label{eq:regdisl:prf:30}
    \b| D u(\ell) \b| \leq C \B( |\ell|^{-1} + \b\|
    \psi_\ell Du \b\|_{\ell^2(\L \cap B_{s_2}(\ell))}^2 \B).
  \end{equation}

  Next, we extend the argument to the case when
  $B_{3|\ell|/4}(\ell) \cap \Gamma \neq \emptyset$. We shall, in fact, present
  two different (but closely related) arguments in order to motivate the
  remaining proofs in this section.

  {\it (1) Algebraic Manipulations: } Consider now the case $\ell_1 > 0$ and
  recall that $\Del_{\tau} u(\ell) = R D_{\tau} Su(\ell)$. Let $v, \eta_1$ be
  defined as before, then we have again
  \begin{align*}
    S \Del_{\tau} u(\ell) = D_\tau Su(\ell) = \< HSu, v\>
    = \< HSu, [\eta_1 v] \> + \< HSu, [(1-\eta_1) v] \>.
  \end{align*}
  The estimate for the second group is identical as above; we obtain
  \begin{displaymath}
    \b|\< HSu, [(1-\eta_1) v] \>\b| \lesssim |\ell|^{-1}.
  \end{displaymath}
  Since, in the support of $\eta_1$, we have $\Del=RDS$, the first group now
  rewrites upon defining $B := \ddel V(\bf0)$ as
  \begin{align*}
    \< HSu, [\eta_1 v] \> 
    &=
      \sum \b\< B DSu, D[\eta_1 v] \b\>
      = 
      \sum \b\< B [RDS]u, [RDS]R[\eta_1v] \b\> \\
    &=
      \sum \b\< B \Del u, \Del R[\eta_1 v] \b\> 
      =
      \< \tilde{H} u, R[\eta_1v]\> \\
    &=
      \< g, \Del R[\eta_1v] \> = \< g, R D[\eta_1 v] \>.
  \end{align*}
  We can now argue analogously as in case $\ell_1 < 0$, to deduce that
  \begin{displaymath}
    \b| \< HSu, [\eta_1 v] \> \b| \lesssim |\ell|^{-2} \log|\ell| 
    + \|\psi_\ell \Del u \|_{\ell^2(\L \cap B_{s_2}(\ell))}^2.
  \end{displaymath}
  Thus, we have so far proven that
  \begin{equation}
    \label{eq:regdisl:prf:30_Duprime}
    \b| \Del  u(\ell) \b| \leq C \B( |\ell|^{-1} + \b\|
    \psi_\ell \Del u \b\|_{\ell^2(\L \cap B_{s_2}(\ell))}^2 \B)
    \qquad \forall \ell \in \L, \text{sufficiently large.}
  \end{equation}
  
  {\it (2) Reflection Argument: } An introspection of the previous paragraph
  indicates that, what we have in fact done is to derive an equation for $Su$
  which has identical structure to the equation satisfied by $u$, except that
  the branch-cut $\Gamma$ has been replaced with
  $\Gamma_S := \{ (x_1, \hat{x}_2) \sep x_1 \leq \hat{x}_1 \}$. We can therefore
  argue, much more briefly, as follows:
  
  According to Remark \ref{rem:reflection}, we have $\del\E_S(Su) = 0$ (recall that
  in the definition of $\E_S$ we have replaced $\up$ with $S_0 \up$). This new
  problem is structurally identical to $\del \E(u) = 0$, except that the
  branch-cut $\Gamma$ is now replaced with $\Gamma_S$. Therefore, it follows
  that \eqref{eq:regdisl:prf:30} holds, but $u$ replaced with $Su$ and for all
  $\ell_1 > \hat{x}_1$, $|\ell|$ sufficiently large. It is now immediate to see
  that we can replace $DSu$ with $RDSu = \Del u$ without changing the
  estimate. Thus we obtain again \eqref{eq:regdisl:prf:30_Duprime}.  

{\it Conclusion: } We now consider arbitrary $\ell$.  We rewrite \eqref{eq:regdisl:prf:30_Duprime}
in a way that allows us to apply the argument similar of Step 2 in the proof of
Lemma \ref{th:reg1}. We begin by noting that
  \begin{equation}
    \label{eq:regdisl:prf:40}
    \b\|
    \psi_\ell \Del  u \b\|_{\ell^2(\L \cap B_{s_2}(\ell))}^2
    \leq \| \psi_\ell \|_{\ell^4(B_{s_2}(\ell))}^2 \, \| \Del  u
    \|_{\ell^2(B_{s_2}(\ell))} \| \Del  u \|_{\ell^\infty(B_{s_2}(\ell))}.
  \end{equation}
  Fix $\epsilon > 0$, then there exists $r_0 > 0$ such that $\| \Del u
  \|_{\ell^2(B_{s_2}(\ell))} \leq \epsilon$, whenever $|\ell| \geq
  r_0$.

  Let $w(r) := \max_{|k| \geq r} |\Del  u(k)|$, then
  \eqref{eq:regdisl:prf:30_Duprime} and \eqref{eq:regdisl:prf:40}
  imply that
  \begin{displaymath}
    w(2r) \leq C\b( r^{-1} + \epsilon w(r) \b) \qquad \text{ for } r
    \geq r_0.
  \end{displaymath}
  We can now apply the argument of {\it Step 2} in the proof of Lemma
  \ref{th:reg1} to obtain that $w(r) \lesssim r^{-1}$ and hence
  $|\Del u(\ell)| \lesssim |\ell|^{-1}$.
\end{proof}

Having established a preliminary pointwise decay estimate on $\Del \ua$, we now
apply a boot-strapping technique to obtain an optimal bound.

\begin{proof}[Proof of Theorem \ref{th:disl:regularity}, Case $j =
  1$]
  In view of Remark \ref{rem:reflection} (cf. part (2) in the proof of Lemma
  \ref{th:regdisl:subopt}) we may assume, without loss of generality, that
  $\ell$ belongs to the left half-plane, i.e., $\ell_1 < \hat{x}_1$. We again
  define $v$ and $\eta_1$ as in the proof of Lemma \ref{th:regdisl:subopt}, and
  $B := \ddel V(\bf0)$, to write
  \begin{align*}
    D_\tau u(\ell) 
    &= 
      \< H u, v \> = \sum_{k \in \L} \< B Du(k), Dv(k) \> \\
    &= 
      \sum_{k \in \L} \< B \Del u(k), \Del v(k) \> 
      + \sum_{k \in \L} \B( \< B Du(k), Dv(k) \> 
      - \< B \Del u(k), \Del v(k) \>  \B) \\
    &=: {\rm T}_1 + {\rm T}_2.
  \end{align*}

  To estimate the first group we note that ${\rm T}_1 = \< g, \Del v\>$, hence
  we can employ the residual estimates from Lemma
  \ref{th:regdisl:residuals}. Combining Lemma \ref{th:regdisl:residuals} with
  Lemma \ref{th:regdisl:subopt} we have $|g(k)| \lesssim |k|^{-2}$, which
  readily yields
  \begin{align*}
    \b| {\rm T}_1 \b| 
    \leq
    \sum_{k \in \L} (1+|k|)^{-2} (1+|\ell-k|)^{-2}  
    \lesssim 
    |\ell|^{-2} \log |\ell|.
  \end{align*}
  Here, we used the observation that 
  \begin{displaymath}
    |\Del_\rho D_\tau G(k-\ell)| \lesssim (1+|k-\ell|)^{-2}
  \end{displaymath}
  due to the fact that $D_\rho S w(k) = D_{\rho'} w(k')$, where
  $|\rho-\rho'| + |k-k'| \lesssim 1$.
  
  To estimate ${\rm T}_2$, we observe that $\Del  w(k) = D w(k)$ for all
  $k \in \L \setminus U_\Gamma$, where we define $U_\Gamma$ to be a
  discrete strip surrounding $\Gamma$,
  $U_\Gamma := \L \cap (\Gamma + B_{\rcut})$. Thus, employing again
  Lemma \ref{th:regdisl:subopt}, 
  \begin{align*}
    \b|{\rm T}_2\b| \lesssim \sum_{k \in U_{\Gamma}} (1+|k|)^{-1}
    (1+|\ell-k|)^{-2} \lesssim |\ell|^{-2} \log |\ell|,
  \end{align*}
  where the final inequality crucially uses the fact that $\ell_1 < \hat{x}_1$,
  which implies that $|\ell-k| \gtrsim |\ell|+|k|$.
\end{proof}

\subsection{Proof of Theorem \ref{th:disl:regularity}, Case $j > 1$}
\label{sec:prf:disl:jgt1}
In view of case $j = 1$ and also of Lemma \ref{th:reg2}(b) it is natural to
conjecture that
\begin{displaymath}
  |\Del^i u(\ell)| \lesssim |\ell|^{-i-1} \log |\ell|.
\end{displaymath}
Suppose that we have proven this for $i = 1, \dots, j-1$. Then the triangle
inequality immediately yields
\begin{displaymath}
  |\Del^j u(\ell)| \lesssim |\ell|^{-j} \log |\ell|,  
\end{displaymath}
which is of course sub-optimal, but it allows us again to apply a bootstrapping
argument. In the dislocation case, this requires two steps, corresponding to
cases (a) and (b) of the following lemma.

\begin{lemma}[Residual Estimates]
  \label{th:prf:disl:higher-order-residual}
  Assume the conditions of Theorem \ref{th:disl:regularity} hold.

  (a) Suppose, further, that $2 \leq j \leq k-2$ and that there exist
  $C_1, R_1 > 0$ such that
  \begin{displaymath}
    |\Del^{i}  \ua(\ell)| \leq C_1 |\ell|^{-i-1} \log |\ell| \qquad \quad
    \text{for } 1 \leq i \leq j-1,  |\ell| \geq R_1,
  \end{displaymath}
  then there exists $g : \L \to (\R^3)^\Rg$ and $C_2, R_2$ such that
  \begin{align*}
    \< \tilde{H} \ua, v \> 
    &= 
      \< g, \Del v \>, \qquad \quad \text{where, for $|\ell| \geq R_2$, } \\
    |g(\ell)| &\leq C_2 |\ell|^{-2}, \\
    |\Del^{i} g(\ell)| 
    &\leq 
      C_2 |\ell|^{-2-i} \qquad \text{for } 
      i = 1, \dots, j-2, \quad \text{and} \\
    |\Del^{j-1} g(\ell)| 
    &\leq 
      C_2  |\ell|^{-1-j} \log |\ell|.
  \end{align*}

  (b) If, in addition, we also have that
  $|\Del^j \ua(\ell)| \leq C_1 |\ell|^{-j}$, then
  \begin{displaymath}
    |\Del^{j-1} g(\ell)| \leq C_2 |\ell|^{-1-j} \qquad \text{for }
    |\ell| \geq R_2.
  \end{displaymath}
\end{lemma}
\begin{proof}
  Many estimates in this proof are very similar to estimates that we have proven
  in previous results, hence we only give a brief outline. We begin by setting
  again $u \equiv \ua$ and recalling from \eqref{eq:prf:disl:tilH-residual-eqn}
  that
  \begin{align*}
    &\< \tilde{H}u, v \> = \< g^{(1)} + g^{(2)}, \Del v \> - \< f, v \>,
      \qquad \text{where} \\[2mm]
    & g^{(1)} = \b(\ddel V(\bfO) - \ddel V(e) \b) \Del u, \qquad
      g^{(2)} = \del V(e) - \ddel V(e) \Del u - \del V(e+\Del u),
  \end{align*}
  and $f$ is given by \eqref{eq:disl:delE0-f}. We now analyze the terms
  $g^{(j)}$ and $f$ in turn.
  
  {\it The term $g^{(1)}$: } Let $\ell_1 > \hat{x}_1$ (the case
  $\ell_1 \leq \hat{x}_1$ can be treated by a simplified argument).  Let
  $\alpha_1, \dots, \alpha_i \in \Rg$, $\bfrho \in \Rg^2$, then
  \begin{align*}
    \Del_{\alpha_1} \cdots \Del_{\alpha_i} V_{\bfrho}(e(\ell))
    &= R D_{\alpha_1} \cdots D_{\alpha_i} S V_{\bfrho}(RD S_0 \up(\ell)) \\
    &= R D_{\alpha_1} \cdots D_{\alpha_i} V_{\bfrho}(D S_0 \up(\ell)).
  \end{align*}
  Applying Lemma \ref{th:disl:up_lemma}(iii) it is easy to show that for
  $|\ell|$ sufficiently large,
  \begin{displaymath}
    \label{eq:eq:app_prf:g_dis:10}
    \b| \Del_{\alpha_1} \cdots \Del_{\alpha_i} V_{,\bfrho}(e(\ell)) \b| \leq C
    |\ell|^{-i-1} \quad \text{for } i \geq 1, \quad \alpha_i \in \Rg,
    \bfrho \in \Rg^2.
  \end{displaymath}
  Hence, and recalling the discrete product formula
  \eqref{eq:disc_product_rule}, we obtain in {\it case (a)}
  \begin{align}
    \notag
    \b| \Del_{\alpha_1} \cdots \Del_{\alpha_i} g^{(1)}(\ell) \b| 
    &\lesssim 
      |\ell|^{-i-3} \log |\ell| + |\ell|^{-1} |\Del^{i+1} u(\ell)| \\[1mm]
    \notag
    &\lesssim 
      \cases{ 
            |\ell|^{-i-2} + |\ell|^{-i-3} \log |\ell|, & i \leq j-2, \\
            |\ell|^{-i-2} + |\ell|^{-i-2} \log |\ell|, & i = j-1
      }
                                                         \\[1mm]
    \label{eq:app_prf:g_dis:modification_for_case_c}
    &\lesssim 
      \cases{
        |\ell|^{-i-2}, & i \leq j-2, \\
        |\ell|^{-i-2} \log |\ell|, & i = j-1.
      }
  \end{align}

  %
  In {\it case (b)} of the foregoing calculation, the log-factor in the $i=j-1$
  case is dropped, hence we then obtain the improved estimate
  $| \Del_{\alpha_1} \cdots \Del_{\alpha_j} g^{(1)}(\ell) | \lesssim
  |\ell|^{-1-j}$.

  {\it The term $g^{(2)}$: } The higher-order estimate for the term $g^{(2)}$
  can be performed very similarly as in the point defect case in
  \S~\ref{sec:reg:higher_pt}, but expanding about $e$ instead of
  $\bfO$. Applying $|\Del^i e(\ell)| \lesssim |\ell|^{-i-1}$, the hypothesis
  $|\Del^i u| \lesssim |\ell|^{-1-i}\log |\ell|$ and Lemma
  \ref{th:disl:up_lemma}(iii), and hence arguing analogously as in
  \S~\ref{sec:reg:higher_pt} we obtain
  \begin{displaymath}
    \b|\Del_{\alpha_1} \cdots \Del_{\alpha_i} g^{(2)}(\ell) \b| \lesssim
    |\ell|^{-i-4} \log^2 |\ell| + |\ell|^{-2} \log |\ell| |\Del^{i+1}
    u(\ell)| \b)  \lesssim |\ell|^{-2-i}.
  \end{displaymath}

  {\it The term $f$: } Recall from the proof of Lemma
    \ref{th:regdisl:residuals} that there exists $g^{(3)}$ such that
    $|g^{(3)}(\ell)| \lesssim |\ell|^{-2}$ and $\Del g^{(3)} = f$.  Setting
    $g = g^{(1)} + g^{(2)} - g^{(3)}$ this already completes the proof of the
    case $j = 2$. Applying Lemma \ref{th:disl:delE0}
    $|\Del^{i-1} g^{(3)}| \lesssim |\ell|^{-i-1}$.

  {\it Conclusion: } Summarising the estimates for difference operators applied
  to $g^{(1)}, g^{(2)}, \tilde{g}^{(3)}$ and choosing
  $\tilde{g} = g^{(1)} + g^{(2)} - g^{(3)}$ we obtain both of the decay
  estimates claimed in parts (a) and (b)
\end{proof}

\begin{proof}[Proof of Theorem \ref{th:disl:regularity}, Case $j > 1$]
  By induction, suppose that 
  \begin{equation}
    \label{eq:prf:disl:regfinal:10}
    |\Del^i \ua(\ell)| \lesssim
    |\ell|^{-i-1} \log |\ell| \text{ for } i = 1, \dots, j-1.     
  \end{equation}
  and consequently also 
  \begin{displaymath}
    |\Del^j \ua(\ell)| \lesssim |\ell|^{-j-2} \log^r |\ell|, 
  \end{displaymath}
  with $r = 1$. However, suppose more generally that $r \in \{0, 1\}$.
  
  In the following we assume again, without loss of generality, that
  $\ell_1 < \hat{x}_1$ (cf. Remark \ref{rem:reflection} and proof of Lemma
  \ref{th:regdisl:subopt}), and further that $|\ell|$ is sufficiently large.

  Let $u := \ua$, $\bfrho \in \Rg^j$ and let $v(k) := D_\bfrho \Gr(k - \ell)$,
  then
  \begin{align}
    \notag
    D_\bfrho u(\ell) 
    &= \< H u, v \> 
      = \< \tilde{H} u, v \> + \< (H-\tilde{H}) u, v \> \\
    \label{eq:regdisl:prf:70}
    &= \< g, \Del v \>  + \sum_{\ell \in U_\Gamma} \B(
      \< B Du, Dv \> - \< B \Del u, \Del v \> \B) \\
    \notag
    &=: {\rm T}_1 + {\rm T}_2.
  \end{align}

  {\it The term ${\rm T}_2$} can be estimated analogously as in the proof of the
  case $j = 1$ in \S~\ref{sec:regdisl:lin+res}, noting that by the same argument
  as used there,
  $|\Del D_{\bfrho} \mathcal{G}(k-\ell)| \lesssim |k-\ell|^{-j-1}$. Thus, one
  obtains
  \begin{displaymath}
    \b| {\rm T}_2 \b| \lesssim |\ell|^{-j-1} \log |\ell|.
  \end{displaymath}

  {\it The term ${\rm T}_1$: } First, we split
  \begin{displaymath}
    \< g, \Del v\> = \sum_{|k-\ell| \leq |\ell|/2} \< g(k), \Del v(k) \> 
    + \sum_{|k-\ell| > |\ell|/2} \< g(k), \Del v(k) \> =: {\rm S}_1 + {\rm S}_2.
  \end{displaymath}
  The second term is readily estimated, using $|\Del v(k)| \lesssim
  |\ell-k|^{-j-1}$, by
  \begin{displaymath}
    |{\rm S}_2| \lesssim \sum_{|k-\ell| > |\ell|/2} |k|^{-2} |\ell-k|^{-j-1}
    \lesssim |\ell|^{-j-1} \log |\ell|.
  \end{displaymath}

  To estimate ${\rm S}_1$ we first notice that, provided that $|\ell|$ is chosen
  sufficiently large, this sum only involves values of $g, v$ away from
  $\Gamma$, that is, $\Del \equiv D$ and we can write
  \begin{displaymath}
    {\rm S}_1 = \sum_{|k-\ell| \leq |\ell|/2} \< g(k), D v(k) \>.
  \end{displaymath}
  We are now in a position to mimic the argument of Lemma \ref{th:reg2} almost
  verbatim, only having to take care to take into account the slower decay of
  $g$. Namely, according to the hypothesis stated at the beginning of the
  present proof, and employing Lemma \ref{th:prf:disl:higher-order-residual} we
  have 
  $|D^{i} g(k)| \lesssim |k|^{-i-2} \log^r |k|$. This in turn yields an
  additional log-factor in the estimate
  \begin{displaymath}
    |{\rm S}_1| \lesssim |\ell|^{-j-1} \log^{r+1} |\ell|.
  \end{displaymath}
  In summary, we have $|{\rm T}_1| \lesssim |\ell|^{-j-1} \log^{r+1} |\ell|$.
  
  {\it Conclusion: } Arguing initially with $r = 1$, we obtain from
  the preceding arguments that
  $|D^j u(\ell)| \lesssim |\ell|^{-j-1} \log^2|\ell|$. This initial
  estimate implies that, at the beginning of the proof, we may in fact
  choose $r = 0$, and therefore, we even obtain the improved bound
  $|D^j u(\ell)| \lesssim |\ell|^{-j-1} \log |\ell|$. Recalling that
  we assumed (without loss of generality) $\ell_1 < \hat{x}_1$, so
  that in fact we have
  $|\Del^j u(\ell)| \lesssim |\ell|^{-j-1}\log|\ell|$, this completes
  the proof.
\end{proof}

\section{Proofs: Approximation Results}
\label{sec:prf_approx}

In this section we prove the approximation results formulated in \S\S
\ref{sec:pt:clamped}--\ref{sec:pt:ac} and
\ref{sec:dis:clamped}--\ref{sec:dis:ac}.

\subsection{Preliminaries}

We briefly establish two auxiliary results that will be needed for our
subsequent analysis. The first result is the discrete Poincar\'e inequality on
an annulus.

\begin{lemma}
  \label{th:poincare_annyulus}
  Let $0 < R_1 < R_2$, $\Sigma := \L \cap (B_{R_2} \setminus B_{R_1})$.  Then,
  there exist constants $c_{\rm P}$, $C_{\rm P}$, and $R_{\rm P}$ that depend only on the choice of $\T_\L$ such that, whenever $R_2-R_1 \geq c_{\rm P}$,
  \begin{displaymath}
    \| u - a \|_{\ell^2(\Sigma)} \leq R_2 C_{\rm P}
    \| D u \|_{\ell^2(\Sigma')} \qquad \forall u : \Sigma' \to \R^d,
  \end{displaymath}
  where $\Sigma' := \L \cap (B_{R_2+R_{\rm P}} \setminus B_{R_1-R_{\rm P}})$ and $a := \mint_{B_{R_2}\setminus B_{R_1}} Iu \dx$.
\end{lemma}
\begin{proof}
  Denote $S := B_{R_2} \setminus B_{R_1}$ and $S' := \bigcup_{T\in\T_\L , T\cap S \ne \emptyset} T$.
  We choose $c_{\rm P}$ so that for any $\ell \in \Sigma$ there exists $T\in\T_\L$ such that $\ell\in T$ and $T\subset S$.
  This immediately yields
%
  \begin{displaymath}
    \| u - a \|_{\ell^2(\Sigma)} \leq C \| I(u - a) \|_{L^2(S)}.
  \end{displaymath}
  Then by first using the continuous Poincar\'e inequality on $S$ we get
  \begin{displaymath}
    \| u - a \|_{\ell^2(\Sigma)}
    \leq R_2 C \| \D u \|_{L^2(S)}
    \leq R_2 C \| \D u \|_{L^2(S')}
    .
  \end{displaymath}
  It then remains to choose $R_{\rm P} := \sup_{T\in\T_\L} \diam(T)$ and notice that any $T\subset S'$ has its vertices in $\Sigma'$, hence
  $\|\D u \|_{L^2(S')} \lesssim \| Du \|_{\ell^2(\Sigma')}$.
\end{proof}


Next, we state a quantitative version of the inverse function theorem, which we
adapt from \cite[Lemma B.1]{LuOr:acta}.

\begin{lemma}
  \label{th:app_min_ift}
  Let $X$ be a Hilbert space, $w_0 \in X$, $R, M > 0$, and $E \in
  C^2(B^X_R(w_0))$ with Lipschitz continuous hessian, $\| \ddel E(x) -
  \ddel E(y) \|_{L(X, X^*)} \leq M \| x - y \|_X$ for $x, y \in
  B^X_R(w_0)$.  Suppose, moreover, that there exist constants $c, r >
  0$ such that
  \begin{displaymath}
    \< \ddel E(w_0) v, v \> \geq c \| v \|_X^2, \quad 
    \| \del E(w_0) \|_Y \leq r, \quad \text{and} \quad 2 M r c^{-2} < 1,
  \end{displaymath}
  then there exists a unique
  $\bar{w} \in B^X_{2 r c^{-1}}(w_0)$ with $\del E(\bar{w}) = 0$ and
  \begin{displaymath}
    \< \ddel E(\bar{w}) v, v \> \geq \b(1 - 2 M r c^{-2} \b) c \| v \|_X^2.
  \end{displaymath}
\end{lemma}

In the context of our analysis $E$ will be the energy to be minimised in the
approximate problem, $w_0$ a projection of the solution to the exact problem to
the approximation space $X$, and $B_R^X(w_0)$ is an $O(1)$ neighourhood within
which the approximate problem $\del E(\bar{w}) = 0$ has some regularity. The
stability constant $c$ and the consistency error $r$ determine in which
neighbourhood, namely $2 r c^{-1}$, an approximate solution $\bar{w}$ may be
found. The neighbourhoods that we employ here are exclusively
$\UsH$-neighbourhoods, that is, the solution $\bar{w}$ obtained via the IFT is
locally unique with locality measured in the energy-norm.

\subsection{Clamped boundary conditions}\label{sec:approx:dir}
The discrete Poincar\'e inequality readily yields the following approximation
estimate.

\begin{lemma}
	\label{th:atm:dir:basic_approx}
	Let $\eta \in C^1(\R^d)$ be a cut-off function satisfying $\eta(x) =
	1$ for $|x| \leq 4/6$ and $\eta(x) = 0$ for $|x| \geq 5/6$. For $R >
	0$ we define
	$T_R : (\R^d)^\L \to \Usz(\L)$ by
	\begin{equation}\label{eq:atm:TR}
          T_R u(\ell) := \eta\b( \ell / R \b) (u(\ell) - a_R) \qquad
          \text{where} \qquad a_R := \mint_{B_{5R/6}\setminus B_{4R/6}}
            Iu(x) \dx.
	\end{equation}
	If $R$ is sufficiently large, then $D T_R u(\ell) = Du(\ell)$ for all
        $\ell \in \L \cap B_{R/2}$,
	\begin{align}
		\label{eq:atm:dir:basic_approx}
		\b\| D T_R u - Du \b\|_{\ell^2} \leq~& C \| D u \|_{\ell^2(\L \setminus B_{R/2})},
		\qquad\text{and}
		\\
		\label{eq:atm:dir:basic_stab}
		\b\| D T_R u \b\|_{\ell^2}
		\leq
		~&
		C \| D u \|_{\ell^2(\L \cap B_R)}
		,
	\end{align}
	where $C$ is independent of $R$ and $u$.
\end{lemma}
\begin{proof}
	We start with expressing
	\begin{align*}
		D_\rho T_R u(\ell)
		=~&\hphantom{\displaystyle (1-\mathstrut)}
		\eta(\ell+\rho) D_\rho u(\ell)
		+
		D_\rho \eta(\ell) \, (u(\ell)-a_R)
		\qquad\text{and} \\
		D_\rho u(\ell) - D_\rho T_R u(\ell)
		=~&
		(1-\eta(\ell+\rho)) D_\rho u(\ell)
		-
		D_\rho \eta(\ell) \, (u(\ell)-a_R).
	\end{align*}
	It then remains to (i) take an $\ell^2$ norm, considering that
        $\eta(\ell+\rho)=0$ for $|\ell|>R$ and $1-\eta(\ell+\rho)=0$ for
        $|\ell|<B_{R/2}$ when $R \geq 6\rcut$, (ii) use that
        $|D_\rho \eta(\ell/R)| \leq C R^{-1}$, (iii) apply the discrete
        Poincar\'e inequality (Lemma \ref{th:poincare_annyulus}), and (iv)
        enforce $R$ large enough so that $\frac56 R + \rcut+ R_{\rm P} \leq R$
        and hence $\frac46 R - \rcut - R_{\rm P} \geq \frac12 R$.
\end{proof}

\begin{proof}[Proof of Theorem \ref{th:pt:clamped}]
	Let $w_R := T_R \ua$. Since $D w_R \to D\ua$ as $R\to\infty$ strongly in
	$\ell^2$ and $\E \in C^2$, we can conclude that $\ddel\E(w_R)
	\to \ddel\E(\ua)$ in the operator norm. Therefore, for $R$
	sufficiently large,
	\begin{displaymath}
		\b\< \ddel\E(w_R) v, v \b\> \geq
		\smfrac12 c_0 \| \D v \|_{L^2}^2 \qquad \forall v \in \Us_0(\Omega_R),
	\end{displaymath}
	where $c_0 > 0$ is the stability constant for $\ua$ from
        \eqref{eq:pt:strongstabeq}. Moreover, it is easy to deduce that
	\begin{align*}
		\b\< \del\E(w_R), v \b\> = \b\< \del\E(w_R) -
		\del\E(\ua), v \b\> 
		\leq~&
		C \| D w_R - D\ua \|_{\ell^2} \| D v \|_{\ell^2}
		\\\leq~&
		C \| D w_R - D\ua \|_{\ell^2} \| \D v \|_{L^2}
		\qquad \forall v \in \Us(\Omega_R).
	\end{align*}

	The inverse function theorem, Lemma \ref{th:app_min_ift}, implies that,
        for $R$ sufficiently large, there exists $\ua^0_R \in \Us(\Omega_R)$,
        which is a strongly stable solution to \eqref{eq:pt:clamped}, and
        satisfies
	\begin{displaymath}
		\| D w_R - D \ua^0_R \|_{\ell^2} \leq C \| D w_R - D\ua \|_{\ell^2}.
	\end{displaymath}
	Applying first Lemma~\ref{th:atm:dir:basic_approx} and then the
	regularity estimate, Theorem~\ref{th:pt:regularity}, yields the first bound in \eqref{eq:pt:clamped:errest}:
	\begin{align} \notag
		\| \D \ua^0_R - \D \ua \|_{L^2}^2
		\leq
                 C \| D \ua^0_R - D \ua \|_{\ell^2}^2
		&\leq
		C \| D w_R - D \ua \|_{\ell^2}^2
		\leq C \| D \ua \|_{\ell^2(\R^d \setminus B_{R/2})}^2
		\\ \label{eq:pt:clamped:10}
		&\leq C \int_{\R^d \setminus B_{R/2}} |x|^{-2d} \dx
		\leq C R^{-d}.
	\end{align}

 The second bound in \eqref{eq:pt:clamped:errest} is a standard corollary: For
 $R$ sufficiently large, $\E$ is twice differentiable along the segment $\{
 (1-s) \ua + s \ua^0_R \sep s \in [0, 1]\}$ and hence
	\begin{align*}
		\b| \E(\ua^0_R) - \E(\ua) \b| &=
		\bg|\! \int_0^1 \B\< \del\E\b( (1-s)
		\ua + s \ua^0_R\b), \ua^0_R - \ua \B\> \ds \bg| \\
		&\hspace{-2cm}= \bg|\! \int_0^1 \B\< \del\E\b( (1-s)
		\ua + s \ua^0_R\b) - \del\E(\ua), \ua^0_R - \ua \B\>  \ds \bg| 
		\leq C \| D\ua^0_R - D \ua \|_{\ell^2}^2.
		\qedhere
	\end{align*}
\end{proof}

\begin{proof}[Proof of Theorem \ref{th:disl:clamped}]
  The previous proof can be repeated almost verbatim, additionally considering
  that (i) restricting $u$ to an open set $\Adm$ does not affect applicability
  of the inverse function theorem, and (ii) using $d=2$ and the dislocation
  regularity estimate in \eqref{eq:pt:clamped:10} yields the
  $C R^{-2} (\log R)^2$ bound.
\end{proof}

\subsection{Periodic boundary conditions for point defects}\label{sec:approx:per}
\label{sec:per}
We start with with a norm equivalence result for $\Us^\per(\Omega_R)$.

\begin{lemma} \label{th:atm:per:norm_equiv} There exist $c, C>0$, independent of
  $R$, such that
\[
c \|\D v\|_{L^2(\omega_R)} \leq \|D v\|_{\ell^2(\Omega_R)} \leq C \|\D v\|_{L^2(\omega_R)} \qquad\text{for all }v\in \Us^\per(\Omega_R).
\]
\end{lemma}
\vspace{-1.5em}
\begin{proof}
  In addition to \eqref{eq:T_Rg_conform} which was needed for the norm
  equivalence in $\Us^{1,2}$, the assertion follows upon not
ing that
  $\|D v\|_{\ell^2(\Omega_R)}$ is supported on
  $\bigcup_{\rho\in\Rg} (\Omega_R + \rho)$ which is contained in a finite
  (independent of $R$) number of periodic images of $\omega_R$.
\end{proof}

The key technical ingredient in the proof for the clamped boundary conditions
was the estimate $\| D T_R \ua - D \ua \|_{\ell^2} \leq C R^{-d/2}$. To obtain a
similar truncation operator we define
$T_R^\per : \UsH(\L) \to \Us^\per(\Omega_R)$ via
\begin{displaymath}
  T_R^\per u(\ell) := T_R u(\ell), \quad \text{for } \ell \in \Omega_R,
\end{displaymath}
and extend it periodically on all of $\Omega_R^\per$.  We then immediately
obtain the same approximation error estimate, as an immediate corollary of
\eqref{eq:atm:dir:basic_approx}.

\begin{lemma}
  \label{th:atm:per:interp_err}
  Let $u \in \UsH(\L)$ and $B_{R+\rcut} \subset \Omega_R$, then
  \begin{equation}
    \label{eq:atm:per:interp_err}
    \b\| D T_R^\per u  - D u \b\|_{\ell^2(\Omega_R)} \leq C \| Du
    \|_{\ell^2(\L \setminus B_{R/2})},
  \end{equation}
  where $C$ is independent of $u$ and $\Omega_R$.
\end{lemma}

\medskip

Using this lemma, we can obtain a consistency estimate.

\begin{lemma}
  \label{th:atm:per:cons}
  Under the assumptions of Theorem \ref{th:pt:per} there exists
  a constant $C$ such that, for all sufficiently large $R$,
  \begin{displaymath}
    \b\< \del\E^\per_R(T_R^\per \ua), v \b\> \leq C R^{-d/2} \| \D v \|_{L^2(\omega_R)}
    \qquad \forall v \in \Us^\per(\Omega_R).
  \end{displaymath}
\end{lemma}
\begin{proof}
  Given a test function $v \in \Us^\per(\Omega_R)$, we construct a test function
  $w \in \Usz(\L)$ by letting $w := T_R Iv$, where we identify $Iv$ with a
  lattice function defined on $\L$.  Hence, by the assumption that
  $\Omega_R \supset B_R$, we have $Dw(\ell)=0$ for all $\ell \notin \Omega_R$.
  Thus,
  \begin{align}
    \notag
    \b\< \del\E^\per_R(T_R^\per \ua), v \b\> &= \b\<
    \del\E_R^\per(T_R^\per \ua), v \b\> - \b\< \del\E(\ua), w \b\>
    \\
    \notag
    &= \sum_{\ell \in \Omega_R} \Big(\b\< \del V(D T_R^\per \ua(\ell)), D v(\ell) \b\>
      - \b\< \del V(D \ua(\ell)), Dw(\ell) \b\>\Big)
    \\
    \label{eq:atm:per:cons:10}
    &=
    \sum_{\ell \in \Omega_R} \b\< \del V(D T_R^\per \ua(\ell)) -
    \del V(D \ua(\ell)), Dv(\ell) \b\>
    \\ \notag
    & \qquad +
    \sum_{\ell \in \Omega_R} \b\<\del V(D \ua(\ell)), Dw(\ell) - Dv(\ell) \b\>
  \end{align}

  The first group on the right-hand side of \eqref{eq:atm:per:cons:10}
  can be estimated, as in the proof of Theorem \ref{th:pt:clamped},
  by
  \begin{align*}
    \sum_{\ell \in \Omega_R}
    \b\< \del V(D T_R^\per \ua(\ell)) -
    \del V(D \ua(\ell)), Dv(\ell) \b\> &\leq C \| D T_R^\per \ua
    - D\ua \|_{\ell^2(\Omega_R)} \|Dv\|_{\ell^2(\Omega_R)} \\[-2mm]
    &\leq C R^{-d/2} \|Dv\|_{\ell^2(\Omega_R)}.
  \end{align*}

  To estimate the second group, we note that $D w=D v$ in $B_{R/2}$, hence
  \begin{align*}
    \sum_{\ell \in \Omega_R} \b\<\del V(D \ua(\ell)), Dw(\ell) - Dv(\ell) \b\>
    =~&
    \sum_{\ell \in \Omega_R\setminus B_{R/2}} \b\<\del V(D \ua(\ell)), Dw(\ell) - Dv(\ell) \b\>
    \\\leq~&
    C \|D \ua\|_{\ell^2(\Omega_R\setminus B_{R/2})} \|Dw-Dv\|_{\ell^2(\Omega_R\setminus B_{R/2})}
    \\\leq~&
    C R^{-d/2} \|Dw-Dv\|_{\ell^2(\Omega_R)}
    .
  \end{align*}

  It now remains to note that
  $\| Dw \|_{\ell^2(\Omega_R)} \leq C \| D v \|_{\ell^2(\L\cap B_R)}$ thanks
  to \eqref{eq:atm:dir:basic_stab}.  Using the norm equivalence, Lemma
  \ref{th:atm:per:norm_equiv}, concludes the proof.
\end{proof}

The second and main challenge for the proof of Theorem \ref{th:pt:per} is that,
since $\Us^\per(\Omega_R) \not\subset \UsH(\L)$, the positivity of
$\ddel\E^\per_R(T_R^\per \ua)$ is not an immediate consequence of positivity of
$\ddel\E(\ua)$ and continuity of $\ddel\E$. Establishing stability requires a
more involved argument, which we provide next.

\begin{theorem}[Stability of Periodic Boundary Conditions]
  \label{th:atm:per:stab}
  Let $\Omega_R$ be a family of periodic computational
  domains satisfying the assumptions of Theorem \ref{th:pt:per}.
  Let $u \in \Us^{1,2}$ and $u_R \in \Us^\per(\Omega_R)$ such that
  $\| Du_R - Du \|_{\ell^\infty(\Omega_R)} \to 0$ as $R \to
  \infty$. 

  For $R$ sufficiently large, the stability constants
  \begin{equation}
    \label{eq:atm:per:stab:1}
    \lambda := \inf_{\substack{v \in \Usz(\L) \\ \|\D v\|_{L^2} = 1}}
    \b\< \ddel\Ea(u) v, v \b\> \qquad \text{and} \qquad
    \lambda_R := \inf_{\substack{v \in \Us^\per(\Omega_R) \\ \| \D v
        \|_{L^2} = 1}} \b\< \ddel\E^\per_{\Omega_R}(u_R) v, v \b\>
  \end{equation}
  satisfy $\lambda_R \to \lambda$ as $R \to \infty$.
\end{theorem}

\medskip 
The proof  relies on two auxiliary lemmas.

\begin{lemma}
  \label{th:decomp_weak_conv}
  Let $w_j \in \UsH(\L)$ such that $Dw_j \rightharpoonup Dw$, weakly
  in $\ell^2$, for some $w \in \UsH$. Then there exist radii $R_j
  \uparrow \infty$ such that, for any sequence $R_j' \uparrow \infty,
  R_j' \leq R_j$,
  \begin{align} \label{eq:decomp_weak_conv_1}
    & D T_{R_j'} w_j \to D w \quad \text{strongly in } \ell^2
    \text{,} \quad  && D w_j - D T_{R_j'} w_j \rightharpoonup 0 \quad
    \text{weakly in } \ell^2,
    \\ \label{eq:decomp_weak_conv_2}
    & \D T_{R_j'} w_j \to \D w \quad \text{strongly in } L^2,
    \quad \text{ and} \quad  && \D w_j - \D T_{R_j'} w_j \rightharpoonup 0 \quad
    \text{weakly in } L^2
    .
  \end{align}
\end{lemma}
\begin{proof}
  We first prove \eqref{eq:decomp_weak_conv_1}.
  Since weak convergence implies strong convergence in finite
  dimensions, it follows that $D w_j(\ell) \to Dw(\ell)$ for all $\ell
  \in \L$. Therefore, $\| Dw_j - Dw \|_{\ell^2(\L \cap B_R)} \to 0$
  for any $R > 0$. Hence, there exists a sequence $R_j \uparrow
  \infty$, such that $\| Dw_j - Dw \|_{\ell^2(\L \cap B_{R_j})} \to
  0$. 

  Then for any $R_j' \leq R_j$
  \begin{align*}
    \| D T_{R_j'}^\per w_j - D w \|_{\ell^2} 
    &= \| D T^\per_{R_j'} w_j - D T^\per_{R_j'} w \|_{\ell^2(\L \cap B_{R_j'})} + \| D
    T^\per_{R_j'} w - D w \|_{\ell^2} \\
    &\leq C \| D w_j - D w \|_{\ell^2(\L \cap B_{R_j'})} + \| D
    T^\per_{R_j'} w - D w \|_{\ell^2} \\
    & \qquad \to 0 \qquad \text{as } j \to \infty,
  \end{align*}
  where, in the transition to the second line we used \eqref{eq:atm:dir:basic_stab}.

  The statements in \eqref{eq:decomp_weak_conv_2} follow directly from
  \eqref{eq:decomp_weak_conv_1} by applying Lemma \ref{th:atm:per:norm_equiv}.
\end{proof}

\begin{lemma}
  \label{th:strong_weak}
  Let $\varphi_N, \psi_N \in \ell^2(\L)$, such that $\varphi_N \to
  \varphi$ strongly in $\ell^2$ and $\psi_N \rightharpoonup 0$ weakly
  in $\ell^2$. Then, $\lim_{N \to \infty} \< \varphi_N, \psi_N \>_{\ell^2} =
  0.$
\end{lemma}
\begin{proof}
  We write
  \begin{displaymath}
    \< \varphi_N, \psi_N \>_{\ell^2} = \< \varphi_N - \varphi, \psi_N \>_{\ell^2} + \<
    \varphi, \psi_N \>_{\ell^2}.
  \end{displaymath}
  The first term on the right-hand side tends to zero due to strong
  convergence of $\varphi_N$, while the second term on the right-hand
  side tends to zero due to weak convergence of $\psi_N$.
\end{proof}

\begin{proof}[Proof of Theorem \ref{th:atm:per:stab}]

  Let $H := \ddel\Ea(u)$ and $H_R :=
  \ddel\E^\per_R(u_R)$. Throughout the proof suppose that $R$
  is sufficiently large so that all statements and operations are
  meaningful.
  
  {\it 1. Upper bound: } Let $v \in \Usz(\L)$, $\|\D v \|_{L^2} =
  1$, then for $R$ sufficiently large, $v$ can also be thought to
  belong to $\Us^\per(\Omega_R)$, then $\< H v, v \> = \< H_R v, v \>$ and
  hence
  \begin{displaymath}
    \overline{\lambda} := \limsup_{R \to \infty} \lambda_R \leq \< H v, v \>.
  \end{displaymath}
  Taking the infimum over all $v$ we obtain that $\overline{\lambda}
  \leq \lambda$.

  {\it 2. Decomposition: } Let $\ul{\lambda} :=
  \liminf_{R \to \infty} \lambda_R = \lim_{j \to \infty}
  \lambda_{R_j}$ for some subsequence $R_j \uparrow \infty$.
  For simplicity of notation, we denote $\Omega_j := \Omega_{R_j}$, $u_j := u_{R_j}$, and $H_j := H_{R_j}$.

  Then let $v_j \in \Us^\per(\Omega_j)$, $\|\D v_j \|_{L^2(\omega_j)} =
  1$, such that
  \begin{displaymath}
    \< H_j v_j, v_j \> \leq \ul{\lambda} + j^{-1}.
  \end{displaymath}

  As in the proof of Lemma \ref{th:atm:per:cons} let $w_j' := T^\per_{R_j} Iv_j$, then $\| \D w_j' \|_{L^2} \leq C \| \D v_j \|_{L^2(\omega_j)} \leq C$, where $C$ is independent of $j$.
  Upon extracting another subsequence (which we still label with $j$), we may assume, without loss of
  generality, that there exists $v \in \UsH(\L)$ such that
  \begin{align*}
    D w_j' \rightharpoonup D v \qquad &\text{weakly in } \ell^2
    \quad \text{ as } j \to \infty,
    \qquad\text{and}
    \\
    \D w_j' \rightharpoonup \D v \qquad &\text{weakly in } L^2
    \quad \text{ as } j \to \infty
    .
  \end{align*}

  According to Lemma \ref{th:decomp_weak_conv} there exists a sequence
  $r_j \uparrow \infty$ such that $r_j \leq
  R_j/2$, and
  \begin{displaymath}
    w_j := T^\per_{r_j} w_j' \qquad \text{satisfies} \qquad
    \D w_j \to \D v \quad \text{ strongly in $L^2$.}
  \end{displaymath}
  Note that thanks to the choice $r_j \leq R_j/2$ we have that $w_j = T^\per_{r_j} v_j$.

  Moreover, noting that $w_j \in \Usz(\L)$ as well as $w_j \in
  \Us^\per(\Omega_j)$, we can define $z_j := v_j - w_j$ and write
  \begin{align}
    \notag
    \< H_j v_j, v_j \> &= \b\< H_j z_j, z_j \b\> + 2 \b\<
    H_j w_j, z_j \b\> + \b\< H_j w_j, w_j \b\>
    \\
    \label{eq:atm:per:stab:decomp}
    &=: a_j + b_j + c_j.
  \end{align}
  
  {\it 3. Estimating $a_j$: } Our first step will be to observe that we have chosen $r_j$ such that, for all $\ell\in B_{r_j/2}$, $Dw_j(\ell) =
    Dv(\ell)$ and hence $Dz_j(\ell) = 0$. We will then exploit the fact that
  $Du(\ell) \to 0$ as $|\ell| \to \infty$. Let $H_j^0 :=
  \ddel\E_{R_j}^\per(0)$, then
  \begin{align*}
    a_j &= \< H_j z_j, z_j \> = \< H_j^0 z_j, z_j\> - \< (H_j-H_j^0)
    z_j, z_j \> \\
    &\geq \< H_j^0 z_j, z_j \> - C \| D u_j \|_{\ell^\infty(\supp(Dz_j)} \| \D z_j
    \|_{L^2}^2 \\
    &\geq \< H_j^0 z_j, z_j \> - C \| \D u_j \|_{\ell^\infty(\Omega_j \setminus B_{r_j/2})} \| \D z_j \|_{L^2}^2 \\
    &=
    \< H_j^0 z_j, z_j \> - o(1) \| \D z_j \|_{L^2}^2
    =
    \< H_j^0 z_j, z_j \> - o(1)
    ,
  \end{align*}
  where $o(1)$ denotes a quantity that converges to zero as $j \to \infty$, and where we used boundedness of $\| \D z_j \|_{L^2}^2$.

  Next, since $B_{r_j/2}\supset B_{\Rcore}$ for $j$ large enough, we have that $Dz_j(\ell) = 0$ for all $\ell \in \L \cap
  B_{\Rcore}$. Therefore, $\<H_j^0 z_j, z_j\>$ is independent of the
  defect core structure, which we can express as
  \begin{displaymath}
    \< H_j^0 z_j, z_j \> = \< H_j^\per z_j, z_j \>, 
  \end{displaymath}
  where $H_j^\per$ is the homogeneous and periodic finite difference
  operator
  \begin{displaymath}
    \< H_j^\per z, z \> = \sum_{\ell \in \omega_j \cap \mA \Z^d}
    \< \ddel V(\bfO) D z(\ell), D z(\ell) \>.
  \end{displaymath}
  
  Define also the lattice homogeneous lattice hessian (finite difference
  operator) $H^{\rm hom}$,
  \begin{displaymath}
    \< H^{\rm hom} z, z \> := \sum_{\ell \in \mA \Z^d}
    \< \ddel V(\bfO) D z(\ell), D z(\ell) \>,
  \end{displaymath}
  and let 
  \begin{displaymath}
    \lambda_j^\per := \inf_{\substack{z \in \Us^\per(\omega_j \cap
        \mA\Z^d) \\
        \|\D z \|_{L^2(\omega_j)} = 1}} \< H_j^\per z, z \> \quad \text{and}
    \quad
    \lambda^{\rm hom} := \inf_{\substack{z \in \UsH(\mA\Z^d) \\
        \|\D z \|_{L^2(\R^d)} = 1}} \< H^{\rm hom} z, z \>.
  \end{displaymath}
  Then it follows from \cite[Theorem 3.6]{Hudson:stab} that
  $\lambda_j^\per \to \lambda^{\rm hom}$ as $j \to \infty$.  Moreover, we have
  from \eqref{eq:pt:stab_lattice} that $\lambda^{\rm hom} \geq \lambda$ (see \S
  \ref{sec:prf_far_field_stab} for the proof).


  Combining the foregoing calculations we obtain that 
  \begin{equation}
    \label{eq:atm:per:stab:est_aj}
    a_j \geq \lambda \| \D z_j \|_{L^2}^2 - o(1).
  \end{equation}

  {\it 4. Estimating $c_j$: }
  Since $Dw_j$ vanish outside $\Omega_j$, we can estimate
  \[
    \|\<H_j w_j - H w_j, w_j\>\|
    =
    \<\ddel\Ea(u_j) w_j - \ddel\Ea(u) w_j, w_j\>
    \leq C \| Du_j - Du \|_{\ell^\infty(\Omega_j)} \|\D w_j\|_{L^2}^2
  \]
  and hence we have
  \begin{align} 
    \notag
    c_j = \< H_j w_j, w_j \>
    \geq~& \< H w_j, w_j \>
    - C \| Du_j - Du \|_{\ell^\infty(\Omega_j)}
    \| \D w_j \|_{L^2}^2
    \\ \label{eq:atm:per:stab:est_cj}
    \geq~& \lambda \| \D w_j \|_{L^2}^2
    - o(1)
    \| \D w_j \|_{L^2}^2
    = \lambda \| \D w_j \|_{L^2}^2
    - o(1)
    ,
  \end{align}

  {\it 5. Estimating $b_j$: } We let $z_j' := w_j' - w_j \in
  \UsH(\L)$ and using the fact $z_j'=z_j$ in $B_{R_j/2} \supset B_{r_j} \supset \supp(D w_j)$, we have, similarly to step 4,
  \begin{displaymath}
    b_j = 2 \< H_j w_j, z_j \> = 2 \< H_j w_j, z_j' \> = 2 \< H w_j, z_j' \>
    - o(1)
    , 
  \end{displaymath}
  According to Lemma
  \ref{th:decomp_weak_conv}, $D z_j' \rightharpoonup 0$ weakly in
  $\ell^2$ as $j \to \infty$. Since $Dw_j$ converges strongly in
  $\ell^2$, it follows that $g_j(\ell) := \ddel V_\ell(Du(\ell))
  Dw_j(\ell)$ also converges strongly in $\ell^2$ and hence Lemma
  \ref{th:strong_weak} implies that
  \begin{equation}
    \label{eq:atm:per:stab:est_bj}
    b_j = 2\< H w_j, z_j' \> - o(1) = 2\< g_j, D z_j' \> - o(1) \to 0 \quad \text{as } 
    j \to \infty.
  \end{equation}

  {\it 6. Completing the proof: } Combining \eqref{eq:atm:per:stab:decomp},
  \eqref{eq:atm:per:stab:est_aj}, \eqref{eq:atm:per:stab:est_cj},
  \eqref{eq:atm:per:stab:est_bj} we obtain
  \begin{align*}
    \< H_j v_j, v_j \>
    &\geq
    \lambda \b( \| \D w_j \|_{L^2}^2 + \| \D z_j \|_{L^2}^2 \b) - o(1)
    \\&=
    \lambda \b( \| \D w_j + \D z_j \|_{L^2}^2 - \<\D w_j, \D z_j\>_{L^2} \b) - o(1)
    =
    \lambda \| \D v_j \|_{L^2}^2 - o(1)
    ,
  \end{align*}
  where in the last line we used that $\<\D w_j, \D z_j\>_{L^2} = o(1)$ which
  follows on adapting Lemma~\ref{th:strong_weak} to the $L^2$ space.
\end{proof}

\begin{proof}[Proof of Theorem \ref{th:pt:per}]
  Repeating the proof of Theorem \ref{th:pt:clamped} almost
  verbatim, but using $T_R^\per$ instead of $T_R$ and employing
  the consistency estimate of Corollary \ref{th:atm:per:cons} and the
  stability result of Theorem \ref{th:atm:per:stab} we obtain, for
  sufficiently large $R$, the existence of a strongly stable solution
  $\ua_R^\per$ to \eqref{eq:pt:per} satisfying
  \begin{equation}
    \label{eq:atm:per:rate_step1}
    \| D T_R^\per \ua - D \ua^\per_R \|_{\ell^2(\Omega)} \leq C R^{-d/2}.
  \end{equation}
  The geometry error estimate (the first bound in \eqref{eq:pt:per:errest}) follows from
  \begin{align*}
    \| D \ua_R^\per - D \ua \|_{\ell^2(\Omega)}  &
    \leq \| D T_R^\per \ua - D \ua \|_{\ell^2(\Omega)} + CR^{-d/2} \\
  & \leq C \| D \ua \|_{\ell^2(\L
    \setminus B_{R/2})} + C R^{-d/2} \leq C R^{-d/2}.
  \end{align*}

  To estimate the energy error, arguing similarly as in the proof of
  Theorem \ref{th:pt:clamped}, and using the fact that
  $\E^\per_R(T^\per_R \ua) = \E(T_R \ua)$, we
  obtain
  \begin{align*}
    \b| \E^\per_R(\ua^\per_R) - \E(\ua) \b| 
    &\leq \b| \E^\per_R(\ua^\per_R) - \E^\per_R(T^\per_R \ua) \b|
    + \b| \E(T_R\ua) - \E(\ua) \b| \\
    &\leq C \B( \b\| D \ua^\per_R - D T^\per_R \ua \b\|_{\ell^2}^2
    + \b\| D T_R \ua - D \ua \b\|_{\ell^2}^2 \B).
  \end{align*}
  Applying the projection error estimate
  \eqref{eq:atm:dir:basic_approx}, the regularity estimate
  \eqref{eq:pt:regularity} with $j = 1$ and the error estimate
  \eqref{eq:atm:per:rate_step1}, we obtain the second bound in \eqref{eq:pt:per:errest}.
\end{proof}

\subsection{Boundary conditions from linear elasticity}\label{sec:approx:lin}

%
%
\begin{proof}[Proof of Theorem \ref{th:pt:lin}]
  {\it 1. Geometry error estimate: } We first use
  $|D\ua(\ell)| \leq C |\ell|^{-d}$ to estimate the consistency error
    \begin{align*}
      \< \del\E^\lin_R(\ua) - \del\E(\ua), v \> &= \sum_{\ell \in \L \setminus
        \Omega_R} \< \delta V_\ell(D\ua(\ell)) - \delta
      V^\lin(D\ua(\ell)), D v(\ell) \> \\
      &\leq C \sum_{\ell \in \L\setminus\Omega_R}
      \b|D\ua(\ell)\b|^2 \, |Dv(\ell)| \\
      &\leq C \b\| D\ua \b\|_{\ell^4(\L \setminus \Omega_R)}^2 \,
      \| D v \|_{\ell^2}
      \leq C R^{d/2-2d}
      \| \D v \|_{L^2}
      = C R^{-3d/2}
      \| \D v \|_{L^2}
      .
    \end{align*}
  
  Moreover, using an analogous linearisation argument it is
  straightforward to establish that
  \begin{displaymath}
    \b\|\ddel\E^\lin_R(\ua) - \ddel\E(\ua) \b\| \leq C \| D\ua
    \|_{\ell^\infty(\L \setminus \Omega_R)} \leq C R^{-d},
  \end{displaymath}
  where $\|\cdot\|$ denotes the $\UsH \to
  (\UsH)^*$ operator norm. This implies that
  \begin{displaymath}
    \< \ddel\E^\lin_R(\ua) v, v \> \geq (c_0 - C R^{-d}) \| \D v
    \|_{L^2}^2 \qquad \forall v \in \UsH.
  \end{displaymath}

  In particular, for $R$ sufficiently large, $\ddel\E^\lin_R$ is
  uniformly stable. The inverse function theorem, Lemma
  \ref{th:app_min_ift}, implies that, for $R$ sufficiently large, there
  exists a strongly stable solution $\ulin_R \in \UsH$ to
  \eqref{eq:pt:lin} satisfying the first bound in \eqref{eq:pt:lin:errest}.

  {\it 2. Energy error estimate: }
  Suppressing
  the argument $(\ell)$, we estimate
  \begin{displaymath}
    \hspace{-1cm} | V(D\ua) - V^\lin(D\ua)
    \b| \\
    \leq C |D\ua|^3
    \leq C|\ell|^{-3d}
  \end{displaymath}
  and therefore
  \begin{align*}
    \b| \E^\lin_R(\ulin_R) - \E(\ua) \b| &\leq \b| \E^\lin(\ulin_R) -
    \E^\lin_R(\ua) \b| + \b| \E^\lin_R(\ua) - \E(\ua) \b| \\
    &\leq C \| D \ulin_R - D \ua \|_{\ell^2}^2
    	+ \sum_{\ell \in \L \setminus \Omega_R} | V(D\ua) - V^\lin(D\ua)
    	    \b| \\ 
    &\leq C R^{-3d} + C \sum_{\ell \in \L \setminus \Omega_R} |\ell|^{-3d}
    ~~\leq~ C R^{-3d} + C R^{-2d}
    .
    \qedhere
  \end{align*}
\end{proof}

We follow the same programme for the proof for dislocations.

\begin{proof}[Proof of Theorem \ref{th:dis:lin}]
  {\it 1. Geometry error estimate: }
  We estimate the consistency error
  \begin{align*}
    \< \del\E^\lin_R(\ua) - \del\E(\ua), v \> &= \sum_{\ell \in \L \setminus
      \Omega_R} \< \delta V(e(\ell) + \Del\ua(\ell)) - \delta
    V^\lin(e(\ell) + \Del\ua(\ell)), D v(\ell) \> \\
    &\leq C \sum_{\ell \in \L\setminus\Omega_R}
    \b|e(\ell) + \Del\ua(\ell)\b|^2 \, |Dv(\ell)| \\
    &\leq C \b\| e(\ell) + \Del\ua(\ell) \b\|_{\ell^4(\L \setminus \Omega_R)}^2 \,
    \| D v \|_{\ell^2}
    \leq C R^{-1} \| \D v \|_{L^2}
    ,
  \end{align*}
  where we used Lemma \ref{th:disl:up_lemma} to estimate $|e(\ell)| \lesssim
  |\ell|^{-1}$ and Theorem \ref{th:disl:regularity} to estimate $|\Del
  \ua(\ell)| \lesssim |\ell|^{-2} \log |\ell|$.

  An analogous linearisation argument yields
  \begin{displaymath}
    \b\|\ddel\E^\lin_R(\ua) - \ddel\E(\ua) \b\| \leq C \| e(\ell) + \Del\ua(\ell)
    \|_{\ell^\infty(\L \setminus \Omega_R)} \leq C R^{-1}.
  \end{displaymath}
  This implies that
  \begin{displaymath}
    \< \ddel\E^\lin_R(\ua) v, v \> \geq (c_0 - C R^{-1}) \| D v
    \|_{\ell^2}^2 \qquad \forall v \in \UsH,
  \end{displaymath}
  and hence Lemma \ref{th:app_min_ift} yields all the statements except for the
  second bound in \eqref{eq:dis:lin:errest}.

  {\it 2. Energy error estimate: } Denoting $g := e(\ell) + \Del\ua$ and again
  suppressing the argument $(\ell)$, we estimate
  \begin{align*}
    & \hspace{-1cm} \b| V(g) - V^\lin(g) -
    V(e) + V^\lin(e)
    \b| \\
    &\leq \B|\smfrac16 \b\< \delta^3 V(\bfO) g,
    g, g \b\> - \smfrac16 \b\< \delta^3 V(\bfO) e, e,
    e \b\> \B| + C
    \b(|g|^4 + |e|^4\b) \\
    &\leq C \B( |g - e| |g|^2 + |g - e|^2 |g| + |g - e|^3
    + |g|^4 + |e|^4 \B) \\
    & = C \B( \b|\Del\ua\b| |g|^2 + \b|\Del\ua\b|^2 |g| + \b|\Del\ua\b|^3
    + |g|^4 + |e|^4 \B) \\
    &\leq C |\ell|^{-4} \log |\ell|
  \end{align*}
  
  Therefore,
  \begin{align*}
    \b| \E^\lin_R(\ulin_R) - \E(\ua) \b| &\leq \b| \E^\lin_R(\ulin_R) -
    \E^\lin_R(\ua) \b| + \b| \E^\lin_R(\ua) - \E(\ua) \b| \\
    &\leq C \| D \ulin_R - D \ua \|_{\ell^2}^2 
    	+ \sum_{\ell \in \L \setminus \Omega_R} \b| V(g) - V^\lin(g) -
        V(e) + V^\lin(e)
        \b| \\ 
    &\leq C R^{-2}
    	+ \sum_{\ell \in \L \setminus \Omega_R} |\ell|^{-4} \log|\ell|
    \leq C R^{-2} + C R^{-2} \log R
    .
    \qedhere
  \end{align*}
\end{proof}

\subsection{Boundary conditions from nonlinear elasticity}\label{sec:approx:ac}

%

\begin{proof}[Proof of Proposition \ref{prop:ACresult}] 
The right-hand side of \eqref{eq:ac:pt:est} can be easily estimated using the assumptions \eqref{eq:ac:restrictions} and the regularity estimate \eqref{eq:pt:regularity}.
Indeed, denote the set $A := B_{c_0 c_3 R^{1+2/d}}\setminus B_R$ and estimate
\[ 
\| h D^2 \ua \|_{\ell^2(\L\cap (\omega_R\setminus B_R))}^2
\leq
\| h D^2 \ua \|_{\ell^2(\L\cap A)}^2
\leq C
\int_{A} \B(\smfrac{|x|}{R}\B)^{2\beta} |x|^{-2d-2} \dx
= C R^{-d-2},
\] 
where the assumption $\beta < \smfrac{d+2}{2}$ was used in the last step.
The second term is bounded as in the earlier sections,
\[
\| D \ua \|_{\ell^2(\L\setminus B_{R_\c/2})}^2
\leq
\| D \ua \|_{\ell^2{\textstyle(}\L\setminus B_{c_2 R^{1+2/d}}{\textstyle)}}^2
\leq C \b(R^{1+2/d}\b)^{-d} = C R^{-d-2}.
\]

It remains to note that there exists $R_0$ such that $C R^{-d-2} \leq \eta$ for
$R\geq R_0$, and hence $\uac_R$ exists and \eqref{eq:pt:rate-ac} follows from
\eqref{eq:ac:pt:est}.
\end{proof}

\begin{proof}[Proof of Proposition \ref{th:dis:ac}]
  We first note that $\Del^2 \up(\ell) = O(|\ell|^{-2})$ as $|\ell|\to\infty$,
  which is an immediate consequence of Lemma \ref{th:disl:up_lemma}.
  Hence in the right-hand side of \eqref{eq:ac:pt:est}, $\up$ dominates $\ua$
  and one can easily bound, similarly to the point defect case,
\begin{align*}
&\b\| h \Del^2 (\up+\ua) \b\|_{\ell^2(\L\cap (\omega_R\setminus B_R)}^2
\leq C
\int_{B_{c_0 c_3 R^{p}}\setminus B_R} \big(\smfrac{|x|}{R}\big)^2 |x|^{-4} \dx
= C R^{-2},
\\
&\b\| \Del \ua \b\|_{\ell^2(\L\setminus B_{R_\c/2})}^2
\leq
\b\| \Del \ua \b\|_{\ell^2{\textstyle(}\L\setminus B_{c_2 R^{p}}{\textstyle)}}^2
\leq C (R^{p} \log(R^{p}))^{-2} <C R^{-2}.
\end{align*}
The existence of $\uac_R$ and \eqref{eq:dis:ac:rate} hence follow by choosing
$R_0$ such that $C R_0^{-2} \leq \eta$.
\end{proof}

\appendix

\section{Continuum Elasticity}
\label{sec:elast_intro}
%

\subsection{Cauchy--Born model}
\label{sec:elast:cb}
Consider a Bravais lattice $\mA \Z^d$ with site potential $V :
(\R^m)^{\Rghom} \to \R \cup \{+\infty\}$. Consider the homogeneous
continuous displacement field $u : \R^d \to \R^m$, $u(x) = \mF x$ for
some $\mF \in \R^{m \times d}$. Then interpreting $u$ as an atomistic
configuration, the {\em energy per unit undeformed volume} in the
deformed configuration $u$ is
\begin{displaymath}
  W(\mF) := V(\mF \cdot \Rghom) / \det\mA.
\end{displaymath}
If $u, \up : \R^d \to \R^m$ are both ``smooth'' (i.e.,
$|\D^2 u(x)|, |\D^2 \up(x)| \ll 1$), then
\begin{displaymath}
  \int_{\R^d} \B( W(\D u) - W(\D\up) \B) \dx
\end{displaymath}
is a good approximation to the atomistic energy-difference $\sum_{\ell
  \in \mA\Z^d} V(Du(\ell)) - V(D\up(\ell))$.

The potential $W : \R^{m \times d} \to \R \cup \{+\infty\}$ is called the
Cauchy--Born strain energy function. Detailed analyses of the Cauchy--Born model
are presented in \cite{BLBL:arma2002, E:2007a, OrtnerTheil2012}. In these
references it is shown that both the Cauchy--Born energy and its first variation
are second-order consistent with atomistic model, and resulting error estimates
are derived.

\subsection{Linearised elasticity}
\label{sec:elast:lin}
A continuum linear elasticity model that is consistent with the
atomistic description can be obtained by expanding the Cauchy--Born
strain energy function $W$ to second order: 
%
\begin{displaymath}
  W(\mG) \sim W(\bm0) + \partial_{\mF_{i\alpha}} W(\bm0)
  \mG_{i\alpha} +
  \smfrac12 \partial_{\mF_{i\alpha}\mF_{j\beta}} W(\bm0) \mG_{i\alpha}\mG_{j\beta},
\end{displaymath}
where we employed summation convention.

Let $\bbC_{i\alpha}^{j\beta} := \partial_{\mF_{i\alpha}\mF_{j\beta}} W(\bm0)$,
then employing cancellation of the linear terms, we obtain the linearised
energy-difference functional
\begin{displaymath}
  \frac{1}{2\det\mA} \int_{\R^d} \sum_{\rho,\vsig\in\Rghom}
  V_{,\rho\vsig}(\bm0) : \D_\rho u \otimes \D_\vsig u \dx
  = \frac12 \int_{\R^d}  \bbC_{i\alpha}^{j\beta} \partial_{x_\alpha}
  u_i \partial_{x_\beta} u_j \dx,
\end{displaymath}
and the associated equilibrium equation is
\begin{displaymath}
  \frac{1}{\det\mA}\sum_{\rho,\vsig\in\Rghom}
  V_{,\rho\vsig}(\bm0) \D_\rho \D_\vsig u =
\bbC_{i\alpha}^{j\beta} \frac{\partial^2 u_i}{\partial
    x_\alpha \partial x_\beta} = 0 \qquad \text{for } i = 1, \dots, m.
\end{displaymath}
(This equation becomes non-trivial when supplied with boundary
conditions or an external potential, either or both arising from the
presence of a defect.)

If the lattice $\mA \Z^d$ is stable in the sense that, for some
$\gamma > 0$,
\begin{displaymath}
  \sum_{\ell \in \mA\Z^d} \b\< \ddel V(\bm0) Dv(\ell), Dv(\ell)
  \b\> \geq \gamma \| \D v \|_{L^2}^2
\end{displaymath}
(cf. \eqref{eq:pt:strongstabeq}, \eqref{eq:reg:lgf:stab_H}) then the
tensor $\bbC$ satisfies the Legendre--Hadamard condition and hence the
linear elasticity equations are well-posed in a suitable function
space setting \cite{Wallace98, Hudson:stab}.

Finally, we remark that, the linear elasticity model can also be
obtained by first deriving a quadratic expansion of the atomistic
energy and then taking the long-wavelength limit (continuum
limit). This yields the relationship between the continuum Green's
function and the lattice Green's function exploited in the proof of
Lemma \ref{th:green_fcn}.

\section{Remarks}

\subsection{Cutoff in reference versus deformed configuration}
\label{sec:app_cutoff_ref}
We briefly show why one may always choose a cut-off in reference
configuration. We focus on the simpler point defect case with
$m = d \in \{2,3\}$, but with minor modifications the argument applies also to
the dislocation case (cf. \S~\ref{sec:dis:atm}).

Suppose that, instead, we choose a cut-off in deformed configuration. The site
energy is now a function of all differences
\begin{displaymath}
  V_\ell(Du(\ell)) = V\b( (D_\rho u(\ell))_{\rho \in
    (\L-\ell)\setminus \{0\}} \b),
\end{displaymath}
but $V_\ell$ effectively only depends on those $D_\rho u(\ell)$ for which
$|\rho+D_\rho u(\ell)| < r_{\rm def}$. Using ideas and notation from
\cite{OrtnerTheil2012} it is possible to generalise the definition of the total
energy $\E(u) = \sum_{\ell\in \L} V_\ell(Du(\ell))$ to this case.

Suppose now that $\ua \in \arg\min \{ \E(u) \sep u \in \UsH\}$. Since $\D \ua$
is piecewise constant and belongs to $L^2$ it follows that $|\D \ua(x)| \to 0$
uniformly as $|x| \to \infty$, and in particular, $\ua$ belongs to the space
\begin{displaymath}
  \Adm := \b\{ v \in \cap \Us^{1,2} \bsep \| \D v \|_{L^\infty} < m_{\Adm} \text{ and
    } |\D v(x)| < 1/2 \text{ for
  } |x| > r_{\Adm} \b\},
\end{displaymath}
provided that $m_{\Adm}, r_{\Adm}$ are chosen sufficiently
large. Moreover, possibly upon enlarging $m_{\Adm}, r_{\Adm}$, all
displacements $u \in \UsH$ with $\| \D u - \D \ua \|_{L^2} \leq 1/2$
belong to $\Adm$ as well. Since our approximation error analysis only
employs local arguments, it is therefore sufficient to define
$V_\ell(Du(\ell))$ for $u \in \Adm$ only.

We now show that a finite interaction range in deformed configuration
gives rise to a finite interaction range in reference configuration,
for displacements from $\Adm$. Let $u \in \Adm$ and $y = x = u$, then
we can estimate
\begin{displaymath}
  |y(\ell+\rho) - y(\ell)| \geq |\rho| - \int_0^1 \b| \D_\rho
  u(\ell+t\rho) \b| \dt
  \geq |\rho| \bg( 1 - \int_0^1 |\D u(\ell+t\rho)| \dt \bg).
\end{displaymath}
Since the bound $|\D u| < 1 / 2$ is violated at most on a segment of
length $2 r_{\Adm}$ it follows that
\begin{displaymath}
  |y(\ell+\rho) - y(\ell)| \geq |\rho| \bg(1 - m_{\Adm} \frac{2
    r_{\Adm}}{|\rho|} - \frac{1}{2} \frac{|\rho| - 2 r_{\Adm}}{|\rho|} \bg) \geq \frac{|\rho|}{4},
\end{displaymath}
for all sufficiently large $|\rho|$, and in particular, $|y(\ell+\rho)
- y(\ell)| \geq r_{\rm def}$ for all sufficiently large $|\rho|$; say,
$|\rho| > r_{\rm ref}$.

Thus, we conclude that, for $u \in \Adm$, $V_\ell(Du(\ell))$ depends
effectively only on $(D_\rho u(\ell))_{|\rho| < r_{\rm ref}}$.

\subsection{Far-field stability}
\label{sec:prf_far_field_stab}
In this appendix, we prove the claim made in
\S~\ref{sec:pt:regularity} that {\em strong stability of an
  equilibrium \eqref{eq:pt:strongstabeq} implies strong stability of
  the homogeneous lattice \eqref{eq:pt:stab_lattice}.} More generally
we establish that, if \eqref{eq:pt:strongstabeq} holds for {\em any}
$u \in \UsH$, then \eqref{eq:pt:stab_lattice} holds as well.

We pick a test function on the homogeneous lattice $v \in
\Usz(\mA\Z^d)$ with support contained in $B_{s}$. Next, we take a
sequence $\ell_n \in \L, |\ell_n| \to \infty$ and define shifted test
functions on $\L$, $v^{(n)} \in \Usz(\L)$, via $v^{(n)}(\ell) :=
v(\ell - \ell_n)$ for $\ell \in \L \cap B_{r_1}(\ell_n)$ and
$v^{(n)}(\ell) = 0$ otherwise, which is well-defined provided that
$|\ell_n|$ is sufficiently large (without loss of generality).

Since $\D u \in L^2$, $\|Du(\ell)\|_{\ell^\infty(B_s(\ell_n))} \to 0$
as $n \to \infty$, which readily implies that, for all $\eta \in
\mA\Z^d \cap B_s(0)$, $\ddel V_{\ell_n+\eta}(Du(\ell_n+\eta)) \to
\ddel V(\bfO)$ as $n \to \infty$. Thus,
\begin{align*}
  \< \ddel \E(u) v^{(n)}, v^{(n)} \> &= \sum_{\ell \in \L \cap
    B_{s}(\ell_n)} \< \ddel V_{\ell}(Du(\ell)) D v^{(n)}(\ell), D
  v^{(n)}(\ell) \b\>
  \\
  &= \sum_{\eta \in \mA \Z^d \cap B_{s}(0)} \b\< \ddel
  V(Du(\ell_n+\eta)) D v(\eta),
  D v(\eta) \b\> \\
  &\overset{n \to \infty}{\longrightarrow} \sum_{\eta \in \mA \Z^d
    \cap B_{s}} \b\< \ddel V(\bfO) D v(\eta), D v(\eta) \b\> = \<
  \ddel \E_\mA(0) v, v \>.
\end{align*}
Hence, the result follows.

\section{List of Symbols}
\label{sec:symbols}
\begin{itemize}
\item $\mA \Z^d$: homogeneous reference lattice; $\L$: defective reference
  lattice (point defects) or $\L = \mA \Z^d$ (dislocations);
  p. \pageref{sym:lattice}
\item $\ell, k$: lattice sites; $\rho,\vsig,\tau$: lattice directions
\item $\Rcore$: defect core radius (point defects); p. \pageref{sym:Rcore}
\item $\T_\L$: auxiliary triangulation of reference domain $\L$;
  p. \pageref{sym:T-L}
\item $\Usz, \UsH$: discrete function spaces; p. \pageref{eq:pt:defnUszUsH}
\item $\Rg_\ell, \Rg$: interaction ranges; $\rcut$: interaction radius;
  p. \pageref{sym:Rg-rcut}
\item $V, V_\ell$: site energy potential; p. \pageref{sym:V}
\item $\E$: energy-difference functional; p. \pageref{sym:E} for point defects
  and p. \pageref{eq:disl:Ediff} for dislocations
\item $W$: Cauchy--Born strain energy potential; p. \pageref{sec:elast:cb}
\item $\burg$: Burgers vector; $\burg_{12} = (\burg_1, \burg_2)$: in-plane
  component; p. \pageref{sym:burg}
\item $\hat{x}, \rdisl$: position and radius of dislocation core; $\Gamma$:
  branch-cut, or slip-half-plane; p. \pageref{sym:xhat-Gamma}; $\Omega_\Gamma$:
  right half-space; p. \pageref{eq:disl:defn_halfspace}
\item $S_0$: slip operators for total displacements; $S$: slip operator for
  relative displacements; $R$: dual slip operator; p. \pageref{eq:disl:defn_Sop}
  and \pageref{sym:S-R}
\item $\ulin$: linear elasticity solution for a dislocation; $\up$: predictor
  displacement for a dislocation; p. \pageref{eq:disl:linel_pde} and
  p. \pageref{eq:disl:defn_u0}
\item $\Adm$: admissible set for dislocation problem;
  p. \pageref{eq:disl:defn_Adm}
\item $\Gamma_S, \E_S$: ``reflected'' dislocation geometry and energy difference
  functional; p. \ref{rem:reflection}.
\item $e$: elastic strain of predictor $\up$; $\Del$: elastic gradient operator;
  p. \pageref{eq:defn_elastic_strain}.
\end{itemize}

\bibliographystyle{plain}
\bibliography{qc}
\end{document}

%% file: notation.tex
\renewcommand{\cases}[1]{\left\{ \begin{array}{rl} #1 \end{array} \right.}
\newcommand{\smfrac}[2]{{\textstyle \frac{#1}{#2}}}

\def\Xint#1{\mathchoice
{\XXint\displaystyle\textstyle{#1}}%
{\XXint\textstyle\scriptstyle{#1}}%
{\XXint\scriptstyle\scriptscriptstyle{#1}}%
{\XXint\scriptscriptstyle\scriptscriptstyle{#1}}%
\!\int}
\def\XXint#1#2#3{{\setbox0=\hbox{$#1{#2#3}{\int}$ }
\vcenter{\hbox{$#2#3$ }}\kern-.6\wd0}}
\def\mint{\Xint-}

\renewcommand\b{\big}
\newcommand\B{\Big}
\newcommand\bg{\bigg}
\newcommand\Bg{\Bigg}

\newcommand\sep{\,|\,}
\newcommand\bsep{\,\b|\,}

\newcommand\diam{{\rm diam}}

\newcommand\supp{{\rm supp}}

\newcommand\R{\mathbb{R}}
\newcommand\N{\mathbb{N}}
\newcommand\Z{\mathbb{Z}}
\newcommand\C{\mathbb{C}}

\newcommand\dx{\,{\rm d}x}

\newcommand\dt{\,{\rm d}t}
\newcommand\ds{\,{\rm d}s}

\newcommand\dk{\, {\rm d}k}

\newcommand{\<}{\langle}
\renewcommand{\>}{\rangle}

\newcommand\ul{\underline}

\newcommand\mA{{\sf A}}
\newcommand\mB{{\sf B}}

\newcommand\mF{{\sf F}}
\newcommand\mG{{\sf G}}

\newcommand\mI{{\sf Id}}

\newcommand\bfg{{\bm g}}
\newcommand\bfrho{{\bm \rho}}
\newcommand\bftau{{\bm \tau}}

\newcommand\bfh{{\bm h}}
\newcommand\bfO{{\bm 0}}

\newcommand\bbC{\mathbb{C}}

\newcommand\D{\nabla}
\newcommand\del{\delta}
\newcommand\ddel{\delta^2}

\renewcommand\a{{\rm a}}
\renewcommand\c{{\rm c}}
\newcommand\ac{{\rm ac}}
\renewcommand\i{{\rm i}}

\renewcommand\L{\Lambda}

\newcommand\Us{\dot{\mathscr{W}}}
\newcommand\Usz{\Us^{\rm c}}
\newcommand\UsH{\Us^{1,2}}
\newcommand\UsHd{\Us^{-1,2}}

\newcommand\E{\mathscr{E}}
\newcommand\Equad{\mathscr{F}}

\newcommand\Ea{\E^\a}

\newcommand\Rg{\mathcal{R}}

\newcommand\Nhd{\mathcal{N}}

\newcommand\vsig{\varsigma}
\newcommand\rcut{r_{\rm cut}}

\newcommand\ua{\bar{u}}
\newcommand\per{{\rm per}}
\newcommand\lin{{\rm lin}}
\newcommand\ulin{u^\lin}
\newcommand\uac{\bar{u}^\ac}

\newcommand\T{\mathcal{T}}

\newcommand\Br{\mathcal{B}}
\newcommand\Gr{\mathcal{G}}
\newcommand\GrC{G}
\newcommand\Ftd{\mathcal{F}_{\rm d}}
\newcommand\Ft{\mathcal{F}}
